\documentclass{article}

 \usepackage{amsmath,amsfonts,mathrsfs,amsthm,amssymb}
 \usepackage[all]{xy}
 \usepackage{multirow}
 \usepackage{pstricks}
 \usepackage{pst-vue3d}
 \usepackage{graphicx}

 \newtheorem{thm}{Theorem}[section]
 \newtheorem{cor}[thm]{Corollary}
 \newtheorem{lem}[thm]{Lemma}
 \newtheorem{prop}[thm]{Proposition}
 \newtheorem{defn}[thm]{Definition}
 \newtheorem{rem}[thm]{Remark}
 
 \newtheorem{example}[thm]{Example}
 \numberwithin{equation}{section}

 \newcommand\ambient[2]{\innprod{#1}{#2}}
 \newcommand\p{\gothic{p}}
 \newcommand\metric{\mu}
 \newcommand\uvprod{\nu}
 \newcommand\Lg{LG(2,4)}
 \newcommand\x{{\bf x}}
 \newcommand\normal{{\bf N}}
 \newcommand\soMC[2]{\omega^{#1}_{#2}}
 \newcommand\F{{\mathcal F}}  
 \newcommand\mf[1]{\textsf{v}_{#1}}
 \newcommand\mfp[1]{\textsf{v}_{#1}{}'}
 \newcommand\mft[1]{\tilde{\textsf{v}}_{#1}}
 \newcommand\jet{\xi} 
 \newcommand\pr{{\rm pr}} 
 \newcommand\Hess{\textsf{Hess}} 
 \newcommand\Q{\mathcal{Q}} 
 \newcommand\sgn{{\rm sgn}} 
 \newcommand\trfr{\textsf{trfr}} 
 \newcommand\Sid{S^{1,2}} 
 \renewcommand\S[1]{\mathcal{S}_{[#1]}} %
 \newcommand\sech{{\rm sech}} 
 \newcommand\stack[2]{\begin{tabular}{#1} #2 \end{tabular}}
 \newcommand\R{\mathbb{R}}
 \renewcommand\P{\mathbb{P}}
 \newcommand\ra{\rightarrow}
 \renewcommand\l{\left}
 \renewcommand\r{\right}
 \newcommand\g{\mathfrak{g}}
 \newcommand\parder[1]{\frac{\partial}{\partial #1}}
 \newcommand\mat[2]{\left( \begin{array}{#1} #2 \end{array} \right)}
 \newcommand\transpose[1]{{#1}^{\ensuremath{\mathsf{T}}}}
 \renewcommand\p{\mathfrak{p}}
 
 \renewcommand\sp{\mathfrak{sp}}
 \newcommand\Z{\mathbb{Z}}
 \newcommand\rank{{\rm rank}}
 \newcommand\qbox[1]{\quad\mbox{#1}\quad}
 
 \newcommand\qRa{\quad\Rightarrow\quad}
 \newcommand\w{\textstyle{\bigwedge}}
 \newcommand\innprod[2]{\langle #1, #2 \rangle}
 \newcommand\so{\mathfrak{so}}
 \newcommand\diag{{\rm diag}}
 \newcommand\trans[1]{\,\pitchfork\!\!(#1)}
 \def\Lag/{Lagrangian--Grassmannian}
 \def\MA/{Monge--Amp\`ere}
 \def\MC/{Maurer--Cartan}
 \def\cs/{c.s.} 
 \newcommand\inj{\hookrightarrow}
 \newcommand\grad{{\rm grad}}
 \newcommand\intprod{\mathbin{\hbox{\vrule height1.4pt width4pt depth-1pt\vrule height4pt width0.4pt depth-1pt}}}
  
 \renewcommand\arraystretch{1.1}
 
 \title{Conformal geometry of surfaces in the \Lag/\\ and second order PDE}
 \author{Dennis The}
 
\begin{document}
 \maketitle

 {\abstract Of all real Lagrangian--Grassmannians $LG(n,2n)$, only $\Lg$ admits a distinguished (Lorentzian) conformal structure and hence is identified with the indefinite M\"obius space $\Sid$.  Using Cartan's method of moving frames, we study hyperbolic (timelike) surfaces in $\Lg$ modulo the conformal symplectic group $CSp(4,\R)$.  This $CSp(4,\R)$-invariant classification is also a contact-invariant classification of (in general, highly non-linear) second order scalar hyperbolic PDE in the plane.  Via $\Lg$, we give a simple geometric argument for the invariance of the general hyperbolic \MA/ equation and the relative invariants which characterize it.  For hyperbolic PDE of non-\MA/ type, we demonstrate the existence of a geometrically associated ``conjugate'' PDE.  Finally, we give the first known example of a Dupin cyclide in a Lorentzian space.
} 
  
 \tableofcontents

 \section{Introduction}
  
 In this article we have two goals in mind: (i) investigate the local differential geometry of hyperbolic (timelike) surfaces in the real \Lag/ $\Lg$ modulo the conformal symplectic group $CSp(4,\R)$, and (ii) investigate the local contact geometry of (in general, highly non-linear) scalar hyperbolic PDE in the plane.  Let us describe each in turn and clarify their connection to each other.
  
 Recall that $\Lg$ is the set of isotropic 2-planes with respect to a given symplectic form $\eta$ on $\R^4$.  As a manifold, $\Lg$ is 3-dimensional and admits a transitive action by the symplectic group $Sp(4,\R)$.  For our purposes, $\eta$ will be only defined up to scale, hence we use instead $CSp(4,\R)$.   Our basic question then is: \emph{Given two (embedded) surfaces $M,\tilde{M}$ of $\Lg$, does there exist an element $g \in CSp(4,\R)$ such that $\tilde{M} = g\cdot M$? } We will address this question in the small and seek differential invariants which determine if a neighbourhood of a point in $M$ is equivalent to a neighbourhood of a point in $\tilde{M}$.  Most of our study will assume $M,\tilde{M}$ come equipped with parametrizations $i : U \ra M$, $\tilde{i} : U \ra \tilde{M}$ so that equivalence means $\tilde{i} = g \cdot i$ on $U$.  Ultimately however, we will be interested in the {\em unparametrized} equivalence problem, whereby $i : U \ra M$ and $\tilde{i} : \tilde{U} \ra \tilde{M}$ are equivalent if there is $g \in CSp(4,\R)$ and a diffeomorphism $\varphi : U \ra \tilde{U}$ such that $\tilde{i} \circ \varphi = g \cdot i$ on $U$.   While the general study of submanifolds in homogeneous spaces is classical \cite{Jensen1977}, our study of surfaces in $\Lg$ modulo $CSp(4,\R)$ has not appeared in the literature.

 A feature of $\Lg$ absent in higher dimensional \Lag/s is that $\Lg$ is endowed with a canonical (up to sign) $CSp(4,\R)$-invariant Lorentzian conformal structure $[\metric]$, or equivalently a unique $CSp(4,\R)$-invariant cone field $\mathcal{C}$.   This is a manifestation of the well-known isomorphism $Sp(4,\R) \cong Spin(2,3)$, where $Spin(2,3)$ is a double-cover of $SO^+(2,3)$, i.e.\ the identity connected component of $SO(2,3)$.   At a deeper level, this comes from the graph isomorphism of the (complex) $\textsf{B}_2$ and $\textsf{C}_2$ Dynkin diagrams, or corresponding (real) Satake diagrams.  With respect to $[\metric]$, $\Lg$ is conformally flat and is diffeomorphic to the indefinite M\"obius space $\Sid$.
 
  The study of the M\"obius space (conformal sphere) $S^n$  of definite signature has a long history.  We highlight only the work of Akivis \& Goldberg \cite{AG1996} which contains an extensive bibliography of the literature on conformal geometry and which was a source of inspiration for our work here.    Using moving frames, Akivis \& Goldberg carry out a unified study of submanifolds in conformal spaces.  However, their study of the indefinite signature case (Section 3.3 in \cite{AG1996}) is very brief -- e.g. the timelike case occupies only half of p.104 in \cite{AG1996}.  Four-dimensional conformal structures in all signatures are studied, but it seems that the three-dimensional indefinite case $\Sid \cong \Lg$ was not substantially addressed in their work and has not appeared anywhere in the literature.  Thus, one of our goals is to fill in this gap.  It should be noted that hypersurface theory in $S^n$ for $n \geq 4$ is significantly different than for $S^3$.  For $n \geq 4$, conformal rigidity of a generic hypersurface is determined by the first fundamental form and trace-free second fundamental form (c.f. Theorem 2.3.1 in \cite{AG1996}); for $n=3$, third order invariants come into play \cite{SS1980}.  In the indefinite case, a similar phenomenon occurs.  We also remark that $\Lg$ can be identified with the Lie quadric $Q^3$ in Lie sphere geometry \cite{Cecil2008}, whose elements correspond to oriented spheres and points in $\R^2$.  However, here as well, no study of surfaces in $Q^3$ has appeared in the literature.

 In spirit, our study is similar to the classical theory of surfaces in Euclidean space $\R^3$ modulo the Euclidean group $E(3)$ consisting of rotations, translations and reflections.  The mean and Gaussian curvatures feature prominently in this theory.  We derive here analogous local invariants for hyperbolic surfaces in $\Lg$ modulo  $CSp(4,\R)$.  Unlike the Euclidean case, any hyperbolic surface is locally conformally flat (see Lemma \ref{lem:conf-flat}), so has no intrinsic geometry.  Thus, our question is an extrinsic one solely concerned with their embedding into the ambient space $\Lg$.
 
 Our motivation for this study comes from the contact geometry of a scalar second order PDE in the plane (hereafter, simply referred to as a PDE).  A PDE $F=0$ can be realized geometrically as a (7-dimensional) hypersurface in the second jet-space $J^2 = J^2(\R^2,\R)$.  This ambient jet space is equipped with a canonical contact system $C$ and the geometric theory of such PDE is concerned with the study of such hypersurfaces modulo contact transformations, i.e.\ those diffeomorphisms preserving $C$.  One has the well-known contact-invariant classification into equations of elliptic, parabolic, and hyperbolic type \cite{Gardner1967}.  Less known is the more refined subclassification of hyperbolic PDE into those of \MA/ (MA), Goursat, and generic types (also called class 6-6, 6-7, and 7-7), based on properties of the so-called Monge subsystems \cite{GK1993}.  All MA PDE are of the form $a(z_{xx} z_{yy} - (z_{xy})^2) + bz_{xx} + cz_{xy} + dz_{yy} + e = 0$ (the coefficients are functions of $x,y,z,z_x,z_y$) and have been well-studied, while relatively little progress has been made on the latter classes, which consist of non-linear equations.  For a recent study of generic hyperbolic PDE, see \cite{The2008}, \cite{The2010}.  We also highlight the fact that (nonlinear) hyperbolic PDE of the type mentioned above arise as hydrodynamic reductions of certain integrable PDE in three independent variables \cite{Smith2009}.  We give an equivalent definition of these three hyperbolic classes in terms of the 2nd order $CSp(4,\R)$-invariant classification of hyperbolic surfaces in $\Lg$, c.f. Table \ref{2-classification}.  Roughly, this is possible because:
 \begin{enumerate}
 \item Every fibre $J^2|_\jet$, for $\jet \in J^1$, of the projection map $\pi^2_1: J^2 \ra J^1$ is diffeomorphic to $\Lg$.
 \item The intersection of a hypersurface $F=0$ with each fibre $J^2|_\jet$ is a two-dimensional surface; the hyperbolicity condition is a first-order condition on this surface.
 \item Regarding $C$ as a rank 4 distribution on $J^1$ endowed with a (conformal) symplectic form $\eta$, any contact transformation of $J^2$ fixing $\jet \in J^1$ acts on the fibre $J^2|_\jet = LG(C_\jet,[\eta])$ by an element of $CSp(C_\jet,[\eta])$.
 \end{enumerate}
 Ours is a fibrewise study of $\pi^2_1 : J^2 \ra J^1$ and $CSp(4,\R)$-invariants of surfaces in $\Lg$ yield contact invariants for a PDE.  {\em This fibrewise study adds nothing new for MA equations, but new contact invariants are obtained for Goursat and generic equations.}  Our work here is principally a contribution to the study of fully non-linear hyperbolic PDE.
 
 Added impetus for this work comes from the recent study of curves in general Lagrangian--Grassmannians \cite{Zelenko2005}, \cite{ZL2007}, integrable PDE in three independent variables (hypersurfaces in $LG(3,6)$) \cite{FHK2007}, integrable $GL(2,\R)$ geometry \cite{Smith2009}, symplectic MA equations in four independent variables (hypersurfaces in $LG(4,8)$) \cite{DF2010}, and higher-dimensional MA equations \cite{AAMP2010}.  While the $LG$ realization of the fibres of $\pi^2_1 : J^2 \ra J^1$ dates back at least to work of Yamaguchi \cite{Yamaguchi1982}, we feel this viewpoint is not well-known and has not been sufficiently explored.
  
 Let us give an overview of the contents of our paper.   In Section \ref{sec:PDE}, we review the construction of the second jet space $J^2$ via the Lagrange--Grassmann bundle.  We prove in Theorem \ref{thm:symplectic-contact} that $CSp(4,\R)$-invariants yield contact invariants for PDE.  This is easily seen to hold for $J^2(\R^n,\R)$, so more generally: {\em submanifold theory in $LG(n,2n)$ modulo $CSp(2n,\R)$ has implications for the contact-invariant study of (systems of) scalar PDE in $n$ independent variables}.  In Section \ref{sec:prelim}, we delve into the geometry of $\Lg$, taking full advantage of the aforementioned special isomorphism $Sp(4,\R) \cong Spin(2,3)$.  We describe $\Lg$ as a quadric hypersurface $\Q \subset \R\P^4$, and its canonical $CSp(4,\R)$-invariant conformal structure and cone field $\mathcal{C}$.  To any $[z] \in \P V$ there corresponds a basic surface $\S{z} = \P(z^\perp) \cap \Q$ which we call a {\em ``sphere''} of indefinite, definite, or degenerate type (locally, a hyperboloid of one sheet, two sheets, or a cone respectively).  
 
 In Section \ref{sec:hyp-mf}, we use Cartan's method of moving frames to extract invariants of hyperbolic surfaces $M \subset \Lg$.  At second order, we recall the conformal Gauss map via the central sphere congruence.  Three cases arise: surfaces {\em doubly-ruled} or {\em singly-ruled} by null geodesics, or {\em generic} surfaces.  A doubly-ruled hyperbolic surface is shown to be (an open subset of) an indefinite sphere.  In local coordinates, spheres take the same form as a MA equation.  We recover the well-known theorem on contact-invariance of the hyperbolic MA PDE class via the simple argument:
 \begin{enumerate}
 \item A \MA/ PDE intersects any fibre of $\pi^2_1 : J^2 \ra J^1$ as an indefinite sphere.
 \item A contact transformation of $J^2$ maps indefinite spheres in any fibre to indefinite spheres in any other fibre.
 \end{enumerate}
 If $M$ is given a null parametrization, explicit parametrizations of moving frames are given in Appendix \ref{app:param}, leading in particular to two relative invariants $I_1,I_2$ which we call {\em \MA/ invariants} because of their connection to corresponding contact invariants for hyperbolic PDE.  In the PDE setting, the MA invariants  were first calculated by Vranceanu \cite{Vranceanu1940} for equations of the form $z_{xx} = f(x,y,z,z_x,z_y,z_{xy},z_{yy})$.  Much later, Jur\'a\v{s} \cite{Juras1997} calculated these invariants for general $F(x,y,z,z_x,z_y,z_{xx},z_{xy},z_{yy})=0$.  A third calculation appeared in \cite{The2008} which simplified the expression of these invariants.  However, all three calculations were significantly involved and did not appeal to the inherent geometry of surfaces in $\Lg$.  Our computation here of $I_1,I_2$ based on 2-adapted moving frames is conceptually simple and geometrically motivated: their vanishing characterizes indefinite spheres.
  
 In Sections \ref{sec:singly} \& \ref{sec:generic}, we study the geometry of singly-ruled and generic surfaces.  For such surfaces $M$, pairs of third order objects called {\em cone congruences} can be geometrically associated to $M$.  The construction of 3-adapted moving frames leads to the key notion of the {\em conjugate manifold} $M'$, whose dimension $\dim(M')$ is an invariant.  If $M'$ is also a surface, both $M$, $M'$ are envelopes for the central sphere congruence of $M$.  For a PDE, a corresponding fibrewise construction leads to the notion of its {\em conjugate PDE}.  No such notion exists for MA equations since fibrewise these are second order objects (spheres), while the conjugate manifold is a third order construction.  In the generic (2-elliptic) case, curvature lines exist which lead to the Lorentzian analogues of contact spheres, canal surfaces, and Dupin cyclides.
 As a preview, our classification of hyperbolic surfaces / PDE is given in Figure \ref{fig:full-classification}.  Examples which we will discuss in this paper are given in Figure \ref{fig:examples}.  (See Definition \ref{defn:CSI} for the notion of a CSI PDE.)

   \begin{figure}[h] 
  \begin{center} \framebox{
 \xymatrix@1@C=10pt@W=20pt{& *++[F-,]{\stack{c}{$M$ hyperbolic}} \ar[ld]_{\mbox{$I_1=I_2=0$\quad}} \ar[d]_{\mbox{$I_1 = 0$ or}}^{\mbox{$I_2=0$}} \ar[rd] ^{\mbox{$I_1I_2\neq 0$} }&*++[F-,]{\dim(M')=0}&*++[F-,]{\stack{c}{$\dim(M')=1$\\$ M' \mbox{ null curve}$}} \\
 *++[F-,]{\stack{c}{Doubly-ruled by\\ null geodesics: \\ 2-isotropic}}  &
 *++[F-,]{\stack{c}{Singly-ruled by\\ null geodesics:\\ 2-parabolic}} \ar[ld] \ar[d]  \ar[rd] & 
 *++[F-,]{\stack{c}{Generic surfaces: \\ \begin{tabular}{r@{ }l} 2-hyp: & $\epsilon = -1$\\ 2-ell: & $\epsilon = 1$\end{tabular}}} \ar[r] \ar[u] \ar[ru]& *++[F-,]{\stack{c}{$\dim(M')=2$\\ $M'$ hyperbolic}} \ar[d] \ar[rd]\ar[r] &*++[F-,]{M' \mbox{ 2-generic}}\\
 *++[F-,]{\stack{c}{$\dim(M') = 0$\\ ($\delta_1 = \delta_2 = 0$)}}  &  *++[F-,]{\stack{c}{$\dim(M')=1$\\ $M'$ null geodesic\\ ($\delta_1 =\pm 1$, $\delta_2 = 0$) }} &  *+[F-,]{\stack{c}{$\dim(M')=2$\\ $M'$ hyp.\ 2-par.\\ $(\delta_2 = \pm 1)$}} & *++[F-,]{M' \mbox{ 2-isotropic}}  &  *++[F-,]{M' \mbox{ 2-parabolic}} 
 }}
 \caption{Hyperbolic surfaces up to $CSp(4,\R)$-equivalence; also, hyperbolic PDE up to contact-equivalence}
 \label{fig:full-classification}
 \end{center} \end{figure}
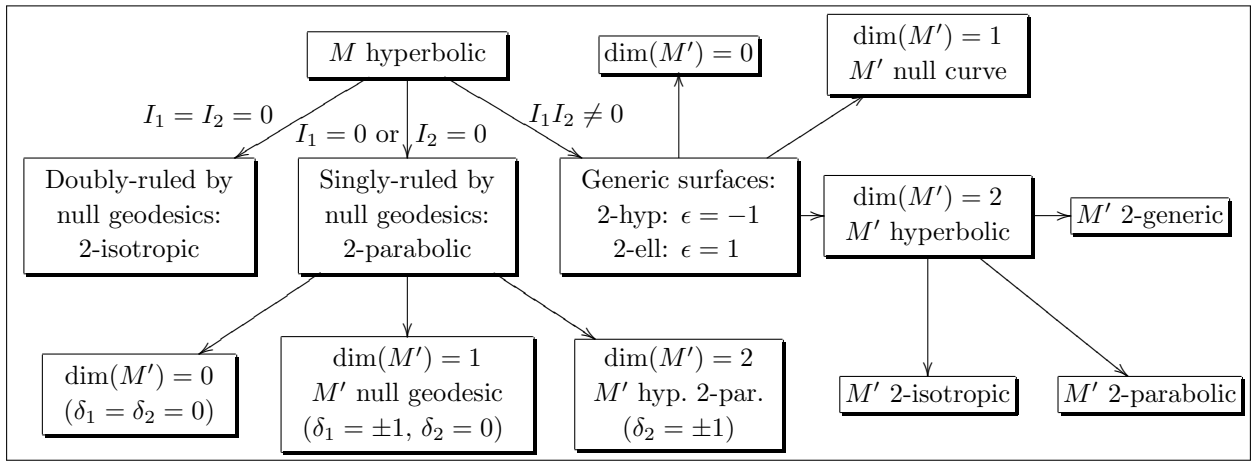

 \begin{center}
 \begin{figure}[h] 
 \[
 \begin{array}{|c|c|c|} \hline
 \mbox{Hyperbolic PDE} & \mbox{Fibrewise classification as a surface $M \subset \Lg$} & \mbox{CSI?}\\ \hline\hline
 \begin{array}{c} a(z_{xx} z_{yy} - (z_{xy})^2) + b z_{xx} + c z_{xy} + d z_{yy} + e = 0 \\ (a,b,c,d,e \mbox{ functions of } x,y,z,z_x,z_y) \end{array} & \mbox{indefinite sphere} & \checkmark\\ \hline
 F(z_{xy},z_{yy}) = 0 & \mbox{indefinite sphere or singly-ruled},\, \dim(M') \leq 1 & \times \\ \hline
 z_{xy} = \frac{1}{2} (z_{yy})^2 & \mbox{singly-ruled},\, \dim(M') = 0 & \checkmark\\ \hline
 z_{xy} = \exp(z_{yy}) & \mbox{singly-ruled},\, \dim(M') = 1, \, \delta_1 = +1 & \checkmark\\ \hline
 z_{xy} = \ln(z_{yy}) & \mbox{singly-ruled},\, \dim(M') = 1, \, \delta_1 = -1 & \checkmark\\ \hline
 F(z_{xx},z_{yy}) = 0 & \mbox{indefinite sphere or 2-elliptic} & \times\\ \hline
 z_{xx} = \exp(z_{yy}) & \mbox{2-elliptic},\, \dim(M') = 2 & \checkmark\\ \hline
 z_{xx} = \frac{1}{3} (z_{yy})^3 & \mbox{2-elliptic},\, \dim(M') = 2 & \checkmark\\ \hline
 z_{xx} z_{yy} = -1 & \mbox{2-elliptic, Dupin cyclide},\, \dim(M') = 2 & \checkmark\\ \hline
 3z_{xx} (z_{yy})^3 + 1 = 0 & \mbox{2-elliptic},\, \dim(M') = 0 & \checkmark\\ \hline
 \displaystyle\frac{(3z_{xx} - 6 z_{xy} z_{yy} + 2(z_{yy})^3)^2}{(2z_{xy} - (z_{yy})^2)^3} = c & \l\{\begin{array}{l}\mbox{2-elliptic if }c > 4,\\ \mbox{2-hyperbolic if } c < 0,\\ \mbox{2-hyperbolic if } c =0 \\
 \quad \mbox{ and } z_{xy} > \frac{1}{2} (z_{yy})^2\end{array};\,\r. \dim(M') = 0 & \checkmark\\ 
 & \mbox{(For $0 < c \leq 4$, the PDE is not hyperbolic.)} &  \\ \hline
 \end{array} 
 \]
 \caption{Examples of hyperbolic PDE}
 \label{fig:examples}
 \end{figure}
  \end{center}
 The last two entries in Table \ref{fig:examples} form the complete list (up to contact equivalence) of maximally symmetric hyperbolic PDE of generic type \cite{The2008}, each having 9-dimensional contact symmetry algebra.  In the case of the latter family, this 9-dimensional Lie algebra is independent of the parameter $c$ and has been shown to be isomorphic to a parabolic subalgebra of the non-compact real form of the 14-dimensional exceptional simple Lie algebra $\g_2$ \cite{The2010}.  We should also note that while all examples in Figure \ref{fig:examples} are of the form $F(z_{xx},z_{xy},z_{yy}) = 0$, our classification applies equally well to those PDE  which have $x,y,z,z_x,z_y$ dependency as well: our classification is fibrewise.

 In Section \ref{sec:conc}, we make some concluding remarks and comment on future directions for research.\\

 \section{The geometry of second order PDE in the plane}
 \label{sec:PDE}

 A scalar second order PDE in the plane can be realized as a hypersurface in $J^2 = J^2(\R^2,\R)$.   We review the geometric construction of the jet spaces $J^k = J^k(\R^2,\R)$ for $k \leq 2$.  This section generalizes to more independent variables than two, but we restrict to two for simplicity and since this is our main focus in this paper.
 
 \subsection{Jet spaces and the Lagrange--Grassmann bundle}
 \label{sec:jet-space}
 
 Identify the zeroth order jet space $J^0$ with $\R^3$, regarded as a trivial bundle with fibre $\R$ over the base $\R^2$.  Define the first order jet space $J^1$ to be the Grassmann bundle $Gr(2,T\R^3)$, with projection $\pi^1_0 : J^1 \ra J^0$.  Any  $\jet \in J^1$ is a 2-plane in $T_{\bar\jet} \R^3$, where $\bar\jet = \pi^1_0(\jet)$. There is a canonical linear Pfaffian system $(\mathcal{I},\mathcal{J})$ on $J^1$ given by
 \[
 \mathcal{I}|_\jet = (\pi^1_0)^*(\jet^\perp), \qquad \mathcal{J}|_\jet = (\pi^1_0)^*(T^*_{\bar\jet} \R^3).
 \]
  The canonical system $C$ (also called the contact distribution) on $J^1$ is defined by $C_\jet = (\pi^1_0)_*{}^{-1}(\jet) \subset T_\jet J^1$.
  
 Let us examine this in local coordinates.  Let $\jet_0 \in J^1$, $\bar\jet_0 = \pi^1_0(\jet_0)$.  Pick local coordinates $(x,y,z)$ on an open set $\bar{U} \subset \R^3$ about $\bar\jet_0$ such that $dx \wedge dy|_{\jet_0} \neq 0$.  By continuity, $dx \wedge dy |_\jet \neq 0$ for all $\jet$ in some neighbourhood $U \subset (\pi^1_0)^{-1}(\bar{U})$ about $\jet_0$.  Since $\jet$ is 2-dimensional and $dx|_\jet, dy|_\jet$ are linearly independent, then on $U$,
 \[
 dz|_\jet = p(\jet) dx|_\jet + q(\jet) dy|_\jet.
 \]
 This defines local coordinates $p = p(\jet), q=q(\jet)$.  Hence, $\jet = span\{ D_x, D_y \} \in Gr(2,T_{\bar\jet} \R^3)$, where 
  \[
 D_x = \l.\parder{x}\r|_{\bar\jet} + p(\jet) \l.\parder{z}\r|_{\bar\jet}, \qquad
 D_y = \l.\parder{y}\r|_{\bar\jet} + q(\jet) \l.\parder{z}\r|_{\bar\jet}.
 \]
 Pulling back $J^0$ coordinates to $J^1$,  $(x,y,z,p,q)$ are local coordinates on $J^1$, $\mathcal{I}$ is generated by the contact 1-form $\sigma = dz-pdx - qdy$, and $\mathcal{J} = \{ dz, dx, dy \}$ is associated with the independence condition $dx \wedge dy \neq 0$.  Finally,
 \begin{align}
 C = \{ \sigma = 0 \} = span\l\{ \parder{x} + p \parder{z},  \parder{y} + q \parder{z}, \parder{p}, \parder{q} \r\}. \label{C-basis}
 \end{align}
 
 The exterior derivative $\eta = d\sigma = dx \wedge dp + dy \wedge dq$ restricts to be a non-degenerate (conformal) symplectic form on the distribution $C$.  (Only the conformal class of $\eta$ is relevant since $\sigma$ is only well-defined up to scale.)  With respect to the basis given in \eqref{C-basis}, $\eta$ is represented by $\mat{cc}{0 & I_2\\ -I_2 & 0}$, where $I_2$ is the $2\times 2$ identity matrix.  We define the second jet space $J^2$ as the Lagrange--Grassmann bundle $LG(C,[\eta])$ over $J^1$, with projection $\pi^2_1 : J^2 \ra J^1$, i.e.\
 \[
  J^2 = \bigcup_{\jet \in J^1} J^2|_\jet := \bigcup_{\jet\in J^1} LG(C_\jet,[\eta]).
 \]
  Namely, any $\tilde\jet \in J^2$ is a Lagrangian (i.e.\ 2-dimensional isotropic) subspace of $(C_\jet, [\eta])$, where $\jet = \pi^2_1(\tilde\jet)$.  Since $(C_\jet,[\eta])$ is isomorphic to $\R^4$ with its canonical (conformal) symplectic form, it is clear that the fibres of $\pi^2_1$ are all diffeomorphic to $\Lg$, which has dimension 3.  This identification is of course not canonical since there is no preferred basepoint.  The canonical system on $J^2$ is given by $\tilde{C}_{\tilde\jet} = (\pi^2_1)_*{}^{-1}(\tilde\jet) \subset T_{\tilde\jet} J^2$.

 Given $\tilde\jet_0 \in J^2$, pick adapted local coordinates $(x,y,z,p,q)$ (as above) on a neighbourhood $U$ of $\jet_0 = \pi^2_1(\tilde\jet_0) \in J^1$ with $dx \wedge dy |_{\tilde\jet_0} \neq 0$.  Define $\tilde{U} = \{ \tilde\jet \in (\pi^2_1)^{-1}(U) : dx \wedge dy |_{\tilde\jet} \neq 0 \}$.
 For any $\tilde\jet \in \tilde{U}$, define $r(\tilde\jet), s_1(\tilde\jet), s_2(\tilde\jet), t(\tilde\jet)$ by
 \[
 dp|_{\tilde\jet} = r(\tilde\jet) dx|_{\tilde\jet} + s_1(\tilde\jet) dy|_{\tilde\jet}, \qquad
 dq|_{\tilde\jet} = s_2(\tilde\jet) dx|_{\tilde\jet} + t(\tilde\jet) dy|_{\tilde\jet}.
 \]
 Since $\tilde\jet$ is Lagrangian, then on $\tilde\jet$ we have $0 = \eta = d\sigma = dx \wedge dp + dy \wedge dq = (s_1 - s_2) dx \wedge dy$ so that $s_1 = s_2$.  Letting $s:= s_1 = s_2$, we have $\tilde\jet = span\{ \tilde{D}_x, \tilde{D}_y \}\in LG(C_\jet,[\eta]) \subset Gr(2, TJ^1)$, where
 \begin{align}
 \tilde{D}_x = \l.\parder{x}\r|_{\jet} + p(\jet) \l.\parder{z}\r|_{\jet} + r(\tilde\jet) \l.\parder{p}\r|_\jet + s(\tilde\jet) \l.\parder{q}\r|_\jet, \qquad
 \tilde{D}_y = \l.\parder{y}\r|_{\jet} + q(\jet) \l.\parder{z}\r|_{\jet} + s(\tilde\jet) \l.\parder{p}\r|_\jet + t(\tilde\jet) \l.\parder{q}\r|_\jet. 
 \label{Lag-DxDy}
 \end{align}
 Pulling back the $J^1$ coordinates $(x,y,z,p,q)$ to $J^2$, we have a local coordinate system $(x,y,z,p,q,r,s,t)$ on $J^2$.  We call such coordinates {\em standard}. The canonical system $\tilde{C}$ is given by $\tilde{C} = \{ \sigma = \sigma^1 = \sigma^2 = 0 \}$, where
 \[
 \sigma = dz - pdx - qdy, \qquad \sigma^1 = dp - rdx - sdy, \qquad \sigma^2 = dq - sdx - tdy.
 \]
 We see that for sections $\R^2 \ra J^2$, $(x,y) \ra (x,y,z,p,q,r,s,t)$, on which $\sigma = \sigma^1 = \sigma^2 = 0$, we have 
 \[
 p=z_x, \quad q = z_y,\quad r = z_{xx},\quad s = z_{xy},\quad  t = z_{yy}.
 \]
 
 \subsection{Contact transformations}
 \label{sec:contact-transformations} 
 
 \begin{defn}
 A (local) contact transformation of $J^1$ {\rm[}or $J^2${\rm]} is a (local) diffeomorphism preserving the canonical system $C$ {\rm[}or $\tilde{C}${\rm]} under pushforward.
 \end{defn}
 
 Any contact transformation $\phi$ of $J^1$ prolongs to a contact transformation $\pr(\phi)$ of $J^2$, i.e. $\phi_*$ acts on $C$, inducing a map of Lagrangian--Grassmannians $\pr(\phi) : J^2|_\jet \ra J^2|_{\jet'}$.  In particular, $\pr(\phi)$ preserves fibres of $\pi^2_1 : J^2 \ra J^1$.  Conversely, by Backl\"und's theorem \cite{Olver1995},  if $\tilde\phi$ on $J^2$ is contact, then $\tilde\phi = \pr(\phi)$ for some $\phi$ contact on $J^1$.
 
 \begin{example} \label{ex:Legendre} It is well-known that (for scalar-valued maps) there are contact transformations of $J^1$ that are not the prolongation of diffeomorphisms of $J^0$.  An example of this is the Legendre transformation 
 $(\bar{x},\bar{y},\bar{z},\bar{p},\bar{q}) = (-p,-q,z-px-qy,x,y)$, which satisfies $\bar\sigma = d\bar{z} - \bar{p} d\bar{x} - \bar{q} d\bar{y} = d(z-px-qy) +x dp +ydq = \sigma$.
 \end{example}

 Since $\phi$ is contact, then $\phi^*\sigma = \lambda \sigma$, for some function $\lambda$ on $J^1$, which implies for $\eta = d\sigma$ that $\phi^*\eta = \phi^* d\sigma = d(\phi^*\sigma) = d\lambda \wedge \sigma + \lambda \eta$.  Restricting to $C = \{ \sigma = 0\}$, we see that $\eta$ is preserved up to the overall factor of $\lambda$, i.e.\ $\phi_* : (C_\jet,[\eta]) \ra (C_{\jet'},[\eta])$ is a conformal symplecticomorphism.  If $\jet' = \phi(\jet) = \jet$, then $\phi_* \in CSp(C_\jet,[\eta])$ and this acts on $J^2|_\jet$.   Conversely, one may ask if, given $\jet \in J^1$, any element of $CSp(C_\jet,[\eta])$ is realized by a local contact transformation of $J^2$.  This question has an affirmative answer, but we postpone the proof until Section \ref{sec:symplectic-grp}.   These observations motivate our study of surfaces in $\Lg$ modulo $CSp(4,\R)$ (as opposed to a smaller subgroup).
  
  \subsection{Contact invariants for PDE induced from $CSp(4,\R)$-invariants for surfaces}
  \label{sec:inv-LG-PDE}
  
 Consider a second order scalar PDE in the plane $F = 0$, regarded as a hypersurface $\Sigma$ in $J^2$.  We use the usual non-degeneracy assumption that $\Sigma$ is transverse to $\pi^2_1 : J^2 \ra J^1$ and that $\pi^2_1 |_\Sigma : \Sigma \ra J^1$ is a submersion.
 Given $\jet \in J^1$, consider the fibre $J^2|_\jet = (\pi^2_1)^{-1}(\jet) = LG(C_\jet,[\eta])$ and $\Sigma|_\jet = \Sigma \cap J^2|_\jet$, which is a 2-dimensional surface.  We now discuss the transfer of differential invariants from the standard $\Lg$ setting to the PDE setting.  For now, simply regard $\Lg$ as the set of isotropic 2-planes on the standard $(\R^4,[\eta])$ with standard basis and $CSp(4,\R)$ consists of linear transformations of $\R^4$ which preserve $\eta$ up to scale.

Let $U$ be a 2-dimensional connected manifold.  Let $J^n(U,\Lg)$ denote the $n$-th order jet space of maps $s : U \ra \Lg$.  The action of $g \in CSp(4,\R)$ prolongs to an action $g^{(n)}$ on $J^n(U,\Lg)$, c.f. \cite{Olver1995} or \cite{Saunders1989}.
 
 \begin{defn}
 An $n$-th order differential invariant for $CSp(4,\R)$ is a function $\kappa : J^n(U,\Lg) \ra \R$ such that $\kappa( g^{(n)} \cdot {\bf p}) = \kappa({\bf p})$, for any $g \in CSp(4,\R)$ and ${\bf p} \in J^n(U,\Lg)$.  For short, we call $\kappa$ a $CSp(4,\R)$-invariant.
 \end{defn}

 \begin{example}  We will endow $\Lg$ with natural coordinates $(r,s,t)$ and the conformal structure $[drdt - ds^2]$.
 Locally, $J^1(U,\Lg)$ is given by $(u,v,r,s,t,r_u,r_v,s_u,s_v,t_u,t_v)$.  The sign of the determinant of
 \[
 \mat{cc}{ r_u t_u - s_u{}^2 & \frac{1}{2} r_u t_v + \frac{1}{2} r_v t_u - s_u s_v\\ \frac{1}{2} r_u t_v + \frac{1}{2} r_v t_u - s_u s_v & r_v t_v - s_v{}^2}
 \]
 is a first-order $CSp(4,\R)$-invariant, distinguishing hyperbolic (timelike), parabolic (null), elliptic (spacelike) surfaces.
 \end{example}
 
 \begin{defn} Let $(W,[\eta])$ be a 4-dimensional real vector space endowed with a conformal symplectic form.  A {\em conformal symplectic (\cs/) basis} $\textsf{{\em b}} = \{ e_i \}_{i=1}^4$ is a basis of $W$ such that the matrix $(\eta(e_i,e_j))_{i,j=1}^4$ is a multiple of $\mat{cccc}{0 & I_2\\ -I_2 & 0}$.  The first two vectors $e_1,e_2$ of $\textsf{b}$ span a Lagrangian subspace, i.e. an element of $LG(W,[\eta])$.
 \end{defn}
 
 \begin{lem}
 A $CSp(4,\R)$-invariant $\kappa : J^n(U,\Lg) \ra \R$ induces a $CSp(C_\jet,[\eta])$-invariant $\kappa_\jet : J^n(U,J^2|_\jet) \ra \R$. 
 \end{lem}

 \begin{proof}
 Fix any \cs/ basis $\textsf{b}$ of $(C_\jet,[\eta])$.  This defines an isomorphism (conformal symplectomorphism) $\psi_\textsf{b} : C_\jet \ra \R^4$ by mapping $\textsf{b}$ onto the standard basis of $\R^4$, hence induces $\tilde\psi_{\textsf{b}} : J^2|_\jet \stackrel{\cong}{\ra} LG(2,4)$ and identifies $\widehat\psi_\textsf{b} : CSp(C_\jet,[\eta]) \stackrel{\cong}{\ra} CSp(4,\R)$.  Note that for $g \in CSp(C_\jet,[\eta])$ and $g_\textsf{b} = \widehat\psi_\textsf{b}(g) \in CSp(4,\R)$, we have:
 \begin{align}
 \psi_{g \cdot \textsf{b}} \circ g = \psi_\textsf{b}, \qquad  \psi_\textsf{b} \circ g = g_\textsf{b} \circ \psi_\textsf{b},
 \label{psi-id}
 \end{align}
 with similar identities for $\tilde\psi_\textsf{b}$.  Given a $CSp(4,\R)$-invariant $\kappa$, define $\kappa_\jet : J^n(U,J^2|_\jet) \ra \R$ by 
 $\kappa_\jet(j^n_x s) := \kappa(j^n_x (\tilde\psi_{\textsf{b}} \circ s))$,
 where $s: U \ra J^2|_\jet$ is a local map and $j^n_x s$ is the $n$-th jet prolongation of $s$ evaluated at $x \in U$.  We claim:
 \begin{enumerate}
 \item $\kappa_\jet$ is well-defined: Let $\textsf{b}'$ be another \cs/ basis, so $\textsf{b}' = g \cdot \textsf{b}$ for some $g \in CSp(C_\jet,[\eta])$.  By \eqref{psi-id}, $\kappa(j^n_x (\tilde\psi_{\textsf{b}'} \circ s)) = \kappa(j^n_x (\tilde\psi_{\textsf{b}} \circ g^{-1} \circ s))= \kappa((g_\textsf{b}{}^{-1})^{(n)} \cdot j^n_x (\tilde\psi_{\textsf{b}} \circ s )) = \kappa(j^n_x (\tilde\psi_{\textsf{b}} \circ s) )$.
 \item $\kappa_\jet$ is $CSp(C_\jet,[\eta])$-invariant: Fix $g\in CSp(C_\jet,[\eta])$.  By \eqref{psi-id}, $\kappa_\jet(g^{(n)} \cdot j^n_x s) = \kappa_\jet( j^n_x (g\cdot s)) = \kappa(j^n_x(\tilde\psi_{\textsf{b}} \circ (g\cdot s))) =  \kappa(j^n_x(g_\textsf{b} \cdot (\tilde\psi_{\textsf{b}} \circ s))) = \kappa(g_\textsf{b}{}^{(n)} \cdot j^n_x( \tilde\psi_{\textsf{b}} \circ s)) = \kappa( j^n_x( \tilde\psi_{\textsf{b}} \circ s)) = \kappa_\jet(j^n_x s)$.
 \end{enumerate}
 \end{proof}

 The function $\kappa_\jet$ accounts for ``vertical'' derivatives, i.e. derivatives only along the fibre $J^2|_\jet$.  Consider the bundle $\mathbb{J}^n:= \dot\bigcup_{\jet \in J^1} J^n(U,J^2|_\jet)$ over $J^1$.  Given a contact transformation $\tilde\phi : J^2 \ra J^2$ and a map $s : U \ra J^2|_\jet$, we have $\tilde\phi \circ s : U \ra J^2|_{\phi(\jet)}$, hence we define $\Phi^{(n)} : \mathbb{J}^n \ra \mathbb{J}^n$ by $\Phi^{(n)}(j^n_x s) := j^n_x (\tilde\phi \circ s)$.  Define $\textsf{K} :  \mathbb{J}^n \ra \R$ by $\textsf{K}({\bf p}) = \kappa_{\jet}({\bf p})$, where ${\bf p} \in J^n(U,J^2|_\jet)$.  We say that $\textsf{K}$ is contact-invariant if $\textsf{K} \circ \Phi^{(n)} = \textsf{K}$ for any contact transformation $\tilde\phi : J^2 \ra J^2$.
 
 \begin{thm} \label{thm:symplectic-contact} Any $CSp(4,\R)$-invariant $\kappa : J^n(U,LG(2,4)) \ra \R$ induces a contact-invariant function $\textsf{K} : \mathbb{J}^n \ra \R$.  In particular, for any constant $c \in \R$, $\textsf{K} = c$ defines a contact-invariant class of PDE.
 \end{thm}
 
 \begin{proof}
 To show $\textsf{K}$ is contact-invariant, we show $\kappa_{\phi(\jet)}(\Phi^{(n)}(j^n_x s)) = \kappa_{\jet}(j^n_x s)$
 for any contact $\tilde\phi = \pr(\phi)  : J^2 \ra J^2$ and $s : U \ra J^2|_\jet$.
   If $\phi(\jet) = \jet$, then $\tilde\phi \in CSp(C_\jet,[\eta])$, so by $CSp(C_\jet,[\eta])$-invariance, $\kappa_\jet \circ \Phi^{(n)} = \kappa_\jet$.  If $\phi(\jet) \neq \jet$, choose any \cs/ basis $\textsf{b} = \{ e_i \}_{i=1}^4$ of $(C_\jet,[\eta])$.  Since $\phi_*$ is a conformal symplectomorphism, then $\textsf{b}' = \{ e_i' = \phi_*(e_i)\}_{i=1}^4$ is also a \cs/ basis.  We have $ \kappa_{\phi(\jet)}(\Phi^{(n)}(j^n_x s)) = \kappa_{\phi(\jet)}(j^n_x (\tilde\phi \circ s)) = \kappa(j^n_x(\tilde\psi_{\textsf{b}'} \circ \tilde\phi \circ s )) = \kappa(j^n_x(\tilde\psi_{\textsf{b}} \circ s )) = \kappa_\jet(j^n_x s)$. 
 \end{proof}

 \begin{rem}
 Theorem \ref{thm:symplectic-contact} clearly generalizes to $J^2(\R^n,\R)$: Submanifold theory in $LG(n,2n)$ modulo $CSp(2n,\R)$ has implications for the contact-invariant study of systems of scalar PDE in $n$ independent variables.
 \end{rem}
 
 For any surface $M \subset \Lg$ (as for a PDE) there is no distinguished choice of (local) parametrization $i : U \ra M$ and hence invariants that we seek should be independent of reparametrization.  Thus, we are ultimately interested in the {\em unparametrized} equivalence problem: {\em Given $i : U \ra M$ and $\tilde{i} : \tilde{U} \ra \tilde{M}$, does there exist a diffeomorphism $\varphi : U \ra \tilde{U}$ and an element $g \in CSp(4,\R)$ such that $\tilde{i} \circ \varphi = g \cdot i$ on $U$?}  We will study this problem in three steps:
 \begin{enumerate}
 \item Apply Cartan's method of moving frames to find invariants for the parametrized equivalence problem.
 \item Give a parametric description of the invariants of the parametrized problem.  There is a natural choice of coordinates $(r,s,t)$ on $\Lg$ which correspond to the 2nd derivative coordinates in the PDE setting.
 \item Investigate how the above invariants change under reparametrization.  Unchanged properties will be invariants for the unparametrized equivalence problem.
 \end{enumerate}
 
 Generally, the simplest surfaces to describe are those with constant $CSp(4,\R)$-invariants $\kappa$.  Surfaces in $\Lg$ which have {\em all} constant symplectic ($CSp(4,\R)$) invariants (CSI) have a natural geometric meaning:
 
 \begin{prop}[Homogeneity theorem, p.42 of \cite{Jensen1977}] \label{thm:homogeneity} For any smooth embedding $i : U \ra G/H$, $M = i(U)$ is an open submanifold of a homogeneous submanifold of $G/H$ iff $M$ has no non-constant $G$-invariants.
 \end{prop}

 \begin{defn} \label{defn:CSI}
 Let $\Sigma \subset J^2$ be a PDE.  If for any (well-defined) $CSp(4,\R)$-invariant $\kappa$ of surfaces in $\Lg$ the corresponding contact-invariant $\textsf{K}$ is constant on $\Sigma$, then we call $\Sigma$ a {\em constant symplectic invariant} (CSI) PDE.
 \end{defn}
 
 Several examples of hyperbolic CSI PDE are given in Figure \ref{fig:examples}.  

 \section{Conformal geometry of $\Lg$}
 \label{sec:prelim}
 
 In Section \ref{sec:PDE}, our main result was that $CSp(4,\R)$-invariants for surfaces in $\Lg$ induce contact-invariants for PDE.  With this in mind, we describe the ambient geometry of $\Lg$ with a view to preparing for our later study of hyperbolic surfaces via the method of moving frames.
 
 \subsection{$\Lg$ and $CSp(4,\R)$}
 \label{sec:symplectic-grp}
 
 On $\R^4$, take the standard basis $\{ e_i \}_{i=1}^4$ with dual basis $\{ \eta^i \}_{i=1}^4$.  Let $\eta = 2(\eta^1 \wedge \eta^3 + \eta^2 \wedge \eta^4)$ be the standard symplectic form represented by $J = \mat{cc}{ 0 & I_2\\ -I_2 & 0}$, so $\eta(x,y) = \transpose{x} J y$.  We define:
 \begin{align*}
 Sp(4,\R) &= \l\{ X \in GL(4,\R) : \transpose{X} J X = J \r\} =  \l\{ \mat{cc}{ A& B\\ C & D} :  \transpose{A} C = \transpose{C} A, \,  \transpose{B} D = \transpose{D} B, \,\transpose{A} D -\transpose{C} B = I_2 \r\}, \\
 \sp(4,\R) &= \l\{ X \in Mat_{4 \times 4}(\R): \transpose{X} J  + J X  =0  \r\} = \l\{ \mat{cc}{ a & b\\ c & -\transpose{a}} :  b,c \mbox{ symmetric} \r\},
 \end{align*}
 which are 10-dimensional.  $\Lg$ is the set of $\eta$-isotropic 2-planes in $\R^4$.  This admits a transitive $Sp(4,\R)$-action, hence $\Lg = Sp(4,\R)/P$ up to a choice of basepoint $o$.  Choosing $o = span\{ e_1, e_2\}$, we see $\dim(P) = 7$ since
 \begin{align*}
 P = \l\{ \mat{cc}{ A & B\\ 0 & (\transpose{A})^{-1}} : A \in GL(2,\R), A^{-1} B \mbox{ symmetric} \r\}, \qquad
 \p = \l\{ \mat{cc}{ a & b\\ 0 & -\transpose{a}} :  b \mbox{ symmetric} \r\}.
 \end{align*}
 Hence, $\dim(\Lg) = 3$.  
The $Sp(4,\R)$-action is not effective: the global isotropy subgroup is $K =  \{ \pm I_4 \} \cong \Z_2$.
 
 In the PDE context (Section \ref{sec:jet-space}), only the conformal class $[\eta]$ is relevant.  This does not affect the definition of $\Lg$ since a 2-plane is $\eta$-isotropic iff it is $(\lambda \eta)$-isotropic.  We consider instead the conformal symplectic group 
 \begin{align*}
  CSp(4,\R) &= \l\{ X \in GL(4,\R) : \transpose{X} J X = \rho J, \mbox{ for some } \rho \in \R^\times \r\} \cong Sp(4,\R) \rtimes_\varphi \R^\times, 
 \end{align*}
 where $\R^\times$ is embedded into $GL(4,\R)$ as $ \rho \mapsto  \diag(\rho I_2,I_2)$ and note that $\varphi : \R^\times \ra Aut(Sp(4,\R))$ by
 \[
 \varphi(\rho) \mat{cc}{ A & B \\ C & D} = \mat{cc}{ \rho I_2 & 0\\ 0 & I_2} \mat{cc}{ A & B \\ C & D} \mat{cc}{ \rho^{-1} I_2 & 0\\ 0 & I_2} = \mat{cc}{ A & \rho B \\ \rho^{-1} C & D}.
 \]
 Since $diag(\sqrt{\rho}{}^{-1} I_2, \sqrt{\rho} I_2) \cdot diag(\rho I_2, I_2) = \sqrt{\rho} I_4$ acts trivially, then $\rho \in \R^\times$ and $diag(\sqrt{\rho} I_2, \sqrt{\rho}{}^{-1} I_2) \in P$ have the same action on $\Lg$.  Most of $\R^\times$ is redundant and it suffices to consider $\widehat{Sp}(4,\R) = Sp(4,\R) \rtimes_\varphi \Z_2$, where $\Z_2$ is generated by $\rho = -1$, i.e. $diag(-I_2,I_2)$.  The stabilizer in $\widehat{Sp}(4,\R)$ of $o$ is $\widehat{P} = P \rtimes \Z_2$, and global isotropy is $K$.
  
 Let us describe a collection of charts on general $LG(n,2n)$.  For more details, see \cite{PT}.  A {\em Lagrangian decomposition} of $(\R^{2n},\eta)$ is a pair $(L_0, L_1)$ of transversal Lagrangian subspaces in $LG(n,2n)$, i.e. $\R^{2n} = L_0 \oplus L_1$ and $L_0 \cap L_1 = 0$.  Define $\trans{L_1} = \{ L \in LG(n,2n) : L \mbox{ transverse to } L_1 \}$ and $B_{sym}(L_0) = \{ \mbox{symmetric bilinear forms on } L_0 \}$.  Associated to any Lagrangian decomposition $(L_0,L_1)$ is a chart $\psi_{L_0,L_1} = \rho_{L_0,L_1} \circ \phi_{L_0,L_1}: \trans{L_1} \ra B_{sym}(L_0)$, 
 where the maps $\phi_{L_0,L_1} : \trans{L_1} \ra Lin(L_0,L_1)$ and $\rho_{L_0,L_1} : L_1 \ra L_0{}^*$ are defined by 
 \begin{align}
 \phi_{L_0,L_1}(L) = T_L, \qquad \rho_{L_0,L_1}(v) = \eta(\cdot,v)|_{L_0}
 \end{align}
 Here, $T_L$ is uniquely defined by $L = graph(T_L) = \{ v + T_L(v) : v \in L_0 \}$.  A priori the image of $\psi_{L_0,L_1}$ is in $B(L_0)$, but in fact $\psi_{L_0,L_1}$ is a bijection onto $B_{sym}(L_0)$.  Moreover, the collection of charts as $(L_0,L_1)$ ranges over all Lagrangian decompositions yields a differentiable atlas for $LG(n,2n)$ \cite{PT}.
 
 Let $L_0 = o = span\{ e_1, e_2 \}$ and $L_1 = span\{ e_3, e_4 \}$.  We describe the chart $\psi = \psi_{L_0,L_1}$ on $\Lg$.  From above,
 \begin{align}
 L = span\{ \tilde{e}_1 = e_1 + r e_3 + s e_4, \,\, \tilde{e}_2 = e_2 + s e_3 + t e_4 \} \in \trans{L_1} \quad \stackrel{\psi}{\longmapsto}\quad \mat{cc}{r & s\\s & t}.
 \label{psi-pt}
 \end{align}
 The right side of \eqref{psi-pt} is the matrix for: (i) $T_L \in Lin(L_0,L_1)$ in the bases $\{ e_1, e_2 \}, \{ e_3, e_4 \}$, or (ii) $\psi(L) \in B_{sym}(L_0)$ in the basis $\{ e_1, e_2 \}$.  Equivalently, exponentiate $X = \mat{cc}{r & s\\ s & t}$ on $\sp(4,\R) / \p$ to obtain $\mat{cc}{ I_2 & 0\\ X & I_2 }/ P$ on  $\Lg = Sp(4,\R) / P$, and the first two columns yield the left side of \eqref{psi-pt}.  From \eqref{C-basis} and \eqref{Lag-DxDy}, the fibrewise identification 
 \begin{align}
 e_1 = \partial_x + p \partial_z,\quad 
 e_2 = \partial_y + q \partial_z,\quad
 e_3 = \partial_p,\quad e_4 = \partial_q \label{LG-J2-id}
 \end{align}
 implies that these $(r,s,t)$ coordinates in $\Lg$ correspond to the natural $(r,s,t)$ coordinates in the $J^2$ setting.  Hereafter, whenever we write $(r,s,t)$, we mean such coordinates in one setting or the other.
  
 Now let us see the $\widehat{Sp}(4,\R)$-action in the chart $\psi$.  The element $\rho = -1 \in \Z_2$ acts as $(r,s,t) \mapsto (-r,-s,-t)$.  For elements in $Sp(4,\R)$ sufficiently close to the identity,
 \begin{align}
 \mat{cc}{A & B\\ C &D} \cdot \mat{cc}{I & 0\\ X &I}/P=  \mat{cc}{A+BX & B\\ C+DX &D}/P
 =  \mat{cc}{I & 0\\ \tilde{X} &I}/P,
 \end{align}
 so we have the M\"obius-like transformation
 $\tilde{X} = (C+DX)(A+BX)^{-1}$.  Table \ref{table:inf-gen} displays the infinitesimal generators.

 \begin{table}[h]
 \begin{center}
 $\begin{array}{|c|c||c|c||c|c|}\hline
 \multicolumn{2}{|c||}{\mbox{Rotations}} & \multicolumn{2}{|c||}{\mbox{Inversions}} & \multicolumn{2}{|c|}{\mbox{Translations}}\\ \hline\hline
 a_{11} &  -2r \partial_r - s\partial_s & b_1 & -r^2 \partial_r - rs\partial_s - s^2 \partial_t & c_1 & \partial_r\\
 a_{12} &  -r \partial_s - 2s\partial_t & b_2 &  -s^2 \partial_r - st\partial_s - t^2 \partial_t & c_2 & \partial_t\\
 a_{21} &  -2s \partial_r - t\partial_s & b_3 &  -2rs \partial_r - (s^2+rt)\partial_s - 2st \partial_t & c_3 & \partial_s\\
 a_{22} &  -s \partial_s - 2t\partial_t & & & & \\ \hline
 \end{array}$
 \caption{Infinitesimal generators for $\sp(4,\R)$-action on $\Lg$}
  \label{table:inf-gen}
 \end{center}
 \end{table}

We now address the claim at the end of Section \ref{sec:contact-transformations}.

 \begin{prop} \label{prop:jet-symp}
 Given $\jet \in J^1$, any element of $CSp(C_\jet,[\eta])$ is realized by a local contact transformation of $J^2$.
 \end{prop}
 
 \begin{proof} Fix $\jet \in J^1$.  Choose coordinates $(x,y,z,p,q)$ on $J^1$ as in Section \ref{sec:jet-space} such that $\jet = (x,y,z,p,q) = (0,0,0,0,0)$.  As in \eqref{LG-J2-id}, we identify the $(r,s,t)$ coordinates in the $\Lg$ and $J^2|_\jet$ settings.
 Next, we realize the vector fields in Table \ref{table:inf-gen} as the $(r,s,t)$ components of a prolongation ${\bf X}^{(2)}$ of a contact symmetry ${\bf X}$.
 \begin{align*}
 \begin{array}{|c|c|c|} \hline
  {\bf X} & {\bf X}^{(2)} & (r,s,t)\mbox{ components of } {\bf X}^{(2)}\\ \hline\hline
 \frac{1}{2}x^2\partial_z & \frac{1}{2} x^2 \partial_z + x \partial_p + \partial_r & \partial_r\\
 xy\partial_z & xy\partial_z + y\partial_p + x\partial_q + \partial_s & \partial_s \\
 \frac{1}{2}y^2\partial_z & \frac{1}{2} y^2 \partial_z + y \partial_q + \partial_t & \partial_t \\ \hline
 x\partial_x & x\partial_x - p \partial_p - 2r \partial_r - s \partial_s & - 2r \partial_r - s \partial_s \\
 y\partial_x & y \partial_x - p \partial_q - r \partial_s - 2s\partial_t & - r \partial_s - 2s\partial_t \\
 x\partial_y & x \partial_y - q \partial_p - 2s \partial_r - t \partial_s & - 2s \partial_r - t \partial_s \\ 
 y\partial_y & y\partial_y - q \partial_q - s \partial_s - 2t \partial_t & - s \partial_s - 2t \partial_t \\ \hline
 p\partial_x + \frac{p^2}{2} \partial_z & p\partial_x + \frac{p^2}{2} \partial_z - r^2 \partial_r - rs\partial_s - s^2\partial_t & - r^2 \partial_r - rs\partial_s - s^2\partial_t \\
 q\partial_x + p\partial_y + pq \partial_z & q\partial_x + p\partial_y + pq\partial_z - 2rs\partial_r - (s^2 + rt)\partial_s - 2st\partial_t & - 2rs\partial_r - (s^2 + rt)\partial_s - 2st\partial_t \\
 q\partial_y + \frac{q^2}{2} \partial_z & q\partial_y + \frac{q^2}{2} \partial_z - s^2 \partial_r - st\partial_s - t^2\partial_t & - s^2 \partial_r - st\partial_s - t^2\partial_t\\ \hline
 \end{array}
 \end{align*}
 Since the components of all ${\bf X}^{(1)} := (\pi^2_1)_* {\bf X}^{(2)}$ above are homogeneous in $(x,y,z,p,q)$, then restricted to $\jet = (0,0,0,0,0)$, these are vertical vector fields for the projection $\pi^2_1$, i.e. they preserve $J^2|_\jet = LG(C_\jet,[\eta])$.  Thus, any infinitesimal generator of the $CSp(C_\jet,[\eta])$-action on $J^2|_\jet$ is realizable by an infinitesimal contact transformation.  But $Sp(4,\R) \cong Sp(C_\jet,\eta)$ is connected so this is true in terms of regular (finite) transformations.  Lastly, if $\rho \in \R^\times$, the contact transformation $(x,y,z,p,q) \mapsto (\rho x, \rho y, \rho z,p,q)$ preserves $J^2|_\jet$ and induces $\eta \mapsto \rho \eta$.
 \end{proof}

 Because of Proposition \ref{prop:jet-symp}, we can easily recognize various transformations of $\Lg$ from their corresponding jet transformations.  For example, in the jet space setting we know the scalings $(x,y,z) \ra (\lambda_1 x, \lambda_2 y, \mu z)$ induces $(r,s,t) \ra (\frac{\mu}{\lambda_1{}^2}r,\frac{\mu}{\lambda_1\lambda_2}s,\frac{\mu}{\lambda_2{}^2}t)$.  Hence, the latter is also a transformation of $\Lg$.

 \subsection{$\Lg$ as a quadric hypersurface in $\R\P^4$}
 \label{sec:quadric}

 By the Pl\"ucker embedding, the full Grassmannian $Gr(2,4)$ embeds as the decomposable elements in $\P(\w^2\R^4)$.   The induced $\widehat{Sp}(4,\R)$-action on $\w^2 \R^4$ preserves $V:= \w^2_0 \R^4 = \{ z \in \w^2 \R^4 : \eta(z) = 0 \}$, which is a 5-dimensional irreducible subspace.
 The restriction of the Pl\"ucker embedding to $\Lg$ has image the quadric hypersurface $\Q = \l\{ [z] \in \P V : \ambient{z}{z} = 0 \r\}$, where $\ambient{z_1}{z_2} = 3(\eta \wedge \eta)(z_1 \wedge z_2)$ is a non-degenerate symmetric bilinear form.
 Note that $3(\eta \wedge \eta)(e_1 \wedge e_2 \wedge e_3 \wedge e_4) = -24(\eta^1 \wedge \eta^2 \wedge \eta^3 \wedge \eta^4)(e_1,e_2,e_3,e_4) = -\det(\eta^i(e_j)) = -1$.
 On $V$, take the basis
 \begin{align}
 \mathcal{B}:\quad \begin{array}{l} e_1 \wedge e_2, \\ e_3 \wedge e_2, \\ e_1 \wedge e_3 - e_2 \wedge e_4, \\ e_1 \wedge e_4, \\ e_3 \wedge e_4 \end{array} \qRa
 \ambient{\cdot}{\cdot}_\mathcal{B} =   \mat{ccccc}{ 0 & 0 & 0 & 0 & -1\\ 0 & 0 & 0 & 1 & 0\\ 0 & 0 & -2 & 0 & 0\\ 0 & 1 & 0 & 0 & 0\\ -1 & 0 & 0 & 0 & 0 },
 \label{B-basis-scalar-prod}
 \end{align}
 so $\ambient{\cdot}{\cdot}$ has signature $(2,3) = (++---)$.  Since $\Lg$ is $\widehat{Sp}(4,\R)$-invariant, so are $\Q$ and $\ambient{\cdot}{\cdot}$.  Fixing the basis $\mathcal{B}$, there is a homomorphism $\Phi : \widehat{Sp}(4,\R) \ra O(2,3)$.  The induced Lie algebra homomorphism $\phi = \Phi_* : \sp(4,\R) \ra \so(2,3)$ given by
 \begin{align}
 \phi \mat{cccc}{ a_{11} & a_{12} & b_1 & b_3 \\ a_{21} & a_{22} & b_3 & b_2 \\ c_1 & c_3 & -a_{11} & -a_{21} \\ c_3 & c_2 & -a_{12} & -a_{22} } 
 = \mat{ccccc}{ a_{11} + a_{22} & b_1 & 2b_3 & b_2 & 0\\ c_1 & a_{22} - a_{11} & -2a_{21} & 0 & b_2\\ c_3 & -a_{12} & 0 & -a_{21} & -b_3\\ c_2 & 0 & -2a_{12} & a_{11}-a_{22} & b_1 \\ 0 & c_2 & -2c_3 & c_1 & -(a_{11} + a_{22})}
 \label{B-iso}
 \end{align}
 is an isomorphism.   Since $Sp(4,\R)$ is connected and $O(2,3)$ has four connected components, then $\Phi$ yields a 2-to-1 covering map $Sp(4,\R) \ra SO^+(2,3)$, so $Sp(4,\R) \cong Spin(2,3)$.  The element $\rho = -1 \in \Z_2$ acts on $\mathcal{B}$ as the matrix $diag(1,-1,-1,-1,1)$
 which has determinant $-1$.  Hence, $\Phi : \widehat{Sp}(4,\R) \ra O^+(2,3)$ is surjective with kernel $K = \{ \pm I_4 \}$.  The effective symmetry group of $\Q$ is $O^+(2,3)$ with stabilizer $P^+ = \Phi(P)$ at $[e_1 \wedge e_2]$.
 
 \begin{rem} As groups, $O(2,3) / SO^+(2,3) \cong \Z_2 \times \Z_2$.  Let $V_+,V_-$ be the maximal positive and negative definite subspaces of $V$.
 Here, $O^+(2,3) \subset O(2,3)$ preserves the orientation of $V_-$, e.g. $-I \in O^-(2,3)$.
 \end{rem}
 
 We can write the condition $\ambient{z}{z}=0$ defining $\Q$ in coordinates as $|x|^2 - |y|^2 := x_1{}^2 + x_2{}^2 - y_1{}^2 - y_2{}^2 - y_3{}^2 = 0$, where $[z] = [(x,y)] \in \Q$ and we may assume $x \in S^1$, $y \in S^2$.  Since $[(x,y)] = [(-x,-y)]$, then $\Lg \cong (S^1 \times S^2) / \Z_2$.

 With respect to the chart $\psi$ in Section \ref{sec:symplectic-grp}, the Lagrangian subspace $L = span\{ \tilde{e}_1, \tilde{e}_2 \}$ in \eqref{psi-pt} satisfies
 \[
 \tilde{e}_1 \wedge \tilde{e}_2 = e_1 \wedge e_2 + r e_3 \wedge e_2 + s(e_1 \wedge e_3 + e_4 \wedge e_2) + t e_1 \wedge e_4 + (rt-s^2) e_3 \wedge e_4,
 \]
 so with respect to $\mathcal{B}$, $L = [1,r,s,t,rt-s^2] \in \Q$.  Thus, $\psi : [1, r,s,t, rt-s^2] \in \Q \longmapsto (r,s,t) \in \R^3$.  Points of $\Lg$ not covered by $\psi^{-1}$ are $\{ [0,a,b,c,d] : ac-b^2 = 0 \}$, which is precisely the degenerate sphere $\S{{\bf Z}}$ (see Section \ref{sec:spheres}), where ${\bf Z} = (0,0,0,0,1)$, and can be regarded as the ``sphere at infinity''. 
 
    Consider a second coordinate chart $\psi'$ centered at ${\bf Z}$, i.e. $span\{ e_3, e_4\}$ by taking $(e_1',e_2',e_3',e_4') = (e_3, e_4, -e_1, -e_2)$ and constructing similar coordinates $[1,r',s',t',r't'-(s')^2]_{\mathcal{B}'} = $ as above with respect to the new basis $\mathcal{B}'$.  Relating $\mathcal{B}'$ to $\mathcal{B}$, we see $[1,r',s',t',r't'-(s')^2]_{\mathcal{B}'} = [rt-s^2,-t,s,-r,1]_{\mathcal{B}'}$.  
Thus, the coordinate change formula $\psi' \circ \psi^{-1}$ is
 \begin{align}
 r' = \frac{-t}{rt-s^2}, \qquad s' = \frac{s}{rt-s^2}, \qquad t' = \frac{-r}{rt-s^2}.
 \label{chart-change}
 \end{align}
 In fact, at least three coordinate charts are needed to cover $\Lg$.  The set of points in $\Q$ not covered by $\psi^{-1}$ or $(\psi')^{-1}$ is $\Q \cap\{ [0,a,b,c,0] : ac-b^2 = 0 \}$.  In the PDE setting, the Legendre transformation (see Example \ref{ex:Legendre}) prolongs to give \eqref{chart-change} on the $J^2$ coordinates $(r,s,t)$ .

 \subsection{The invariant conformal structure}
 \label{sec:conf-str}
 
 There is a distinguished (up to sign) Lorentzian conformal structure $[\metric]$, equivalently a cone field $\mathcal{C} = \{ \metric = 0 \}$, on $\Lg$, which we describe here in three ways.  A fourth description is given in Section \ref{sec:MC}.
 
 A conformal structure $[\metric]$ is an equivalence class of metrics: $\metric \sim \metric'$ iff $\metric' = \lambda^2 \metric$, where $\lambda$ is nonvanishing. For conformal structures, all $\widehat{Sp}(4,\R)$-invariant ones on $\Lg$ correspond to $Ad(\widehat{P})$-invariant ones on
 $\sp(4,\R) / \p = \l\{ \mat{cc}{ 0 & 0 \\ c & 0}: c \mbox{ symmetric} \r\} / \p.$
 The element $\rho = -1$ in the $\Z_2$ factor of $\widehat{P}$ acts as $c \mapsto -c$.  Via the adjoint action,  
 \begin{align}
 \mat{cc}{ A & B \\ 0 & (\transpose{A})^{-1}} \cdot \mat{cc}{ 0 & 0 \\ c & 0} / \p = \mat{cc}{ 0 & 0 \\ (\transpose{A})^{-1} c A^{-1} & 0} / \p.
 \label{abstract-P-action}
 \end{align}
 Note $\det(c)$ transforms as $\det(c) \mapsto \det(A)^{-2} \det(c)$.  The conformal class of its polarization is $Ad(\widehat{P})$-invariant and corresponds to an $\widehat{Sp}(4,\R)$-invariant conformal structure $[\metric]$ on $\Lg$.  In fact, $[\pm \metric]$ are the only two such structures.   Thus, $\Lg$ is endowed with a canonical $CSp(4,\R)$-invariant cone field $\mathcal{C} = \{ \mu = 0 \}$.
  
 For a second description, let $\widehat{\Q} = \{ z \in V : \ambient{z}{z} = 0 \} = cone(\Q)$.  For any $[z] \in \Q$ and $\rho \in \R^\times$, the affine tangent space $\widehat{T}_{[z]} \Q$ is the tangent space to $\widehat{\Q}$ at any $\rho z$, translated to the origin, i.e. 
 \[
 \widehat{T}_{[z]} \Q = T_{\rho z} \mathcal{\widehat{Q}} = T_z \mathcal{\widehat{Q}} = \{ w : \ambient{z}{w} = 0 \} = z^\perp, \quad \dim(z^\perp) = 4.
 \]
 The canonical surjection $\pi : V \ra \P V$ restricts to $\pi : \mathcal{\widehat{Q}} \ra \Q$ and induces vector space isomorphisms $(\pi_*)_{\rho z} : z^\perp / \ell_z \ra T_{[z]} \Q$, where $\ell_z = span\{ z\}$.  These satisfy $(\pi_*)_{\rho z}(\rho w) = (\pi_*)_z(w)$, $\forall w \in z^\perp$.  The form induced by $\ambient{\cdot}{\cdot}$ on $z^\perp / \ell_z$ is non-degenerate and has signature $(1,2)$.  This transfers to $T_{[z]} \Q$ via $(\pi_*)_z$ or $(\pi_*)_{\rho z}$ and these differ by a factor of $\rho^2$.  This conformal class on $T_{[z]} \Q$ induces a conformal structure on all of $\Q$.  It is $\widehat{Sp}(4,\R)$-invariant because so is $\ambient{\cdot}{\cdot}$.

 \begin{rem} \label{rem:identify} Although $(\pi_*)_{\rho z}$ depends on $\rho \in \R^\times$, there is a well-defined mapping of subspaces since $(\pi_*)_{\rho z} = \frac{1}{\rho} (\pi_*)_z$, e.g.\ if $M \subset \Q$ is a submanifold, then $T_{[z]} M$ is identified with a subspace of $z^\perp / \ell_z$.  Alternatively, there is the canonical isomorphism $z^\perp / \ell_z \stackrel{\cong}{\ra} \ell_z \otimes T_{[z]} \Q$, $[w]  \mapsto \rho z \otimes (\pi_*)_{\rho z}(w) = z \otimes (\pi_*)_z(w)$, or equivalently, $\ell_z^* \otimes (z^\perp / \ell_z) \cong T_{[z]} \Q$.
 \end{rem}
%
 
 We have a third description in terms of the local coordinates $(r,s,t)$ in the chart $\psi$.  Taking $c = \mat{cc}{r & s \\ s & t}$ in \eqref{abstract-P-action}, the invariant conformal structure $[\metric]$ has representative $\metric = dr \odot dt - ds \odot ds = dr dt - ds^2$.  Each vector field $Z$ in Table \ref{table:inf-gen} preserves $[\metric]$ under Lie derivation, i.e.\ $\mathcal{L}_Z \metric = \lambda(Z) \metric$ for some scalar function $\lambda(Z)$.  Given a tangent vector $v = v_1 \parder{r} + v_2 \parder{s} + v_3 \parder{t}$, $\metric(v,v) = v_1 v_3 - v_2{}^2$ and this is positive / zero / negative iff $v$ lies inside [on, outside] the null cone.  This is well-defined on the conformal class $[\metric]$. 

 \subsection{Elliptic, parabolic, and hyperbolic surfaces}
 \label{sec:EPH}
 
 \begin{defn}
 Let $M \subset \Lg$ be a surface.  We say $M$ is:
 \begin{enumerate}
 \item elliptic (spacelike)  if for any $x \in M$, $T_x M \cap \mathcal{C}_x$ is a single point; equivalently, $N_xM$ is timelike. 
 \item parabolic (null)  if for any $x \in M$, $T_x M \cap \mathcal{C}_x$ consists of a single line; equivalently, $N_xM$ is null.
 \item hyperbolic (timelike) if for any $x \in M$, $T_x M \cap \mathcal{C}_x$ consists of two distinct lines; equivalently, $N_xM$ is spacelike.
 \end{enumerate}
 \end{defn}

 
\begin{figure}[h]
\begin{center}
\includegraphics[scale=0.4]{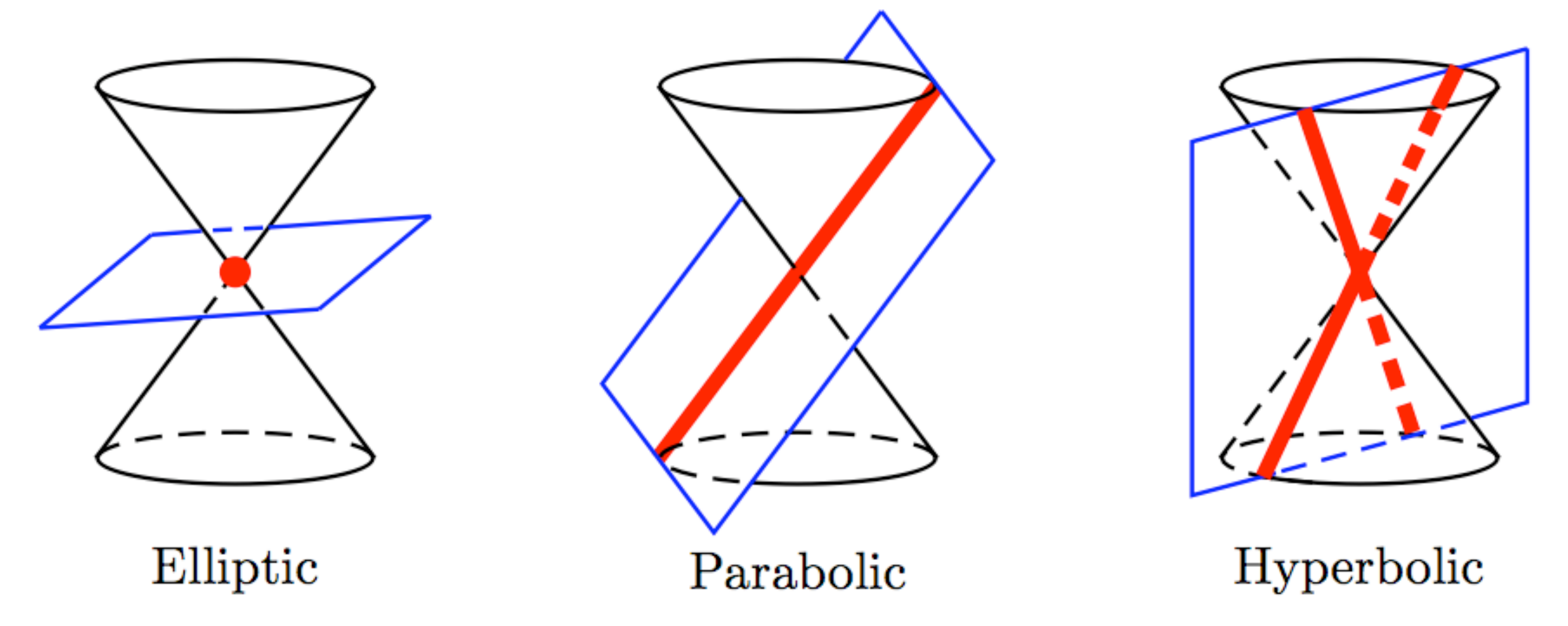}
\end{center}
 \label{fig:intersection}
 \caption{Intersection types}
 \end{figure}

 In the chart $\psi$, suppose $M$ has equation $F(r,s,t) = 0$. On $M$, $0 = dF = F_r dr + F_s ds + F_t dt$.  Assuming $(F_r,F_s,F_t) \neq 0$ on $M$, then $T_x M$ for any $x \in M$ is spanned by the (linearly dependent) vectors $F_r \partial_s - F_s \partial_r$, $F_t \partial_r - F_r \partial_t$, $F_t \partial_s - F_s \partial_t$.  The $\metric$-orthogonal complement of $T_x M$ yields the normal space $N_x M$ spanned by $n = F_t \parder{r} - \frac{1}{2} F_s \parder{s} + F_r \parder{t}$ and we have $\metric(n, n) = F_r F_t - \frac{1}{4} F_s{}^2$, whose sign is well-defined on $[\metric]$.  Hence, $M$ is elliptic / parabolic / hyperbolic iff at each point of $M$, $F_r F_t - \frac{1}{4} F_s{}^2$ is positive / zero / negative.  
 
 \begin{example}
 Let $c\neq 0$ be a constant.  The surface $rt = c$ is elliptic if $c > 0$ and hyperbolic if $c < 0$.  Using the rescaling $(r,s,t) \ra (\lambda r, \lambda s, \lambda t)$ which is a $CSp(4,\R)$ transformation, all are equivalent to $rt = 1$ or $rt = -1$.
 \end{example}

 The canonical cone field $\mathcal{C}$ on $\Lg$ transfers to each fibre $J^2|_\jet$ in the PDE setting.  Alternatively, the sign of $\metric(n,n)$ is a discrete $CSp(4,\R)$-invariant on $M$, which by Theorem \ref{thm:symplectic-contact} yields a contact-invariant for PDE.  Surprisingly, even the following definition has not appeared in the literature.

 \begin{defn}
 $\Sigma = \{ F=0 \}\subset J^2$ is elliptic / parabolic / hyperbolic at $\tilde\jet \in J^2$ if for $\jet = \pi^2_1(\tilde\jet)$,  $T_{\tilde\jet} (\Sigma |_\jet)$ intersects the canonical cone ${\cal C}_{\tilde\jet} \subset T_{\tilde\jet} (J^2|_\jet)$ in one point / along one line / along two distinct lines.  Equivalently, the normal line to $T_{\tilde\jet} (\Sigma |_\jet) \subset T_{\tilde\jet} (J^2|_\jet)$ lies inside / on / outside ${\cal C}_{\tilde\jet}$.
 \end{defn}
 
 Thus, from the $LG$ perspective, ``elliptic / parabolic / hyperbolic'' is really a first order consequence of the existence of the cone field $\mathcal{C}$, which itself is a manifestation of the special isomorphism $Sp(4,\R) \cong Spin(2,3)$.  In standard $J^2$ coordinates $(x,y,z,p,q,r,s,t)$, the sign of $F_r F_t - \frac{1}{4} F_s{}^2$ evaluated pointwise on $\Sigma$ determines the PDE type.  We recover the usual invariant distinguishing elliptic, parabolic, hyperbolic PDE \cite{Gardner1967}.

 \subsection{Spheres in $\Lg$}
 \label{sec:spheres}
 
 Using $\ambient{\cdot}{\cdot}$, there is a bijective (polar) correspondence between points and hyperplanes in $\P V$.
  
 \begin{defn} \label{defn:sphere}
 For any $[z] \in \P V$, we refer to $\S{z} = \P(z^\perp) \cap \Q$ as a {\em sphere}.
 \end{defn}
 
 \begin{rem}
 We caution the reader that spheres as we have defined above are topologically different from the usual spheres in Euclidean geometry.  Rather, a sphere above is the intersection of a hyperplane with $\Q$.  This terminology is borrowed from classical conformal geometry in definite signature.
 \end{rem}
 
 Orthogonality characterizes incidence with $\S{z}$: if $[w] \in \Q$, then $[w] \in \S{z}$ iff $\ambient{w}{z}=0$.
 Hence, spheres are given by {\em linear} equations.  Let 
 $\mathcal{H}_+ = \{ [z] \in \P V : \ambient{z}{z} > 0 \}$ and $\mathcal{H}_- = \{ [z] \in \P V : \ambient{z}{z} < 0 \}$.
 Then $\P V = \mathcal{H}_+ \dot\cup\, \mathcal{H}_- \dot\cup\, \Q$.  The spheres determined by elements in $\mathcal{H}_+, \mathcal{H}_-, \Q$ are referred to as {\em definite, indefinite, degenerate} respectively.  Note that any $[z] \in \Q$ plays a dual role as: (i) a {\em point} in $\Q$, and (ii) a {\em sphere} in $\Q$ with {\em vertex} (singularity) at $[z]$. 
 
 \begin{lem}  \label{lem:sp-transitive} {\mbox{ }} \begin{enumerate}
 \item $Sp(4,\R)$ acts transitively on each of $\mathcal{H}_+, \mathcal{H}_-, \Q$.
 \item Let $\ambient{z}{z} < 0$.  The stabilizer in $Sp(4,\R)$ of $[z]$ acts transitively on $\S{z}$.
 \end{enumerate}
 \end{lem}
 
 \begin{proof}
 Exercise for the reader.
 \end{proof}

 Using Lemma \ref{lem:sp-transitive}, and by examining points in $\mathcal{H}_+, \mathcal{H}_-, \Q$, the corresponding spheres are topologically $\R\P^2$, $(S^1 \times S^1)/\Z_2$, and a pinched torus respectively.  However, this will not play an essential role in the sequel.  More relevant for us is the local picture.
  Let us describe general spheres in the chart $\psi$.  With respect to the basis $\mathcal{B}$ in Section \ref{sec:quadric}, fix $z = (z_0,z_1,z_2,z_3,z_4) \in V$ and note $\ambient{z}{z} = 2(z_1 z_3 - z_2{}^2 - z_4 z_0)$.  Then
 \begin{align}
  \psi(\S{z}) &= \{ (r,s,t) : -z_4 + rz_3 -2s z_2 + tz_1-z_0(rt-s^2) = 0\}.
  \label{spheres}
 \end{align}
 If $z_0=0$, then $\psi(\S{z})$ is a plane, while if $z_0 \neq 0$, then $\psi(\S{z}) = \l\{ (r,s,t) : \l(r-\frac{z_1}{z_0}\r)\l(t-\frac{z_3}{z_0}\r)-\l(s- \frac{z_2}{z_0}\r)^2 = \frac{\ambient{z}{z}}{2z_0{}^2}\r\}$ is a hyperboloid of 2-sheets, a hyperboloid of 1-sheet, or a cone iff $\ambient{z}{z}$ is positive, negative, or zero, respectively.
 
 
 \begin{prop}
 Let $[z] \in \mathcal{H}_-$.  The indefinite sphere $\S{z}$ is doubly-ruled by null geodesics.
 \end{prop}
 
 \begin{proof}
 For any $[w] \in \S{z}$, $\widehat{T}_{[w]} \S{z} = w^\perp \cap z^\perp$ has signature $(+-0)$ since $\ambient{z}{z} < 0$ and $\ambient{\cdot}{\cdot}$ has signature $(2,3)$.  Hence, we can choose a null basis $\{ u_1, u_2, w \}$ of $\widehat{T}_{[w]} \S{z}$.  Let $u$ be either $u_1$ or $u_2$.  Then $u$ correponds to a null direction in the tangent space $T_{[w]} \S{z}$.  Also, $\ambient{u}{u} = \ambient{u}{w} = \ambient{u}{z} = 0$.  For any $a,b \in \R$ (not both zero), $\ambient{au+bw}{au+bw} = 0$, so $[au+bw] \in \Q$ and $\ambient{au+bw}{z}=0$, so $[au+bw]$ is a null projective line contained in $\S{z}$, which is a null line in the chart $\psi$.   Since the conformal structure $[\metric]$ is flat, null lines are null geodesics.
 \end{proof} 

 For our later study of hyperbolic surfaces in $\Lg$, consider the basis
 \begin{align}
 \mathcal{B}_H:\quad \begin{array}{l} e_1 \wedge e_2, \\ e_3 \wedge e_2, \\ e_1 \wedge e_4, \\ e_1 \wedge e_3 - e_2 \wedge e_4, \\ e_3 \wedge e_4 \end{array} \qRa
 \ambient{\cdot}{\cdot}_{\mathcal{B}_H} =   \mat{ccccc}{ 0 & 0 & 0 & 0 & -1\\ 0 & 0 & 1 & 0 & 0\\ 0 & 1 & 0 & 0 & 0\\ 0 & 0 & 0 & -2 & 0\\ -1 & 0 & 0 & 0 & 0 },
 \label{B-hyp-scalar-prod}
 \end{align}
 which is $\mathcal{B}$ \eqref{B-basis-scalar-prod} with third and fourth basis vectors interchanged.
 Any basis $\mf{} = (\mf{0}, ..., \mf{4})$ of $V$ with scalar products $\ambient{\mf{i}}{\mf{j}}$ given by the matrix in \eqref{B-hyp-scalar-prod} will be called a {\em hyperbolic frame}.  Any such frame has a natural geometric interpretation as a 5-tuple of spheres: $([\mf{0}], [\mf{1}], [\mf{2}], [\mf{4}])$ forms a {\em null diamond} inscribed in the indefinite sphere $\S{\mf{3}}$.  The picture given in Figure \ref{fig:null-diamond} is a sample configuration, but should not be misleading: e.g. one or more of the points $[\mf{i}]$, $i=1,2,3,4$, may lie at on the sphere at infinity.  (If $[\mf{3}]$ is, then the hyperboloid degenerates to a plane.)
 
%
 
 \begin{figure}[h]
 \begin{center}
 \includegraphics[scale=0.4]{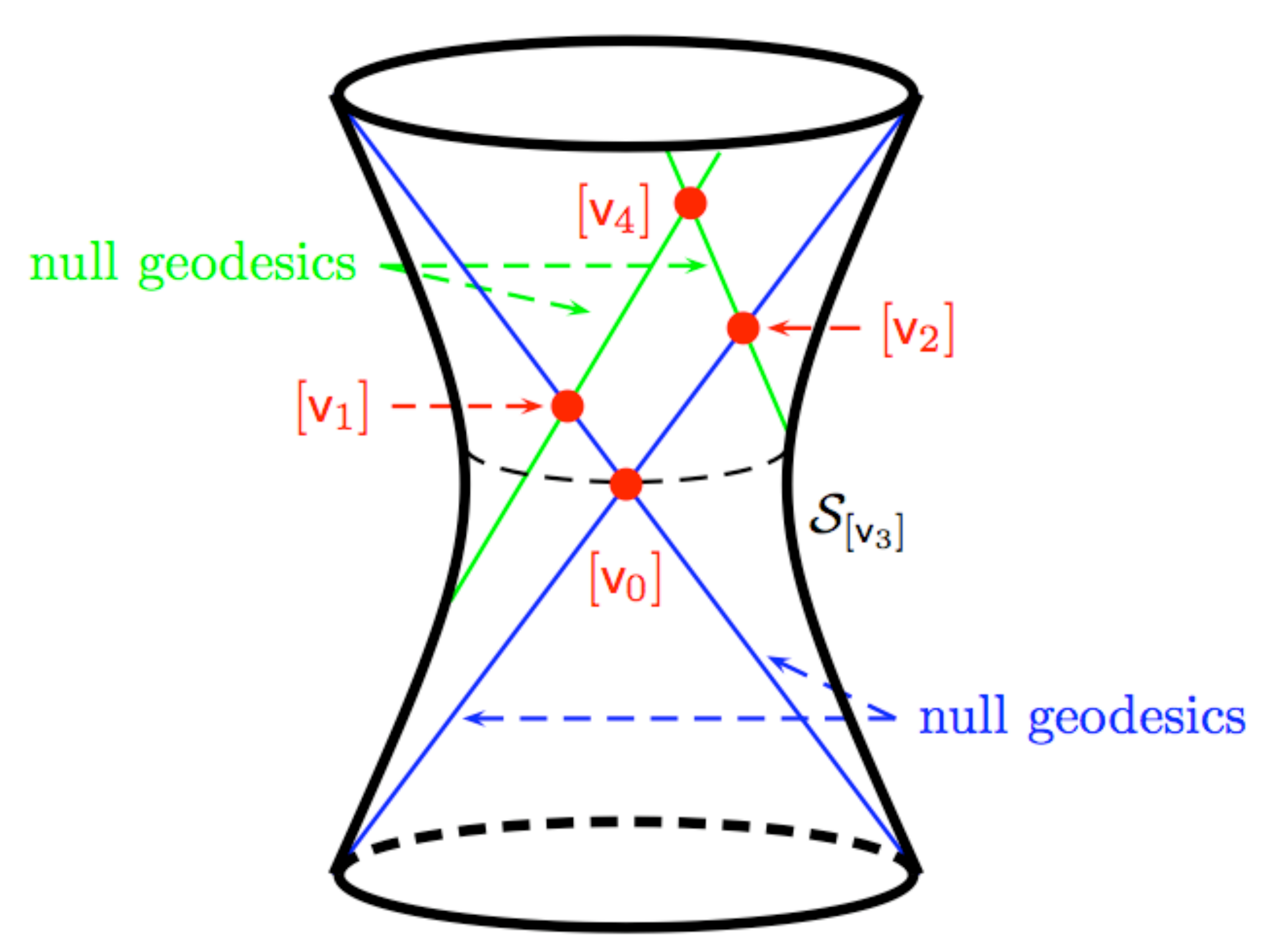}
 \caption{Geometric interpretation of a hyperbolic frame}
 \label{fig:null-diamond}
 \end{center}
 \end{figure}
 
 Let $\mathcal{P}_x \in \Lg$ correspond to $[x] \in \Q \subset \P V$ under the Pl\"ucker map.  The following is obvious:
 \begin{lem}
 Let $[x],[y] \in \Q$.  Then $\ambient{x}{y} \neq 0$ iff $\R^4 = \mathcal{P}_x \oplus \mathcal{P}_y$ is a Lagrangian decomposition.
Thus, if $[x] \neq [y]$ and $\ambient{x}{y} = 0$, then $\dim(\mathcal{P}_x \cap \mathcal{P}_y) = 1$.
 \end{lem}

 Thus, given a basis $\mf{}$ of $V$ as above, we can interpret it in the symplectic setting.  Corresponding to $\mf{}$, we have
 \begin{itemize}
 \item  we have a pair of Lagrangian decompositions $\R^4 = \mathcal{P}_{[\mf{0}]} \oplus \mathcal{P}_{[\mf{4}]} $ and $\R^4 = \mathcal{P}_{[\mf{1}]} \oplus \mathcal{P}_{[\mf{3}]} $.
 \item $\mathcal{P}_{[\mf{0}]} = (\mathcal{P}_{[\mf{0}]} \cap \mathcal{P}_{[\mf{1}]}) \oplus (\mathcal{P}_{[\mf{0}]} \cap \mathcal{P}_{[\mf{3}]})$ and $\mathcal{P}_{[\mf{4}]} = (\mathcal{P}_{[\mf{4}]} \cap \mathcal{P}_{[\mf{3}]}) \oplus (\mathcal{P}_{[\mf{4}]} \cap \mathcal{P}_{[\mf{1}]})$.
 \end{itemize}

 Finally, {\em inversion} with respect to an indefinite or definite sphere $\S{z}$ is defined in a natural manner: Let $[w] \in \Q$.  Its inversion $[\tilde{w}]$ with respect to $\S{z}$ is $\tilde{w} = w - \frac{2\ambient{w}{z}}{\ambient{z}{z}} z$.  If $[\tilde{w}] = [w]$, we must have $\tilde{w} = w$ and $\ambient{w}{z} = 0$, i.e. $[w] \in \S{z}$. Inversion with respect to a degenerate sphere is undefined.
  
%

 \subsection{The \MC/ form}
 \label{sec:MC}
 
 Recall the effective symmetry group of $\Q = \Lg$ is $O^+(2,3)$.  Taking the standard symplectic basis $\{ e_i \}_{i=1}^4$ on $\R^4$ and the hyperbolic frame $\mathcal{B}_H$ on $V$, c.f. \eqref{B-hyp-scalar-prod}, the map $\phi = (\Phi_*)_e$ arising from $\Phi : \widehat{Sp}(4,\R) \ra O^+(2,3)$ is
 \[
 \phi \mat{cccc}{ a_{11} & a_{12} & b_1 & b_3 \\ a_{21} & a_{22} & b_3 & b_2 \\ c_1 & c_3 & -a_{11} & -a_{21} \\ c_3 & c_2 & -a_{12} & -a_{22} } 
 = \mat{ccccc}{ a_{11} + a_{22} & b_1 & b_2 & 2b_3 & 0\\ c_1 & a_{22} - a_{11} & 0 & -2a_{21} & b_2\\ c_2 & 0 & a_{11}-a_{22} & -2a_{12} & b_1\\ c_3 & -a_{12} & -a_{21} & 0 & -b_3 \\0 & c_2 & c_1 & -2c_3 & -(a_{11} + a_{22})},
 \]
 Denoting the components of the (left-invariant) MC form $\omega = g^{-1} dg$ on $O^+(2,3)$ as
 \begin{align}
 \omega =  \mat{ccccc}{ \alpha_{11} + \alpha_{22} & \beta_1 & \beta_2 & 2\beta_3 & 0\\ \theta_1 & \alpha_{22} - \alpha_{11} & 0 & -2\alpha_{21} & \beta_2\\ \theta_2 & 0 & \alpha_{11}-\alpha_{22} & -2\alpha_{12} & \beta_1\\ \theta_3 & -\alpha_{12} & -\alpha_{21} & 0 & -\beta_3 \\0 & \theta_2 & \theta_1 & -2\theta_3 & -(\alpha_{11} + \alpha_{22})},
 \label{MC-SO-comps}
 \end{align}
 the MC structure equations $d\omega + \omega \wedge \omega = 0$ are
 \renewcommand\arraystretch{1.2}
 \begin{align}
 \begin{array}{ll}
 0 = d\alpha_{11} + \alpha_{12} \wedge \alpha_{21} + \beta_1 \wedge \theta_1 + \beta_3 \wedge \theta_3, &
	0 = d\alpha_{12} + (\alpha_{11} - \alpha_{22}) \wedge \alpha_{12} + \beta_1 \wedge \theta_3 + \beta_3 \wedge \theta_2, \\
 0 = d\alpha_{22} + \alpha_{21} \wedge \alpha_{12} + \beta_3 \wedge \theta_3 + \beta_2 \wedge \theta_2, &
	0 = d\alpha_{21} + \alpha_{21} \wedge (\alpha_{11} - \alpha_{22}) + \beta_3 \wedge \theta_1 + \beta_2 \wedge \theta_3, \\
 0 = d\beta_1 + 2\alpha_{11} \wedge \beta_1 + 2\alpha_{12} \wedge \beta_3, &
	0 = d\theta_1 + 2\theta_1 \wedge \alpha_{11} + 2\theta_3 \wedge \alpha_{21}, \\
 0 = d\beta_2 + 2\alpha_{22} \wedge \beta_2 + 2\alpha_{21} \wedge \beta_3, &
	0 = d\theta_2 + 2\theta_2 \wedge \alpha_{22} + 2\theta_3 \wedge \alpha_{12}, \\
 0 = d\beta_3 + (\alpha_{11} + \alpha_{22}) \wedge \beta_3 + \alpha_{12} \wedge \beta_2 + \alpha_{21} \wedge \beta_1, &
 0 = d\theta_3 + \theta_1 \wedge \alpha_{12} + \theta_3 \wedge (\alpha_{11} + \alpha_{22}) + \theta_2 \wedge \alpha_{21}.
 \end{array} \label{MC-eqns}
 \end{align}
 \renewcommand\arraystretch{1.1}

 With respect to $O^+(2,3) \ra O^+(2,3) / P^+ = \Q$, the forms $\theta_1,\theta_2,\theta_3$ are semi-basic and $\tilde\metric = \theta_1 \odot \theta_2 - \theta_3 \odot \theta_3$ has $P^+$-invariant conformal class $[\tilde\metric]$.  Hence, $[\tilde\metric]$ is the pullback by the projection of a conformal class $[\metric]$ on $\Lg$.  This is a fourth description of the canonical (up to sign) conformal structure on $\Lg$.
  
 \subsection{Moving frames for surfaces in $\Lg$}
 \label{sec:mf-general}
 
 Identify $O^+(2,3)$ with its orbit through a chosen basis $\mf{}^{(0)}$ of $V$.  (Later, we will take $\mf{}^{(0)} = \mathcal{B}_H$ to study hyperbolic surfaces.)  For any $\mf{} = (\mf{0},\mf{1},\mf{2},\mf{3},\mf{4}) \in O^+(2,3)$, we have $O^+(2,3) \ra \Q$, $\mf{} \mapsto [\mf{0}]$.
 Given any (embedded) surface $i : M \inj \Q  \subset \P V$, the zeroth order frame bundle $\F_0(M) \ra M$  is the pullback of $O^+(2,3) \ra \Q$ by $i$, and the diagram
 \[
 \xymatrix{\F_0(M) \ar[r]^{\tilde{i}}\ar[d]_p & O^+(2,3)\ar[d]\\ M \ar[r]^i & \Q}
 \]
 commutes.  A {\em moving frame} $\mf{}$ for $M$ is a local section of $\F_0(M) \ra M$, i.e. we always have $[\mf{0}] \in M$.  Let $\tilde{\mf{}} = \tilde{i} \circ \mf{}$.

 \begin{rem}
 Notationally, we do not distinguish a moving frame $\mf{}$ from a particular frame $\mf{} \in \F_0(M)$.  One should infer the former if derivatives, pullbacks, etc. are taken and the latter when defining adapted frame bundles.
 \end{rem}

 Since $\mf{} = \mf{}^{(0)} g$ for some (local) map $g: M \ra O^+(2,3)$, then we have the moving frame structure equations $d\mf{} = \mf{}^{(0)} dg = \mf{} g^{-1} dg = \mf{} \soMC{}{}$, or in components, $ d\mf{j} = \soMC{i}{j} \mf{i}$, where the row / column index of $\soMC{}{}$ are on top / bottom respectively.  Here, $\soMC{}{}$ is the MC form of $O^+(2,3)$ given in \eqref{MC-SO-comps}, or more precisely, its pullback to $M$ by $\tilde{\mf{}}$.
  Following the usual practice with moving frames, we will not notationally distinguish the MC form with its pullback.  The integrability conditions for $\soMC{}{}$ are the MC structure equations $d\omega + \omega \wedge \omega = 0$ given in \eqref{MC-eqns}.

 The structure group $H_0$ of $\F_0(M) \ra M$ is $P^+ = \Phi(P) \subset O^+(2,3)$.  Any two moving frames $\mf{},\mf{}'$ satisfy $\mf{}' = \mf{} h$ for some (local) map $h : M \ra H_0$.  The change of frame formula is given by
 \begin{align}
 \tilde{\mf{}}'^*\omega = h^{-1} (\tilde{\mf{}}^*\omega) h + h^{-1} dh.
 \label{MC-frame-change}
 \end{align}
 
More generally, let $G$ be a Lie group with Lie algebra $\g$ and \MC/ form $\omega_G$.  The following theorems \cite{Jensen1977} provide the theoretical basis for Cartan's method of moving frames.

 \begin{thm} \label{thm:Cartan-equiv}
 Let $f_1, f_2 : M \ra G$ be two smooth maps of a connected manifold $M$ into a Lie group $G$.  Then $f_1(x) = g \cdot f_2(x)$ for all $x \in M$ and for a fixed $g \in G$ iff $f_1^* \omega_G = f_2^*\omega_G$.
 \end{thm}
 
 \begin{thm} \label{thm:existence}
  Let $M$ be a smooth manifold, $\psi \in \Omega^1(M,\g)$ a $\g$-valued 1-form.  Then for any $x \in M$, there exists an open neighbourhood $U$ of $x \in M$ and a function $f : U \ra G$ such that $f^* \omega_G = \psi$ iff  $\omega$ satisfies $d\psi + \psi \wedge \psi = 0$.
 \end{thm}
 
 Our application of the method of moving frames will proceed by finding geometrically adapted sections of $\F_0(M) \ra M$, or equivalently by composing with the map $\F_0(M) \ra O^+(2,3)$, adapted lifts of $M$ to $G = O^+(2,3)$.  Higher order frames bundle will be defined and generally denoted $\F_k(M) \ra M$ with structure group $H_k \subset H_0$.

 \section{Moving frames for hyperbolic surfaces}
 \label{sec:hyp-mf}
 
 We begin our study of hyperbolic surfaces in $\Q = \Lg$ via moving frames.  While the main text of this article contains an account of our moving frames study in an abstract form, we encourage the reader to concurrently read Appendix \ref{app:param} which contains parallel calculations in parametric form.  For the remainder of this paper, $i : U \ra \Q$ is a smooth embedded hyperbolic surface and $M = i(U)$.  In defining $\F_0(M)$ as in Section \ref{sec:mf-general}, we choose the initial basis $\mf{}^{(0)} = \mathcal{B}_H$, c.f. \eqref{B-hyp-scalar-prod}.  Thus, any $\mf{} \in \F_0(M)$ is a hyperbolic frame.
 
 \subsection{1-adaptation}
 
  Let $[\mf{0}] \in M$.  The affine tangent space $\widehat{T}_{[\mf{0}]} M$ is the tangent space (translated to the origin) to the cone $\widehat{M} \subset V$ over $M$ at any (nonzero) point of $\ell_{\mf{0}} =  span\{ \mf{0} \}$.  We have $\mf{0} \in \widehat{T}_{[\mf{0}]} M \subset \mf{0}^\perp = span\{ \mf{0}, \mf{1}, \mf{2}, \mf{3} \}$ for any $\mf{} \in \F_0(M)$.  By hyperbolicity of $M$, there is: (i) a complementary normal line to $T_{[\mf{0}]} M$ in $T_{[\mf{0}]} \Q$, and (ii) two distinguished null lines in $T_{[\mf{0}]}M$.   These data lift to $\widehat{T}_{[\mf{0}]} M$ to give the flags of subspaces:
   \begin{align*}
 &\ell_{\mf{0}} \subset n_i|_{\mf{0}} \subset \widehat{T}_{[\mf{0}]} M \subset \mf{0}^\perp \subset V, \qquad \dim(n_i|_{\mf{0}}) =2,  \qquad n_1|_{\mf{0}}  \cap n_2|_{\mf{0}}  = \ell_{\mf{0}} \\
  & \ell_{\mf{0}} \subset N|_{\mf{0}}  \subset \mf{0}^\perp \subset V, \qquad \dim(N|_{\mf{0}} ) = 2 ,\qquad N|_{\mf{0}}  \cap \widehat{T}_{[\mf{0}]} M = \ell_{\mf{0}}.
 \end{align*}
 Define 
 $\F_1(M) = \{ \mf{} \in \F_0(M) : (\mf{1}, \mf{2}) \in (n_1|_{\mf{0}}  \times n_2|_{\mf{0}} ) \cup (n_2|_{\mf{0}}  \times n_1|_{\mf{0}} ), \, \mf{3} \in N|_{\mf{0}}  \}$ and 
 $H_1 = H_1' \rtimes (\Z_2 \times \Z_2)$, where
 \begin{align*}
  H'_1 &= \l\{ X(r_1,r_2,s_1,s_2,s_3) = \mat{ccccc}{r_1 & s_1 & s_2 & 2 s_3 & \frac{s_1 s_2 - s_3{}^2}{r_1} \\ 0 & r_2 & 0 & 0 & \frac{r_2 s_2}{r_1}\\ 0 & 0 & \frac{1}{r_2} & 0 & \frac{s_1}{r_1r_2}\\ 0 & 0 & 0 & 1 & -\frac{s_3}{r_1} \\ 0 &0 & 0 &0 & \frac{1}{r_1}} : r_1 r_2 > 0,  s_i \in \R \r\}
  \end{align*}
  and $\Z_2  \times \Z_2$ has generators $R_1 = diag\l(-1,\mat{cc}{0 & -1\\ -1 & 0}, -1,-1\r)$ and $R_2 = diag(1,-1,-1,-1,1)$.
Thus, for any $\mf{} \in \F_1(M)$, $\widehat{T}_{[\mf{0}]} M = span\{ \mf{0}, \mf{1}, \mf{2} \}$. 
 Comparing with \eqref{MC-SO-comps}, any 1-adapted moving frame $\mf{}$ has $\theta_3 = 0$ and $\{ \theta_1,\theta_2 \}$ is a local coframing on $T^*M$.  Explicitly, the 1-adapted moving frame structure equations $d\mf{j} = \soMC{i}{j} \mf{i}$ are:
 \begin{align}
 d\mf{0} &= (\alpha_{11} + \alpha_{22}) \mf{0} + \theta_1 \mf{1} + \theta_2 \mf{2} \nonumber\\
 d\mf{1} &= \beta_1 \mf{0} + (\alpha_{22} - \alpha_{11}) \mf{1} - \alpha_{12} \mf{3} + \theta_2 \mf{4} \nonumber\\
 d\mf{2} &= \beta_2 \mf{0} + (\alpha_{11} - \alpha_{22}) \mf{2} - \alpha_{21}  \mf{3} + \theta_1 \mf{4} \label{1-str-eqs}\\
 d\mf{3} &= 2\beta_3  \mf{0} - 2\alpha_{21}  \mf{1} - 2\alpha_{12}  \mf{2} \nonumber\\
 d\mf{4} &= \beta_2 \mf{1} +\beta_1 \mf{2} - \beta_3 \mf{3} - (\alpha_{11} + \alpha_{22}) \mf{4}\nonumber
 \end{align}
 The components of $\soMC{}{}$: (i) can be calculated in terms of the moving frame, e.g. $\theta_1 = \ambient{d\mf{0}}{\mf{2}} = -\ambient{d\mf{2}}{\mf{0}}$, and (ii) satisfy the MC equations \eqref{MC-eqns}.
 
 \subsection{2-adaptation: central sphere congruence}
 \label{sec:2-adaptation}
  
 By the $d\theta_3$ MC equation \eqref{MC-eqns},  $0 = \theta_1 \wedge \alpha_{12} + \theta_2 \wedge \alpha_{21}$.  Hence, Cartan's lemma implies there are second order functions $\lambda_{ij}$ on $M$ so that:
 \begin{align}
 \alpha_{12} = \lambda_{11} \theta_1 + \lambda_{12} \theta_2, \qquad \alpha_{21} = \lambda_{12} \theta_1 + \lambda_{22} \theta_2.
 \end{align}
 Since $\ambient{d\mf{1}}{\mf{3}} = 2\alpha_{12} = 2\lambda_{11} \theta_1$, $\ambient{d\mf{2}}{\mf{3}} = 2\alpha_{21} = 2\lambda_{22} \theta_2$, then $\lambda_{ij}$ are the components of the second fundamental form $II$ of $(M,\metric)$.  However, the conformal change $\widehat\metric = e^{2u} \metric$, induces $\widehat{II} = II - \metric [\grad_\metric(u)]^\perp$, so $II$ is not well-defined on the conformal class $[\metric]$.  Letting $H = \frac{1}{2} \textsf{tr}_\metric (II)$ be the mean curvature, the trace-frace second fundamental form $\trfr_\metric(II) = II - H\metric$ is well-defined on $[\metric]$, c.f. \cite{HJ2003}.  Indeed, $\metric_{ij} = \mat{cc}{0 & 1\\ 1 & 0}$ are the components of $\ambient{\cdot}{\cdot}$ with respect to the basis $\{ \mf{1}, \mf{2} \} \mod \mf{0}$, and $\trfr_\metric(II)$ has components
$a_{ij} = \lambda_{ij} - H \metric_{ij}$, where $H = \frac{1}{2}\metric^{k\ell} \lambda_{k\ell} = \lambda_{12}$ (with $\metric^{ij}$ the inverse of $\metric_{ij}$). Thus, $\metric^{ij} a_{ij} = 0$ and $a_{ij}$ only has diagonal components $diag(\lambda_{11}, \lambda_{22})$.
  The eigenvalues of $a_i{}^k = a_{ij} \metric^{jk}$ are the roots of $0 = det(a_{ij} - s \metric_{ij}) = \lambda_{11}\lambda_{22} - s^2$.  
 
  \begin{defn}
 Let $\mf{}$ be any 1-adapted moving frame on $M$.  Let $\mf{c} = \mf{3} + 2\lambda_{12} \mf{0}$.
 \begin{enumerate}
 \item The map $\gamma : M \ra \mathcal{H}_-$, $p \mapsto [\mf{c}|_p]$ is called the {\em conformal Gauss map}.
 \item Pointwise, $\S{\mf{c}} = \P(\mf{c}^\perp) \cap \Q$ is the {\em central tangent sphere} (CTS).  The collection of CTS over $M$, i.e.\ the image of $\gamma$, is called the {\em central sphere congruence} (CSG) on $M$.
 \item The 2-adapted frame bundle of $M$ is $\F_2(M) = \{ \mf{}\in \F_1(M) : [\mf{3}] = [\mf{c}] \}$ with structure group $H_2 = H_2' \rtimes (\Z_2 \times \Z_2)$ with $H_2' = \{ X(r_1,r_2,s_1,s_2,0) \in H_1' \}$.
 \end{enumerate}
 \end{defn}
 
 \begin{align}
 \begin{array}{|c|c|c|c|c|c|c|c|c|}\hline
 & \bar\alpha_{12} & \bar\alpha_{21} & \bar\theta_1 & \bar\theta_2 & \bar\lambda_{12} & \bar\lambda_{11} & \bar\lambda_{22} & \bar{\mf{}}_c\\ \hline
 H_1' & r_2 (\alpha_{12} - s_3\theta_2) & \frac{1}{r_2} (\alpha_{21} - s_3\theta_1) & \frac{r_1}{r_2} \theta_1 &  r_1 r_2 \theta_2 & \frac{\lambda_{12} - s_3}{r_1} & \frac{r_2{}^2}{r_1} \lambda_{11} & \frac{1}{r_1 r_2{}^2} \lambda_{22} & \mf{c}\\
 R_1 & \alpha_{21} & \alpha_{12} & \theta_2 & \theta_1 & \lambda_{12} & \lambda_{22} & \lambda_{11} & -\mf{c}\\
 R_2 & \alpha_{12} & \alpha_{21} & -\theta_1 & -\theta_2 & -\lambda_{12} & -\lambda_{11} & -\lambda_{22} & -\mf{c}\\ \hline
 \end{array}
 \label{H1-transform}
 \end{align}
 
 From \eqref{H1-transform}, $\S{\mf{c}}$ is indeed $H_1$-invariant, i.e.\ the CSG is well-defined, and moreover $\lambda_{11}, \lambda_{22}$ are relative invariants.
 Since $\ambient{\mf{3}}{\mf{3}} = -2$, then if $\mf{}$ is 2-adapted, we must have $\mf{3} = \pm \mf{c}$ and $\lambda_{12} = 0$.  Each $\S{\mf{3}}$ is incident with $M$ at $[\mf{0}]$ and since $\ambient{d\mf{0}}{\mf{3}}=0$, then $T_{[\mf{0}]} M$ is also tangent to $\S{\mf{3}}$.
Equivalently, $M$ is an envelope for its CSG.

  Define $(D^2 \mf{0})(w_1,w_2) = w_1 \intprod d(w_2 \intprod d\mf{0})$, where $w_1,w_2$ are vector fields on $M$.  We can compute $\trfr_\metric(II)$ via
 \begin{align}
 \ambient{D^2 \mf{0}(w_1,w_2)}{\mf{3}} &= \ambient{w_1 \intprod ((\alpha_{11} + \alpha_{22})(w_2) d\mf{0} + \theta_1(w_2) d\mf{1} + \theta_2(w_2) d\mf{2})}{\mf{3}} \nonumber\\
  &= 2(\alpha_{12}(w_1) \theta_1(w_2) + \alpha_{21}(w_1) \theta_2(w_2)) \nonumber\\
  &= 2(\lambda_{11} \theta_1(w_1)\theta_1(w_2) + \lambda_{22} \theta_2(w_1)\theta_2(w_2) ) 
 \label{CTS-deviation}
 \end{align}
 Thus, $\ambient{D^2 \mf{0}(\cdot,\cdot)}{\mf{3}}$ is the symmetric tensor $2(\lambda_{11} \theta_1{}^2 + \lambda_{22} \theta_2{}^2)$.  Since $\mf{0}$ may be rescaled and $\mf{3}$ can change sign, then $\ambient{D^2 \mf{0}(\cdot,\cdot)}{\mf{3}}$ is only well-defined up to scale.  A similar calculation as above yields
 \begin{align}
 \ambient{D^2 \mf{0}(\cdot,\cdot)}{\mf{3} + 2s\mf{0}} 
 = 2(\lambda_{11} \theta_1{}^2 + \lambda_{22} \theta_2{}^2 - s (\theta_1 \otimes \theta_2 + \theta_2 \otimes \theta_1) )
 = 2(\theta_1, \theta_2) \mat{cc}{ \lambda_{11} & -s\\ -s &\lambda_{22}}\mat{c}{\theta_1\\ \theta_2}
  \label{curvature-lines}
 \end{align}
 The directions along which \eqref{curvature-lines} vanishes determine directions of second-order tangency of $C_n$ with $M$.
 
 \begin{defn}
 {\em Asymptotic directions} are those tangent vectors $w$ for which $\ambient{D^2 \mf{0}(w,w)}{\mf{3}} = 0$.  {\em Principal directions} are eigendirections of $a^i_j = \metric^{ik} a_{kj}$.  Integral curves of principal directions are called {\em curvature lines}.
 \end{defn}
 
 \begin{defn} Define $\epsilon = \sgn(\lambda_{11} \lambda_{22})$.
 $M$ is {\em 2-generic} if it is 2-hyperbolic $(\epsilon = -1)$ or 2-elliptic $(\epsilon = 1)$.
 \end{defn}

 Along asymptotic directions, $\S{\mf{c}}$ has second order contact with $M$. 
 Let us observe that from \eqref{1-str-eqs}, we have
 \begin{align*}
 \theta_2 = 0 \qRa d\mf{0} \equiv 0, \quad
 d\mf{1} \equiv -\alpha_{12} \mf{3} = -\lambda_{11} \theta_1 \mf{3} \qquad \mod\{ \mf{0}, \mf{1} \}\\
 \theta_1 = 0 \qRa d\mf{0} \equiv 0, \quad
 d\mf{2} \equiv -\alpha_{12} \mf{3} = -\lambda_{22} \theta_2 \mf{3} \qquad \mod\{ \mf{0}, \mf{2} \}
 \end{align*}
Recalling that $\lambda_{11},\lambda_{22}$ are relative invariants, then by \eqref{CTS-deviation}, we have a geometric interpretation for $\lambda_{11}, \lambda_{22}$:
 \begin{itemize}
 \item $\lambda_{11} = 0$ iff any curve satisfying $\theta_2 = \theta_3 = 0$ is a null geodesic
 \item $\lambda_{22} = 0$ iff any curve satisfying $\theta_1 = \theta_3 = 0$ is a null geodesic
 \end{itemize}
 
  \begin{table}[h]
 \begin{center}
 \begin{tabular}{|c|c|c|} \hline
 Terminology & Invariant characterization & Geometric interpretation\\ \hline\hline
 2-isotropic & $\lambda_{11} = \lambda_{22} = 0$ & $M$ is doubly-ruled by null geodesics\\ \hline
 2-parabolic & Exactly one of $\lambda_{11}$ or $\lambda_{22}$ is zero & $M$ is singly-ruled by a null geodesic\\ \hline
 2-hyperbolic & $\lambda_{11} \lambda_{22} < 0$ & \begin{tabular}{c} $M$ generic; carries an asymptotic net;\\ no curvature lines\end{tabular} \\ \hline
 2-elliptic & $\lambda_{11} \lambda_{22} > 0$ & \begin{tabular}{c}$M$ generic; no asymptotic directions;\\ carries a net of curvature lines\end{tabular}\\ \hline
 \end{tabular}
 \caption{Second order $CSp(4,\R)$-invariant classification of (un)parametrized hyperbolic surfaces}
  \label{2-classification}
 \end{center}
 \end{table}

 A priori, Table \ref{2-classification} is a classification of parametrized surfaces, but we show in Appendix \ref{app:1-adaptation} that it is also a classification of unparametrized surfaces.
 If $M$ is 2-isotropic, any 2-adapted moving frame satisfies $\alpha_{12} = \alpha_{21} = 0$.  The $d\alpha_{12}, d\alpha_{21}$ equations reduce to
 $0 = \beta_3 \wedge \theta_1 = \beta_3 \wedge \theta_2$, which implies $\beta_3 = 0$ and hence by \eqref{1-str-eqs}, $d\mf{3} = 0$.
 Thus, the CTS is {\em constant} for all of $M$.  Since $\ambient{\mf{0}}{\mf{3}} = 0$, then $[\mf{0}] \in \S{\mf{3}}$, i.e.\ $M$ is an open subset of $\S{\mf{3}}$. 
 
 \begin{prop}
 \label{prop:sphere} A hyperbolic surface $M \subset \Q$ is 2-isotropic iff it is an open subset of an indefinite sphere.
 \end{prop}
 
 \begin{proof}
 The first direction has been established.  For the converse, work in the chart $\psi$.  $M = \S{0,0,1,0,0}$ has local parametrization $[1,u,0,v,uv]$.  The moving frame $ \mf{0} = (1,u,0,v,uv)$, $\mf{1} = (0,1,0,0,0)$, $\mf{2} = (0,0,0,1,0)$, $\mf{3} = (0,0,0,1,0)$, $\mf{4} = (0,0,0,0,1)$ is 2-adapted with structure equations $d\mf{0} = \theta_1 \mf{1} + \theta_2 \mf{2}$ and $d\mf{i} = 0$ for $i=1,2,3,4$, where $\theta_1 = du$, $\theta_2 = dv$, so $\lambda_{11} = \lambda_{22} = 0$.  Using Lemma \ref{lem:sp-transitive}, any indefinite sphere is 2-isotropic.
 \end{proof}

 Since $Sp(4,\R)$ acts transitively on $\mathcal{H}_-$, then there are no invariants for such surfaces.  Indeed, any such is a CSI surface with $s=0$ having 6-dimensional symmetry algebra (as a subalgebra of $\sp(4,\R)$):
 \[
 \partial_r, \quad \partial_t, \quad
 -r^2\partial_r - rs \partial_s - s^2 \partial_t, \quad
 -s^2\partial_r - st \partial_s - t^2 \partial_t, \quad
 -2r\partial_r - s\partial_s, \quad
 -s\partial_s - 2t\partial_t.
 \]
 In \eqref{spheres}, spheres were given in local coordinates $(r,s,t)$.  Fibrewise a MA PDE is a sphere.
 
 \begin{thm}
 The class of hyperbolic MA PDE is contact-invariant.
 \end{thm}
 
 \begin{proof}
 Being hyperbolic and 2-isotropic are $CSp(4,\R)$-invariant conditions.  By Theorem \ref{thm:symplectic-contact}, we are done.  (This argument is equivalent to that given in the Introduction.)
 \end{proof}
 
 Thus, we have reproven a classical fact from a completely different (and arguably simpler) point of view.  Contact-invariance of the elliptic and parabolic MA PDE classes are similarly established.
 
 \begin{rem} Although indefinite spheres admit no $CSp(4,\R)$-invariants, hyperbolic MA PDE certainly do admit contact invariants, e.g. the wave $(z_{xy} = 0)$ and Liouville equation $(z_{xy} = \exp(z))$ are contact-inequivalent \cite{GK1993}.
 \end{rem}
 
  \subsection{\MA/ invariants}
 \label{sec:MA-inv}

 A parametric description of the relative invariants $\lambda_{11}, \lambda_{22}$ appearing in Table \ref{2-classification} is given in Appendix \ref{app:param}.  Computations are made in local coordinates $(r,s,t)$ of the chart $\psi$ and with respect to a {\em null} parametrization $(u,v)$ on a hyperbolic surface $M$ with respect to $[\metric]$, where $\metric = dr \odot dt - ds\odot ds$.  The null parameter condition is $r_u t_u - s_u{}^2 = r_v t_v - s_v{}^2 = 0$.
The functions $\lambda_{11},\lambda_{22}$ are multiples of the {\em \MA/ (relative) invariants} $I_1,I_2$, given by
  \begin{align}
 I_1 = \det\mat{ccc}{ r_{uu} & s_{uu} & t_{uu} \\ r_u & s_u & t_u \\ r_v & s_v & t_v}, \qquad
 I_2 = \det\mat{ccc}{ r_{vv} & s_{vv} & t_{vv} \\ r_u & s_u & t_u \\ r_v & s_v & t_v}.
 \label{I12}
 \end{align}
 A second description of the MA invariants appears in Appendix \ref{app:MA}.  If $M$ is given implicitly by $F(r,s,t) = 0$ and null parameters $(u,v)$ on $M$, we have
 \begin{align}
 I_1 = \transpose{\mat{c}{ r_u \\ s_u \\ t_u}} \mat{ccc}{ F_{rr} & F_{rs} & F_{rt}\\ F_{sr} & F_{ss} & F_{st}\\F_{tr} & F_{ts} & F_{tt}} \mat{c}{ r_u \\ s_u \\ t_u}, \qquad
 I_2 = \transpose{\mat{c}{ r_v \\ s_v \\ t_v}} \mat{ccc}{ F_{rr} & F_{rs} & F_{rt}\\ F_{sr} & F_{ss} & F_{st}\\F_{tr} & F_{ts} & F_{tt}} \mat{c}{ r_v \\ s_v \\ t_v}.
 \label{I12-implicit}
 \end{align}
 
 \begin{rem}
 Equations \eqref{I12} and \eqref{I12-implicit} differ by an overall scaling, but this does not affect the classification result.
 \end{rem}
 
 \begin{prop} \label{prop:red-var} Let $M \subset \Q$ be given by $F(r,s,t) = 0$. The following hold:
 \begin{enumerate}
 \item $F(r,s) = 0$ is hyperbolic iff $F_s \neq 0$.  If so, then it is 2-isotropic or 2-parabolic.
 \item $F(r,t) = 0$ is hyperbolic iff $F_r F_t < 0$.  If so, then it is 2-isotropic or 2-elliptic.
 \end{enumerate}
 \end{prop}
 
 \begin{proof} From Section \ref{sec:EPH}, $M$ is hyperbolic iff $F_r F_t - \frac{1}{4} F_s{}^2 < 0$.  We calculate $I_1,I_2$ and use Table \ref{2-classification} to classify $M$.
 \[
 \begin{array}{|c|c|c|c|} \hline
 \mbox{Equation} & \mbox{Local null basis of } TM & \Hess(F) & \mbox{Invariants}\\ \hline\hline
 F(r,s) = 0 & \begin{array}{l} n_1 = \parder{t}\\ n_2 = F_r \parder{s} - F_s \parder{r} \end{array} & \mat{ccc}{ F_{rr} & F_{rs} & 0\\ F_{sr} & F_{ss} & 0\\0 &0 & 0} & \begin{array}{l} I_1 = 0 \\ I_2 = F_{rr} F_s{}^2 - 2F_{rs} F_r F_s + F_{ss} F_r{}^2 \end{array} \\ \hline
 F(r,t) = 0 & \begin{array}{l} n_1 = F_r \parder{t} - F_t \parder{r} + \sqrt{-F_r F_t}\parder{s} \\
 n_2 = F_r \parder{t} - F_t \parder{r} - \sqrt{-F_r F_t} \parder{s} \end{array} & \mat{ccc}{ F_{rr} & 0 & F_{rt}\\ 0 &0 & 0\\F_{tr} & 0 & F_{tt}}& I_1 = I_2 = F_{rr} F_t{}^2 - 2F_{rt} F_r F_t + F_{tt} F_r{}^2  \\ \hline
 \end{array}
 \]
 In the case $F(r,t) = 0$ is 2-generic, we have $I_1 = I_2 \neq 0$, so $\epsilon = \sgn(I_1 I_2) = \sgn((I_1)^2) = +1$, so it is 2-elliptic.
 \end{proof} 

 The MA invariants $I_1, I_2$ are necessarily also relative contact-invariants for hyperbolic PDE, where we must interpret the null parameters $u,v$ as functions of $\jet = (x,y,z,p,q) \in J^1$, i.e. dependent on the fibre $J^2|_\jet$.  As discussed in the Introduction, the MA invariants for hyperbolic PDE have been calculated several times in the literature by Vranceanu \cite{Vranceanu1940}, Jur\'a\v{s} \cite{Juras1997}, and The \cite{The2008}.  The calculations in these papers were quite involved.  Our computation above of $I_1, I_2$ based on a 2-adapted lift for a hyperbolic surface $M \subset \Lg$ is geometrically simple. 
 
 Let us compare $I_1,I_2$ above with the MA invariants calculated in \cite{The2008}, denoted here $I_1^{(T)}, I_2^{(T)}$:
 \[
 I_1^{(T)} = \det\mat{ccc}{ F_r & F_s & F_t\\ \lambda_+ & F_t & 0\\ \l(\frac{F_t}{\lambda_+} \r)_r & \l(\frac{F_t}{\lambda_+} \r)_s & \l(\frac{F_t}{\lambda_+} \r)_t}, \qquad  I_2^{(T)} = \det\mat{ccc}{ 0 & F_r & \lambda_+\\ F_r & F_s & F_t\\ \l(\frac{F_r}{\lambda_+} \r)_r & \l(\frac{F_r}{\lambda_+} \r)_s & \l(\frac{F_r}{\lambda_+} \r)_t}
 \]
 where $\lambda_+ = \frac{F_s}{2} + \sqrt{-\Delta}$, where $\Delta = F_r F_t - \frac{1}{4} F_s{}^2 < 0$, and (without loss of generality) it is assumed that $F_s \geq 0$ at a given point $\tilde\jet$.  Evaluated on $\Sigma = \{ F=0\}$, hyperbolic PDE are classified by:  MA ($I_1^{(T)} = I_2^{(T)} = 0$),
 Goursat ($I_1^{(T)} = 0$ or $I_2^{(T)} = 0$, but not both), 
 generic ($\epsilon^{(T)} = \sgn(I_1^{(T)}I_2^{(T)}) = \pm 1$).
%
 Observe that $I_1^{(T)}, I_2^{(T)}$ depend only on second derivatives in $(r,s,t)$.  Hence, fibrewise, for any $\jet\in J^1$, they must be $CSp(C_\jet,[\eta])$-invariants of surfaces $\Sigma|_\jet \subset J^2|_\jet$ and must be a function of the second order $CSp(4,\R)$-invariants $I_1,I_2$ for surfaces in $\Lg$.
The hyperbolic MA class is characterized both by $I_1^{(T)} = I_2^{(T)} = 0$ and $I_1 = I_2 = 0$.  The PDE $z_{xy} = \frac{1}{2} (z_{yy})^2$ is in the Goursat class, while $s = \frac{t^2}{2}$ regarded as a surface in $\Lg$ has $I_1 \neq 0$, $I_2 = 0$.  The PDE $3z_{xx} (z_{yy})^3 + 1=0$ is in the generic hyperbolic class with $\epsilon^{(T)} = 1$; the surface $3rt^3+1=0$ in $\Lg$ is generic, and from Proposition \ref{prop:red-var}, it is 2-elliptic.  Thus, both invariant descriptions coincide.

 \begin{thm}
 A hyperbolic PDE $\Sigma \subset J^2$ is:
  \[
 \l\{ \begin{tabular}{l} \MA/ \\ Goursat\\ 2-elliptic\\ 2-hyperbolic \end{tabular} \r. \qbox{iff for any $\jet \in J^1$, $\Sigma|_\jet \subset J^2|_\jet$ is} \l\{ \begin{tabular}{l} an indefinite sphere\\ singly-ruled by null geodesics\\ 2-elliptic $(\epsilon = +1)$\\ 2-hyperbolic $(\epsilon = -1)$\end{tabular}\r.
 \]
 and these PDE classes are contact-invariant and mutually contact-inequivalent.
 \end{thm}
 
 \begin{rem}
 The statement for hyperbolic 2-elliptic / 2-hyperbolic in the above theorem is a definition.
 \end{rem}

 \subsection{3-adaptation: cone congruences and the conjugate manifold}

 In the generic and singly-ruled cases, there are additional geometric objects canonically associated with a third order neighbourhood of $M$.  Suppose that $\mf{}$ is a 2-adapted moving frame, so $\lambda_{12} = 0$.  The $d\alpha_{12}, d\alpha_{21}$ equations yield 
 $0 = (d\lambda_{11} + \lambda_{11} (3\alpha_{11} - \alpha_{22}) ) \wedge \theta_1 + \beta_3 \wedge \theta_2$,
 $0 = (d\lambda_{22} + \lambda_{22} (3\alpha_{22} - \alpha_{11}) ) \wedge \theta_2 + \beta_3 \wedge \theta_1$,
 and hence by Cartan's lemma, there exist third order functions $\lambda_{ijk}$ on $M$ such that 
 \begin{align}
 d\lambda_{11} + \lambda_{11} (3\alpha_{11} - \alpha_{22}) &= \lambda_{111} \theta_1 + \lambda_{112} \theta_2, \\
 \beta_3 &= \lambda_{112} \theta_1 + \lambda_{221} \theta_2,  \label{3-Cartan-lemma}\\
 d\lambda_{22} + \lambda_{22} (3\alpha_{22} - \alpha_{11}) &= \lambda_{221} \theta_1 + \lambda_{222} \theta_2
 \end{align}
 On a generic surface $M$ (so $\lambda_{11} \lambda_{22} \neq 0$), define
 \begin{align}
 \mf{n_1} = \mf{1} - \frac{\lambda_{221}}{\lambda_{22}} \mf{0}, \qquad
 \mf{n_2} = \mf{2} - \frac{\lambda_{112}}{\lambda_{11}} \mf{0}, \qquad
 \mf{n_1'} = \mf{1} + \frac{\lambda_{111}}{3\lambda_{11}} \mf{0}, \qquad
 \mf{n_2'} = \mf{2} + \frac{\lambda_{222}}{3\lambda_{22}} \mf{0}.
\label{n12}
 \end{align}
  \begin{align}
 &\begin{array}{|c|c|c|c|c|c|} \hline
 & \bar\alpha_{11} + \bar\alpha_{22} & \bar\alpha_{22} - \bar\alpha_{11} & \bar\beta_3 \\ \hline
 H_2' & \alpha_{11} + \alpha_{22} - r_2 s_2 \theta_2 -\frac{s_1}{r_2} \theta_1 + \frac{dr_1}{r_1}
 	    & \alpha_{22} - \alpha_{11} - r_2 s_2 \theta_2 + \frac{s_1}{r_2} \theta_1 + \frac{dr_2}{r_2}
	    & \frac{1}{r_1} \beta_3 + \frac{s_1}{r_1 r_2} \alpha_{21} + \frac{r_2 s_2}{r_1} \alpha_{12} \\
 R_1 & \alpha_{11} + \alpha_{22} & -(\alpha_{22} - \alpha_{11}) & \beta_3 \\ 
 R_2 & \alpha_{11} + \alpha_{22} & \alpha_{22} - \alpha_{11} & -\beta_3 \\ \hline
 \end{array} \label{H2-change1}\\
 &\begin{array}{|c|c|c|c|c|c|c|c|c|} \hline
      & \bar\lambda_{111} & \bar\lambda_{112} & \bar\lambda_{221} & \bar\lambda_{222} 
	& \bar{\mf{}}_{n_1} & \bar{\mf{}}_{n_2} & \bar{\mf{}}_{n_1'} & \bar{\mf{}}_{n_2'}\\ \hline
 H_2' & \frac{r_2{}^2 (r_2 \lambda_{111} - 3s_1 \lambda_{11})}{r_1{}^2} 
	& \frac{r_2 (\lambda_{112} + s_2 r_2 \lambda_{11})}{r_1{}^2} 
	& \frac{r_2 \lambda_{221} + s_1 \lambda_{22}}{r_1{}^2 r_2{}^2} 
	& \frac{\lambda_{222} - 3 r_2 s_2  \lambda_{22}}{r_1{}^2 r_2{}^3}
	& r_2 \mf{n_1} & \frac{1}{r_2} \mf{n_2} & r_2 \mf{n_1'} & \frac{1}{r_2} \mf{n_2'}\\
 R_1  & \lambda_{222} & \lambda_{221} & \lambda_{112} & \lambda_{111} & -\mf{n_2} & -\mf{n_1} & -\mf{n_2'} & -\mf{n_1'}\\
 R_2  & \lambda_{111} & \lambda_{112} & \lambda_{221} & \lambda_{222} & -\mf{n_1} & -\mf{n_2} & -\mf{n_1'} & -\mf{n_2'}\\ \hline
 \end{array}
 \label{H2-change2}
 \end{align}
 The two cone pairs $\{ \S{\mf{n_1}}, \S{\mf{n_2}} \}$ and $\{ \S{\mf{n_1'}}, \S{\mf{n_2'}} \}$ are geometrically associated to $M$.  The former normalizes $\bar\lambda_{112} = \bar\lambda_{221} = 0$ and $\beta_3 = 0$, while the latter normalizes $\bar\lambda_{111} = \bar\lambda_{222} = 0$, c.f. \eqref{3-Cartan-lemma}.  We choose the former to define our 3-adaptation, c.f. Remark \ref{rem:3-adapt}.

 \begin{defn} \label{defn:3-adapted-generic}
 Let $M$ be a hyperbolic 2-generic.  The 3-adapted frame bundle and its structure group are
 \[
 \F_3(M) = \{ \mf{}\in \F_2(M) : \{ [\mf{1}],[\mf{2}] \}  = \{ [\mf{n_1}],[\mf{n_2}] \} \}, \qquad H_3 = H_3' \rtimes (\Z_2 \times \Z_2),
 \]
 where 
 $H_3' = \l\{ X(r_1,r_2) := diag\l(r_1, r_2, \frac{1}{r_2}, 1, \frac{1}{r_1} \r) : r_1 r_2 > 0 \r\}$, and $R_1,R_2$ generate the $\Z_2 \times \Z_2$ factor.
 \end{defn}

 If $M$ is singly-ruled, then using $R_1$, we may assume $\lambda_{11} \neq 0$ and $\lambda_{22} = 0$ (hence, $\lambda_{221} = \lambda_{222} = 0$).  In this case, $\mf{n_1}$ and $\mf{n_2'}$ are not well-defined.
 We refer to $\S{\mf{n_2}}$ as the {\em primary cone congruence} and $\S{\mf{n_1'}}$ the {\em secondary cone congruence}.   The change of frame $(r_1,r_2, s_1,s_2) = \l(1,1, \frac{\lambda_{111}}{3\lambda_{11}}, -\frac{\lambda_{112}}{\lambda_{11}}\r)$ normalizes $\beta_3=0$ and $\lambda_{111}=\lambda_{112} = 0$.  Hence, $d\lambda_{11} + \lambda_{11}(3\alpha_{11} - \alpha_{22}) = 0$, i.e.\ $3\alpha_{11} - \alpha_{22} = -d(\ln(\lambda_{11}))$ is exact.

 \begin{defn}  \label{defn:3-adapted-singly-ruled}
 Let $M$ be hyperbolic singly-ruled.   The 3-adapted frame bundle and its structure group are
 \[
 \bar\F_3(M) = \{ \mf{}\in \F_2(M) : [\mf{1}] = [\mf{n_1'}],\,\, [\mf{2}] = [\mf{n_2}] \}, \qquad \bar{H}_3 = H_3' \rtimes \Z_2,
 \]
 where $H_3'$ is as in Definition \ref{defn:3-adapted-generic}, and $R_2$ generates the $\Z_2$ factor.
 \end{defn}

 In both the singly-ruled and generic cases, the residual structure groups $\bar{H}_3$ and $H_3$ preserve $[\mf{4}]$.  Hence, for any 3-adapted frame, the null diamond $([\mf{0}], [\mf{1}],[\mf{2}], [\mf{4}])$ inscribed on $\S{\mf{3}}$ is geometrically associated to $M$. 
 
 \begin{defn}
 For any 3-adapted frame $\mf{}$, we call $[\mf{4}] \in \Q$ the {\em conjugate point}.  For a 3-adapted moving frame $\mf{}$, we call the image of $[\mf{4}]$ (regarded as a map $M \ra \Q$) the {\em conjugate manifold} $M' \subset \Q$ of $M$.
 \end{defn}
 
 Given a 3-adapted $\mf{}$ for $M$, we have $M'$ tangent to $\S{\mf{3}}$ at $[\mf{4}]$ since by \eqref{1-str-eqs}, $\ambient{d\mf{4}}{\mf{3}} = 2\beta_3 = 0$.  Thus, if $M'$ is also a surface, it is a second envelope for CSG of $M$. 
 
 \begin{rem} The requirement $\beta_3 = 0$ for the second envelope motivates our choice of defining the 3-adaptation via $\{ \S{\mf{n_1}}, \S{\mf{n_2}} \}$ instead of $\{ \S{\mf{n_1'}}, \S{\mf{n_2'}} \}$.
 \label{rem:3-adapt}
 \end{rem}
 
 We later express $\dim(M')$ in terms of the invariants of $M$.  We immediately caution the reader on several points.  In general: (i) $M'$ may have singularities, (ii) the CTS of $M$, $M'$ may not agree pointwise, (iii) $(M')' \neq M$, (iv) $M'$ may not have the same type as $M$ (even if $\dim(M') = 2$).

 Thus far, we can canonically assign to any singly-ruled or generic surface $M$ a geometric moving system of spheres $\S{\mf{0}},..., \S{\mf{4}}$.  The residual scaling freedom in individual frame vectors represented in the 3-adapted structure groups $\bar{H}_3$ or $H_3$ will subsequently be reduced by normalizing coefficients in the MC structure equations.  After reducing as much as possible, the residual structure functions will be candidates for invariants for our equivalence problem, but we must still investigate their transformation under $\bar{H}_3$ or $H_3$.

 \section{Hyperbolic surfaces singly-ruled by null geodesics}
 \label{sec:singly}

 For a hyperbolic singly-ruled surface $M$, we constructed in Definition \ref{defn:3-adapted-singly-ruled} the 3-adapted frame bundle $\bar\F_3(M)$ with structure group $\bar{H}_3$.  Any 3-adapted moving frame satisfies $\lambda_{11} \neq 0$, $\lambda_{22} = 0$, $\beta_3 = \alpha_{21} = 0$,
 $\alpha_{12} = \lambda_{11} \theta_1$,
 $3\alpha_{11} - \alpha_{22} = -d(\ln|\lambda_{11}|)$, hence $3d\alpha_{11} = d\alpha_{22}$.  The integral curves of $\theta_1 = \theta_3 = 0$ are null geodesics.
 
 \subsection{Normalization, invariants, and integrability}

 The $d\alpha_{11}, d\alpha_{22}, d\beta_3$ equations \eqref{MC-eqns} imply $3\beta_1 \wedge \theta_1 = -3d\alpha_{11} = -d\alpha_{22} = \beta_2 \wedge \theta_2$ and $\alpha_{12} \wedge \beta_2 = 0$.  By Cartan's lemma, there exist fourth order functions $\Lambda_{11}$, $\Lambda_{12}$ such that
 \[
 \beta_1 = \Lambda_{11} \theta_1 - \frac{\Lambda_{12}}{3} \theta_2, \qquad \beta_2 = \Lambda_{12} \theta_1.
 \]
 The $d\beta_1,d\beta_2$ equations \eqref{MC-eqns} yield $0 = -3(d\Lambda_{11} + 4\Lambda_{11} \alpha_{11}) \wedge \theta_1 + (d\Lambda_{12} + 2\Lambda_{12} (\alpha_{11} + \alpha_{22})) \wedge \theta_2$ and $0 = (d\Lambda_{12} + 2\Lambda_{12}(\alpha_{11} + \alpha_{22})) \wedge \theta_1$.
 By Cartan's lemma, there exist fifth order functions $\Lambda_{111},\Lambda_{121}$ such that
 \begin{align}
 d\Lambda_{11} + 4\Lambda_{11} \alpha_{11} = \Lambda_{111} \theta_1 - \frac{\Lambda_{121}}{3} \theta_2, \qquad
 d\Lambda_{12} + 2\Lambda_{12}(\alpha_{11} + \alpha_{22}) = \Lambda_{121} \theta_1.
 \label{Lambda121}
 \end{align}
 The MC equations \eqref{MC-eqns} for 3-adapted moving frames reduce to
 \begin{align}
  0 = 3d\alpha_{11} + \Lambda_{12} \theta_1 \wedge \theta_2, \qquad 0 = d\alpha_{22} + \Lambda_{12} \theta_1 \wedge \theta_2, \qquad
 0 = d\theta_1 + 2\theta_1 \wedge \alpha_{11}, \qquad 0 = d\theta_2 + 2\theta_2 \wedge \alpha_{22}.
 \label{SR-MC-eqns}
 \end{align}
 Under an $\bar{H}_3$-frame change, $\alpha_{11}, \alpha_{22},\theta_1,\theta_2$ transform according to \eqref{H2-change1}-\eqref{H2-change2} (with $s_1 = s_2 = 0$).  We also have:
  \begin{align}
    &\begin{array}{|c|c|c|c|c|c|c|c|c|c|c|} \hline
 & \bar\beta_1 & \bar\beta_2 & \bar\lambda_{11} & \bar\Lambda_{11} & \bar\Lambda_{12} & \bar\Lambda_{111} & \bar\Lambda_{121} \\ \hline
  X(r_1,r_2) \in H_3'&   \frac{r_2}{r_1} \beta_1 & \frac{1}{r_1r_2} \beta_2 & \frac{r_2{}^2}{r_1} \lambda_{11} & \frac{r_2{}^2}{r_1{}^2} \Lambda_{11} & \frac{1}{r_1{}^2} \Lambda_{12} & \frac{r_2{}^3}{r_1{}^3} \Lambda_{111} & \frac{r_2}{r_1{}^3} \Lambda_{121}\\
  R_2 & -\beta_1 & -\beta_2 & -\lambda_{11} & \Lambda_{11} & \Lambda_{12} & -\Lambda_{111} & -\Lambda_{121}\\ \hline
  \end{array}
 \label{SR-H3-change}
 \end{align}

 \begin{defn}
 Let $\delta_2 = \sgn(\Lambda_{12}) \in \{ -1,0,1\}$ and $\delta_1 = \sgn(\Lambda_{11})\in \{ -1,0,1\}$.
 \end{defn}
 
 From \eqref{SR-H3-change}, $\delta_1,\delta_2$ are $CSp(4,\R)$-invariants and we assume they are locally constant.  
 Setting $r_1 = \lambda_{11}(r_2)^2$, we normalize $\bar\lambda_{11} = 1$, hence $3\bar\alpha_{11} = \bar\alpha_{22}$ and $\bar\alpha_{12} = \bar\theta_1$.
 
 \begin{itemize}
 \item $\delta_2 \neq 0$: Normalize $\bar\Lambda_{12} = \delta_2$ by setting $r_1 = \sgn(\lambda_{11}) \sqrt{|\Lambda_{12}|}$, $r_2 = \sgn(\lambda_{11}) |\Lambda_{12}|^{1/4} |\lambda_{11}|^{-1/2}$.  The structure group is reduced to the identity.  The residual functions are 
 $\zeta_1 = \frac{\Lambda_{11}}{|\lambda_{11}||\Lambda_{12}|^{1/2}}$ and $\zeta_2 = \frac{\delta_2\Lambda_{121}}{4|\lambda_{11}|^{1/2}|\Lambda_{12}|^{5/4}}$.  Under $\bar{H}_3$, $\zeta_1$ and $(\zeta_2)^2$ are invariant so these are the fundamental invariants.
 
 \item $\delta_2 = 0$, $\delta_1 = \pm 1$: Normalize $\bar\Lambda_{11} = \delta_1$ by setting $r_2 = \frac{\sqrt{|\Lambda_{11}|}}{\lambda_{11}}$, $r_1 = \frac{|\Lambda_{11}|}{\lambda_{11}}$.  The structure group is reduced to the identity.  The residual function is $\zeta = \frac{2\delta_1 \Lambda_{111}}{|\Lambda_{11}|^{3/2}}$, whose square is $\bar{H}_3$-invariant.  The fundamental invariant is $\zeta^2$.

 \item $\delta_1 = \delta_2 = 0$: The residual group $\bar{H}_r = \{ diag(r^2,r,\frac{1}{r},1,\frac{1}{r^2}): r > 0 \}$ induces $\bar\alpha_{11} = \alpha_{11} + \frac{dr}{4r}$ by \eqref{MC-frame-change}.  Since $d\alpha_{11} = 0$ by \eqref{SR-MC-eqns}, then $\alpha_{11}$ is locally exact by Poincar\'e's lemma, and we normalize $\bar\alpha_{11} = 0$.  There are no residual invariant functions.  There is a residual change of frame, $\mfp{} = \mf{} h$ for $h : U \ra \bar{H}_r$ locally constant.
 \end{itemize}
 
 \begin{table}[h]
 \begin{center}
 $\begin{array}{|c|c|l|l|} \hline
 \mbox{Classification} &
 \mbox{Invariants} &\mbox{Normalized moving frame structure equations} & \mbox{Integrability conditions}\\ \hline\hline
 \begin{array}{c}
 \dim(M') = 2:\\
 \delta_2 = \pm 1 
 \end{array} & \zeta_1, (\zeta_2)^2 &
 \begin{array}{l}
 d\mf{0} = 2\zeta_2 \theta_1 \mf{0} + \theta_1 \mf{1} + \theta_2 \mf{2} \nonumber\\
 d\mf{1} = (\zeta_1\theta_1 - \frac{\delta_2}{3} \theta_2) \mf{0} + \zeta_2 \theta_1 \mf{1} - \theta_1 \mf{3} + \theta_2 \mf{4} \nonumber\\
 d\mf{2} = \delta_2 \theta_1 \mf{0} - \zeta_2 \theta_1 \mf{2} + \theta_1 \mf{4} \nonumber\\
 d\mf{3} = - 2\theta_1  \mf{2} \nonumber\\
 d\mf{4} = \delta_2 \theta_1 \mf{1} + (\zeta_1\theta_1 - \frac{\delta_2}{3} \theta_2) \mf{2} - 2\zeta_2 \theta_1 \mf{4}
 \end{array} &
 \begin{array}{l}
 0 = d\theta_1 \\
 0 = d\theta_2 - 3\zeta_2\theta_1 \wedge \theta_2\\
 0 = (3d\zeta_2 - 2\delta_2\theta_2 ) \wedge \theta_1
 \end{array} \\ \hline
 \begin{array}{c}
  \dim(M') =1:\\
 \begin{array}{l}
 \delta_2 = 0, \\
 \delta_1 = \pm 1
 \end{array}
 \end{array} &\zeta^2 &
 \begin{array}{l}
 d\mf{0} = 2\zeta\theta_1 \mf{0} + \theta_1 \mf{1} + \theta_2 \mf{2} \nonumber\\
 d\mf{1} = \delta_1\theta_1 \mf{0} + \zeta\theta_1 \mf{1} - \theta_1 \mf{3} + \theta_2 \mf{4} \nonumber\\
 d\mf{2} = -\zeta\theta_1 \mf{2} + \theta_1 \mf{4} \nonumber\\
 d\mf{3} = - 2\theta_1  \mf{2} \nonumber\\
 d\mf{4} = \delta_1 \theta_1 \mf{2} - 2\zeta \theta_1 \mf{4} \nonumber
 \end{array} &
 \begin{array}{l}
 0 = d\theta_1\\
 0 = d\theta_2 - 3\zeta \theta_1 \wedge \theta_2 \\
 0 = d\zeta \wedge \theta_1
 \end{array}
 \\ \hline
 \begin{array}{c}
 \dim(M') =0:\\
 \delta_1 = \delta_2 = 0
 \end{array} & - &
 \begin{array}{l}
 d\mf{0} = \theta_1 \mf{1} + \theta_2 \mf{2} \nonumber\\
 d\mf{1} =  - \theta_1 \mf{3} + \theta_2 \mf{4} \nonumber\\
 d\mf{2} =  \theta_1 \mf{4} \nonumber\\
 d\mf{3} = - 2\theta_1  \mf{2} \nonumber\\
 d\mf{4} = 0
 \end{array} &
 \begin{array}{l}
 0 = d\theta_1 = d\theta_2
 \end{array} \\ \hline
 \end{array}$
 \caption{Invariant sub-classes and structure equations for surfaces singly-ruled by null geodesics}
   \label{table:1-ruled-summary}
 \end{center}

 \end{table}

 Dropping bars, we summarize the results in Table \ref{table:1-ruled-summary}. 
 By Theorem \ref{thm:Cartan-equiv}, our solution to the parametrized equivalence problem for hyperbolic singly-ruled surfaces is:
 
  \begin{thm}[Invariants and equivalence for hyperbolic singly-ruled surfaces]  Let $U$ be connected and $i,\tilde{i} : U \ra \Lg$ hyperbolic singly-ruled surfaces.  If $M = i(U)$ and $\tilde{M} = \tilde{i}(U)$ are $CSp(4,\R)$-equivalent, then $\delta_1, \delta_2$ agree for $M,\tilde{M}$ on $U$, and
 \begin{enumerate}
 \item $\delta_2 \neq 0$: $\zeta_1$ and $(\zeta_2)^2$ agree for $M,\tilde{M}$ on $U$.
 \item $\delta_2 = 0, \, \delta_1 \neq 0$: $\zeta^2$ agrees for $M,\tilde{M}$ on $U$.
 \item $\delta_1 = \delta_2 = 0$: no additional conditions.
 \end{enumerate}
 Conversely, if the above hold, then $M,\tilde{M}$ are $CSp(4,\R)$-equivalent.
 \end{thm}

 In Appendix \ref{app:3-SR}, we show that $\zeta_1, (\zeta_2)^2, \zeta^2$ are invariant under reparametrizations.  Hence, these solve the corresponding unparametrized equivalence problem.  By Theorem \ref{thm:existence}, we have:
 
 \begin{thm}[Bonnet theorem for hyperbolic singly-ruled surfaces]  The integrability conditions in Table \ref{table:1-ruled-summary} are the only local obstructions to the existence of a hyperbolic singly-ruled surface with prescribed invariants.
 \end{thm}
 
 \subsection{Geometric interpretation of invariants}
 
 The $\delta_1, \delta_2$ invariants for $M$ have an interpretation in terms of the conjugate manifold $M'$.  Let $\mf{}$ be a 3-adapted moving frame.
 From the $\mf{4}$ equation \eqref{1-str-eqs}, $T_{[\mf{4}]} M' \subset T_{[\mf{4}]} \Q \cong \mf{4}^\perp / \ell_{\mf{4}}$ is spanned by the coefficients of $\theta_1$ and $\theta_2$, i.e.\
 \begin{align}
 \Lambda_{12} \mf{1} + \Lambda_{11} \mf{2} \qbox{and} \Lambda_{12} \mf{2} \quad\mod \mf{4}.
 \label{SR-conj-basis}
 \end{align}
 
 \begin{prop}  \label{prop:SR-conj} Let $M$ be a singly-ruled hyperbolic surface.  Then $\dim(M')= \rank\mat{cc}{ \Lambda_{12} &  \Lambda_{11} \\ 0 & \Lambda_{12}}$.
 \end{prop}
 
 We have $\dim(M') = 0$ iff $\delta_1 = \delta_2 = 0$.  If $\dim(M') = 1$, then $\delta_2 = 0$ and $\delta_1 \neq 0$.  From Table \ref{table:1-ruled-summary}, we see $d\mf{2}, d\mf{4} \equiv 0 \mod \{  \mf{2}, \mf{4} \}$.  Hence, the null curve $M'$ is in fact a null geodesic.
 Suppose $\dim(M') = 2$, i.e.\ $\delta_2 \neq 0$.  The conformal structure on $M'$ is represented with respect to the basis \eqref{SR-conj-basis} by (multiples of) the matrix $\mat{cc}{ 2\Lambda_{11} \Lambda_{12} & \Lambda_{12}{}^2 \\ \Lambda_{12}{}^2 & 0}$, so $M'$ is hyperbolic.  Let us further classify $M'$.
 Given a 3-adapted $\mf{}$ on $M$, $\mft{} = (\mft{0}, \mft{1}, \mft{2}, \mft{3}, \mft{4}) = (\mf{4},\mf{1},\mf{2},-\mf{3},\mf{0})$ is 2-adapted on $M'$, so the CSGs of $M$, $M'$ agree, with
 \[
 \begin{array}{l}
 \tilde\theta_1 = \beta_2,\\
 \tilde\theta_2 = \beta_1,
 \end{array}
 \qquad 
 \begin{array}{l}
 \tilde\beta_1 = \theta_2, \\
 \tilde\beta_2 = \theta_1,
 \end{array} \qquad
 \begin{array}{l}
 \tilde\theta_3 = 0,\\
 \tilde\beta_3 = 0,
 \end{array}
 \qquad 
 \begin{array}{l}
 \tilde\alpha_{11} = -\alpha_{22},\\
 \tilde\alpha_{22} = -\alpha_{11}, \end{array}
 \qquad
 \begin{array}{l}
 \tilde\alpha_{12} = -\alpha_{12}, \\
 \tilde\alpha_{21} = 0,
 \end{array}
 \qquad
  \begin{array}{l}
 \tilde\lambda_{11} = -\frac{\lambda_{11}}{\Lambda_{12}},\\
 \tilde\lambda_{111} = \frac{\lambda_{11}\Lambda_{121}}{\Lambda_{12}{}^3}.
 \end{array}
 \]
 Hence, $M'$ is singly-ruled.
 In general, $\mft{}$ is not 3-adapted unless $\tilde\lambda_{111} = 0$.
 By the discussion preceding Definition \ref{defn:3-adapted-singly-ruled}, we consider $\mf{}' = (\mf{4},\mf{1}+s_1\mf{4},\mf{2},-\mf{3},\mf{0}+ s_1\mf{2})$, where $s_1 = \frac{\tilde\lambda_{111}}{3\tilde\lambda_{11}} = -\frac{\Lambda_{121}}{3(\Lambda_{12})^2}$.  This is 3-adapted for $M'$, and
 \begin{align*}
 \beta_2' &= \tilde\beta_2 = \theta_1 = \frac{1}{\Lambda_{12}} \beta_2 = \frac{1}{\Lambda_{12}} \tilde\theta_1 = \frac{1}{\Lambda_{12}} \theta_1'  \qRa \Lambda_{12}' = \frac{1}{\Lambda_{12}} \qRa \delta_2' = \delta_2 \qRa \dim((M')') = 2.
 \end{align*}
 Note $\S{\mf{2}}$, $\S{\mfp{2}}$ always agree.  As observed earlier, $\zeta_2$ is nonconstant, so $\Lambda_{121}$ and $s_1$ cannot vanish identically.  Thus, in general $\S{\mf{1}} \neq \S{\mfp{1}}$ and $(M')' \neq M$.

  \subsection{2-parabolic CSI surfaces and other examples}
 
  The functions $\Lambda_{12}, \Lambda_{11}, \Lambda_{111}, \Lambda_{121}$ are calculated in Appendix \ref{app:3-SR} and it is shown that $\delta_2, \delta_1, \zeta_1, (\zeta_2)^2, \zeta^2$ are invariants of unparametrized surfaces.  If $M$ is hyperbolic singly-ruled CSI, then $\dim(M') = 0$ or $\dim(M') = 1$ with $\zeta$ constant.
 
  \begin{prop} Let $M$ be a hyperbolic singly-ruled surface such that $\dim(M') = 2$.  Then $M$ is not homogeneous.
 \end{prop}
 
 \begin{proof}  From Table \ref{table:1-ruled-summary}, the integrability condition $(3d\zeta_2 - 2\delta_2 \theta_2) \wedge \theta_1 = 0$ implies $\zeta_2$ cannot be constant.  By Theorem \ref{thm:homogeneity}, all such surfaces are {\em not} homogeneous.  
 \end{proof}
 
 \begin{example} Consider $(r,s,t) = (\frac{u}{u+v} - \frac{1}{4} \ln(u), \frac{u^2}{u+v} - \frac{3}{4} u, \frac{u^3}{u+v} - \frac{9}{8} u^2)$.  This is a null parametrization with $I_1 = \frac{u(3v-u)}{4(u+v)^3}$, $I_2 = 0$, $\uvprod = \frac{u^2}{(u+v)^2}$, $\Lambda_{12} = -\frac{3}{(3v-u)^2}$.
 Thus, $\delta_2 = -1$ and $\dim(M') = 2$. 
 \end{example}

 \begin{prop}
 All hyperbolic singly-ruled surfaces $M$ of the form $F(s,t) = 0$ have $\dim(M') \leq 1$.
 \end{prop}

 \begin{proof} From Proposition \ref{prop:red-var}, $M$ is hyperbolic if $F_s \neq 0$.  
 By the implicit function theorem, we may assume that $F$ is given locally by $F(s,t) = s - f(t)$.  This has null parametrization
 $(r,s,t) = (g(u) + v, f(u), u)$, where $g'(u) = f'(u)^2$.
 We compute $\uvprod=1$, $I_1 = f''(u) \neq 0$, $I_2 = I_3 =0$, hence from \eqref{Lambda11-12}, $\Lambda_{12} = 0$, so $\delta_2 =0$ and $\dim(M') \leq 1$. 
 \end{proof}
 
 Surfaces $M$ of the form $s = f(t)$ with $f''(t) \neq 0$ satisfy:
 \begin{align*}
 \Lambda_{11} &= \frac{1}{3} \l( \frac{f''''(u)}{f''(u)} \r) - \frac{4}{9} \l( \frac{f'''(u)}{f''(u)}\r)^2 , \qquad 
  \Lambda_{111} = \Lambda_{11} \l( \ln \frac{|\Lambda_{11}| }{|f''(u)|^{4/3}} \r)_u, \qquad
  \zeta = \frac{2\delta_1 \Lambda_{111}}{|\Lambda_{11}|^{3/2}}
 \end{align*}
 and we obtain from \eqref{3-lift-ruled} the 3-adapted moving frame $\mf{}$:
  \begin{align}
 \mf{0} = \x, \qquad \mf{1} = \x_u + \frac{f'''(u)}{3f''(u)} \x , \qquad \mf{2} = \x_v, \qquad \mf{3} = 2 \iota\normal, \qquad \mf{4} = {\bf Z} + \frac{f'''(u)}{3f''(u)}  \x_v \label{SR-explicit-mf}
 \end{align}
 where $\x = (1,r,s,t,rt-s^2)$, $ \normal = \l(0, -f'(u), -\frac{1}{2}, 0, f(u) - uf'(u)\r)$, ${\bf Z} = (0,0,0,0,1)$, and $\iota = \pm 1$ ensure $\mf{}$ differs from $\mathcal{B}_H$ by an element of $O^+(2,3)$.  We observe that the first entry of $\mf{1}, \mf{2}, \mf{4}$ is zero, hence they lie on the sphere at infinity $\S{{\bf Z}}$.  Consequently, we cannot picture the conjugate manifold $M'$ or the normalizing cones in the $(r,s,t)$ coordinate chart $\psi$.  Since $\x_v = (0,1,0,0,u)$, then $\mf{4} \in span\{ {\bf Z}, (0,1,0,0,0)\}$ which is an isotropic subspace and whose projectivization is a null geodesic in $\Q$.

 From \eqref{SR-explicit-mf}, if $f'''(u) = 0$, then $\mf{4}={\bf Z}$ is constant, so $s = \frac{1}{2} t^2$ has $\dim(M') = 0$.   Less obvious is if $f(u) = \frac{1}{u}$, then $\mf{4} = (0,-\frac{1}{u},0,0,0)$, so $[\mf{4}]$ is constant.  From Table \ref{table:1-ruled-summary}, there are no $CSp(4,\R)$-invariants if $\dim(M') = 0$.  Hence,
 
 \begin{thm} Any hyperbolic singly-ruled surface $M$ with $\dim(M') = 0$ is locally $CSp(4,\R)$-equivalent to either of the surfaces $s = \frac{1}{2} t^2$ $(t > 0)$ or $s = \frac{1}{t}$ $(t > 0)$.
 \end{thm}
 
 Let us calculate some more examples:
 \begin{align*}
 &&\begin{array}{|c|c|c|c|c|c|c|}\hline
 \mbox{Equation} & (r,s,t) & \Lambda_{11} & \Lambda_{111} & \delta_1 & \zeta^2\\ \hline\hline
 s+1 = \sqrt{1 - t^2} & (u+v-\tanh(u),-1+\sech(u),\tanh(u)) & 1 & 0 & +1 & 0\\
 1 - s =  \sqrt{t^2+1} & (-u+v+\tan(u),1-\sec(u),\tan(u))  & -1 & 0 & -1 & 0\\
 s = \ln(t) & (-\frac{1}{u} +v,\ln(u),u) & \frac{2}{9u^2} & \frac{4}{27u^3} & +1 & 8 \\
 s = e^t & (\frac{1}{2} e^{2u}+v,e^{u},u) & -\frac{1}{9}  & \frac{4}{27} & -1 & 64\\ 
 s = \sqrt{t} & ( \frac{1}{4} \ln(u) +v, \sqrt{u}, u) & \frac{1}{4u^2} & 0 & +1 & 0\\ 
 \begin{array}{c}s = t^n \\ (n \neq 2, -1, \frac{1}{2}) \end{array}& (\frac{n^2}{2n-1} u^{2n-1} + v, u^n, u) & \frac{2 + n -n^2}{9u^2} & \frac{2(2n-1)(n-2)(n+1)}{27u^3} & * & \frac{16(2n-1)^2}{|n-2||n+1|}\\ \hline
 \end{array}\\
 &&*\, \delta_1 = +1 \mbox{ if } -1 < n < 2; \,\,\delta_1 = -1 \mbox{ if } n < -1 \mbox{ or } n > 2.
 \end{align*}

 \begin{thm}
 Let $M \subset \Lg$ be a hyperbolic singly-ruled surface with $\dim(M') = 1$.  If $\zeta^2$ is constant and if:
 \begin{enumerate}
 \item $\delta_1 = +1$, then $M$ is locally $CSp(4,\R)$-equivalent to $s = t^n$, where $-1 < n < 2$ satisfies $\zeta^2 = \frac{16(2n-1)^2}{(2-n)(n+1)}$.  If $\zeta = 0$ $[\zeta^2 = 8]$ then $M$ is also locally $CSp(4,\R)$-equivalent to $s+1 = \sqrt{1-t^2}$ $[s = \ln(t)]$.
 \item $\delta_1 = -1$, then $M$ is locally $CSp(4,\R)$-equivalent to $s = t^n$, where $n < -1$ or $n > 2$ satisfies $\zeta^2 = \frac{16(2n-1)^2}{|n-2||n+1|} > 4$.  If $\zeta = 0$ $[\zeta^2 = 64]$, then $M$ is also locally $CSp(4,\R)$-equivalent to $1-s = \sqrt{1+t^2}$ $[s = e^t]$.
 \end{enumerate}
 \end{thm}

 \begin{rem}
 At present, we do not know which singly-ruled CSI surfaces have invariants $\delta_1 = -1$ and $0 < \zeta^2 \leq 4$.
 \end{rem}

 \section{Hyperbolic 2-generic surfaces}
 \label{sec:generic}
 
 For $M$ hyperbolic 2-generic, we defined (Definition \ref{defn:3-adapted-generic}) the 3-adapted frame bundle $(\F_3(M),H_3)$.
  
 \subsection{Normalization, invariants and integrability}
 
 Any 3-adapted moving frame $\mf{}$ satisfies $\lambda_{11} \lambda_{22} \neq 0$, $\lambda_{12} = 0$ and $\alpha_{12} = \lambda_{11} \theta_1$, $\alpha_{21} = \lambda_{22} \theta_2$, $\beta_3 = 0$ with
 \begin{align}
 d\lambda_{11} + \lambda_{11}(3\alpha_{11} - \alpha_{22}) = \lambda_{111} \theta_1, \qquad
 d\lambda_{22} + \lambda_{22}(3\alpha_{22} - \alpha_{11}) = \lambda_{222} \theta_2. \label{generic-lambda3}
 \end{align}
 The $d\beta_3$ equation \eqref{MC-eqns} becomes $\lambda_{11} \theta_1 \wedge \beta_2 + \lambda_{22} \theta_2 \wedge \beta_1 = 0$, so by Cartan's lemma,
 \begin{align}
 \beta_1 = b_{11} \lambda_{11} \theta_1 + b_{12} \theta_2, \qquad
 \beta_2 = b_{21} \theta_1 + b_{11} \lambda_{22} \theta_2, \label{generic-beta12}
 \end{align}
 where $b_{ij}$ are fourth order functions on $M$.  The remaining MC equations \eqref{MC-eqns} for 3-adapted moving frames are
  \begin{align}
 0 &= d\alpha_{11} + (\lambda_{11} \lambda_{22} - b_{12})\theta_1 \wedge \theta_2, \qquad
 0 = d\beta_1 + 2\alpha_{11} \wedge \beta_1, \qquad
 0 = d\theta_1 + 2\theta_1 \wedge \alpha_{11}, \label{MC-gen1}\\
 0 &= d\alpha_{22} + (\lambda_{11} \lambda_{22} - b_{21}) \theta_2 \wedge \theta_1, \qquad
 0 = d\beta_2 + 2\alpha_{22} \wedge \beta_2, \qquad
 0 = d\theta_2 + 2\theta_2 \wedge \alpha_{22}. \label{MC-gen2}
 \end{align}
 With $\kappa_1,\kappa_2,\tau$ defined below in \eqref{kappa12}, \eqref{bij}, an $H_3$-frame change induces:
  \begin{align}
 & \begin{array}{|c|c|c|c|c|c|c|c|c|c|c|} \hline
 & \bar\alpha_{11} + \bar\alpha_{22}  & \bar\alpha_{22} - \bar\alpha_{11} & \bar\theta_1 & \bar\theta_2 & \bar\beta_1 & \bar\beta_2 \\ \hline
  X(r_1,r_2) \in H_3' & \alpha_{11} + \alpha_{22} + \frac{dr_1}{r_1} & \alpha_{22} - \alpha_{11} + \frac{dr_2}{r_2} & \frac{r_1}{r_2} \theta_1 & r_1 r_2 \theta_2 & \frac{r_2}{r_1} \beta_1 & \frac{1}{r_1r_2} \beta_2 \\
  R_1 & \alpha_{11} + \alpha_{22} & -(\alpha_{22} - \alpha_{11}) & \theta_2 & \theta_1 & \beta_2 & \beta_1\\ 
  R_2 & \alpha_{11} + \alpha_{22} & \alpha_{22} - \alpha_{11} & -\theta_1 & -\theta_2 & -\beta_1 & -\beta_2 \\ \hline
  \end{array}  \label{generic-H3-change1}\\
    &\begin{array}{|c|c|c|c|c|c|c|c|c|c|c|c|} \hline
 & \bar\lambda_{11} & \bar\lambda_{22} & \bar\lambda_{111} & \bar\lambda_{222} & \bar{b}_{11} & \bar{b}_{12} & \bar{b}_{21} & \bar\kappa_1 & \bar\kappa_2 & \bar\tau\\ \hline
  X(r_1,r_2) \in H_3'& \frac{r_2{}^2}{r_1} \lambda_{11} & \frac{\lambda_{22}}{r_1 r_2{}^2}  & \frac{r_2{}^3}{r_1{}^2} \lambda_{111} & \frac{ \lambda_{222}}{r_1{}^2 r_2{}^3} & \frac{b_{11}}{r_1} & \frac{b_{12}}{r_1{}^2} & \frac{b_{21}}{r_1{}^2} & \kappa_1 & \kappa_2 & \tau\\
  R_1 & \lambda_{22} & \lambda_{11} & \lambda_{222} & \lambda_{111} & b_{11} & b_{21} & b_{12} & \kappa_2 & \kappa_1 & \epsilon\tau\\
  R_2 & -\lambda_{11} & -\lambda_{22} & \lambda_{111} & \lambda_{222} & -b_{11} & b_{12} & b_{21} & -\kappa_1 & -\kappa_2 & \tau\\ \hline
  \end{array}
 \label{generic-H3-change}
 \end{align}

 Setting $r_1 = \sgn(\lambda_{11}) \sqrt{|\lambda_{11} \lambda_{22}|}$, $r_2  = \sgn(\lambda_{11}) \l( \frac{|\lambda_{22}|}{|\lambda_{11}|} \r)^{1/4}$, we can normalize $\bar\alpha_{12} = \bar\theta_1$, $\bar\alpha_{21} = \epsilon \bar\theta_2$, where $\epsilon = \sgn(\lambda_{11} \lambda_{22}) = \pm 1$.  We refer to the corresponding 3-adapted moving frame as {\em normalized}.  The structure group is reduced to: (i) $\Z_2$, generated by $R_1$, if $\epsilon = 1$ ($M$ 2-elliptic), or (ii) the identity if $\epsilon = -1$ ($M$ 2-hyperbolic).
 
 \begin{prop}
 Any hyperbolic 2-generic surface $M \subset \Lg$ admits at most a 2-dimensional symmetry group. 
 \end{prop}

  Equations \eqref{generic-lambda3} reduce to $3\bar\alpha_{11} - \bar\alpha_{22} = 4\kappa_2 \bar\theta_1$, $3\bar\alpha_{22} - \bar\alpha_{11} = 4\kappa_1 \bar\theta_2$, or equivalently
 \[
  \bar\alpha_{11} = \frac{3}{2} \kappa_2 \bar\theta_1 + \frac{1}{2} \kappa_1 \bar\theta_2, \qquad
 \bar\alpha_{22} =\frac{1}{2} \kappa_2 \bar\theta_1 + \frac{3}{2} \kappa_1 \bar\theta_2,
 \]
 where
 \begin{align}
  \kappa_1 = \frac{\epsilon}{4} \bar\lambda_{222} = \frac{\sgn(\lambda_{22})\lambda_{222}}{4|\lambda_{11}|^{1/4}|\lambda_{22}|^{7/4}} , \qquad 
  \kappa_2 = \frac{1}{4} \bar\lambda_{111} = \frac{\sgn(\lambda_{11})\lambda_{111}}{4|\lambda_{11}|^{7/4}|\lambda_{22}|^{1/4}}.
  \label{kappa12}
 \end{align} 
 From \eqref{MC-gen1}-\eqref{MC-gen2}, we have
 $d\bar\theta_1 = - \kappa_1 \bar\theta_1 \wedge \bar\theta_2$,
 $d\bar\theta_2 =  \kappa_2 \bar\theta_1 \wedge \bar\theta_2$.
 Write $df = f_{,1} \bar\theta_1 + f_{,2} \bar\theta_2$ and similarly for iterated coframe derivatives.  The $d\alpha_{11},d\alpha_{22}$ equations in \eqref{MC-gen1}, \eqref{MC-gen2} imply
 \begin{align*}
 4(\kappa_{1,1} + \kappa_1 \kappa_2) \bar\theta_1 \wedge \bar\theta_2 &= d(3\bar\alpha_{22} - \bar\alpha_{11}) = (4\epsilon - 3\bar{b}_{21} - \bar{b}_{12}) \bar\theta_1 \wedge \bar\theta_2,\\
 -4(\kappa_{2,2} + \kappa_1 \kappa_2) \bar\theta_1 \wedge \bar\theta_2 &= d(3\bar\alpha_{11} - \bar\alpha_{22}) = -(4\epsilon - 3\bar{b}_{12} - \bar{b}_{21}) \bar\theta_1 \wedge \bar\theta_2,
 \end{align*}
 which implies
 \begin{align}
 \tau := \bar{b}_{11} = \frac{\sgn(\lambda_{11})b_{11}}{\sqrt{|\lambda_{11}\lambda_{22}|}}, \qquad
 \bar{b}_{12} = \frac{1}{2} \kappa_{1,1} - \frac{3}{2} \kappa_{2,2} - \kappa_1 \kappa_2 + \epsilon, \qquad
 \bar{b}_{21} = \frac{1}{2} \kappa_{2,2} - \frac{3}{2} \kappa_{1,1} - \kappa_1 \kappa_2 + \epsilon. \label{bij}
 \end{align}
 The $d\beta_1, d\beta_2$ equations in \eqref{MC-gen1}, \eqref{MC-gen2} yield the integrability conditions:
%
 \begin{align}
 \epsilon \tau_{,1} &= \bar{b}_{21,2} + 4\kappa_1 \bar{b}_{21} - 2\epsilon \kappa_2 \tau, \qquad \tau_{,2} = \bar{b}_{12,1} + 4\kappa_2 \bar{b}_{12} - 2\kappa_1 \tau. 
 \label{tau-int1}
 \end{align}
 The $\tau$ terms in $\kappa_2 \tau_{,2} - \kappa_1 \tau_{,1}$ cancel, hence the integrability condition $0 = d^2 \tau = -\tau_{,12} + \tau_{,21} - \kappa_1 \tau_{,1} + \kappa_2 \tau_{,2}$ becomes
 \begin{align}
 0 
 &= 2(\kappa_{2,2} - \kappa_{1,1}) \tau -\epsilon (\bar{b}_{21,22} + 4\kappa_{1,2} \bar{b}_{21} + 4\kappa_1 \bar{b}_{21,2} ) + \bar{b}_{12,11} + 4\kappa_{2,1} \bar{b}_{12} + 4\kappa_2 \bar{b}_{12,1} - 3\kappa_1 \tau_{,1} + 3\kappa_2 \tau_{,2} \nonumber\\
 &= -2(\kappa_{1,1} - \kappa_{2,2}) \tau + {\textsf T}(\kappa_1,\kappa_2,\epsilon) \label{tau-int2}
 \end{align}
 where $\textsf{T} = {\textsf T}(\kappa_1,\kappa_2,\epsilon)$ is a third order differential function of $\kappa_1,\kappa_2$ given by
 \begin{align*}
 \textsf{T} &=
  \bar{b}_{12,11} + 4\kappa_{2,1} \bar{b}_{12} + 7\kappa_2 \bar{b}_{12,1} + 12\kappa_2{}^2 \bar{b}_{12} - \epsilon (\bar{b}_{21,22} + 4\kappa_{1,2} \bar{b}_{21} + 7\kappa_1 \bar{b}_{21,2} + 12\kappa_1{}^2 \bar{b}_{21} ).
 \end{align*}
 We have $\kappa_{1,1} = \kappa_{2,2}$ iff $\bar{b}_{12} = \bar{b}_{21}$ iff $b_{12} = b_{21}$ which is $H_3$-invariant.  From \eqref{tau-int2}, we have:
 
 \begin{prop} If $\kappa_{1,1} \neq \kappa_{2,2}$, then $\tau$ is a third order function of $\kappa_1, \kappa_2$.
 \end{prop}

 \begin{defn} We refer to $\kappa_H{}^2 = (\kappa_1 + \kappa_2)^2$ as the {\em conformal mean-squared curvature}, $\kappa_G = \kappa_1 \kappa_2$ as the {\em conformal Gaussian curvature}, and $\tau$ as the {\em conformal torsion}.
 \end{defn}
 
 Note that $\kappa_H{}^2, \kappa_G$ recover $\kappa_1, \kappa_2$ up to sign and swap.
 By Theorem \ref{thm:Cartan-equiv}, our solution to the parametrized equivalence problem for hyperbolic 2-generic surfaces is:
  
 \begin{thm}[Invariants and equivalence for hyperbolic generic surfaces]  \label{thm:generic-param-soln} Let $U$ be connected and $i,\tilde{i} : U \ra \Lg$ hyperbolic 2-generic surfaces.  If $M = i(U)$ and $\tilde{M} = \tilde{i}(U)$ are $CSp(4,\R)$-equivalent, then for $M,M'$ on $U$: (i) $\epsilon$, $\kappa_H{}^2, \kappa_G$ agree, (ii) $\tau$ agree if $\epsilon = 1$; $\tau^2$ agree if $\epsilon = -1$.
 Conversely, if the above hold, then $M,\tilde{M}$ are $CSp(4,\R)$-equivalent.
 \end{thm}

  In Appendix \ref{sec:app-3-gen}, we show that $\{ (\kappa_1)^2, (\kappa_2)^2 \}$ and $\tau$ (if $\epsilon = 1$) or $\tau^2$ (if $\epsilon = -1$) are invariant under reparametrizations.  Hence, these solve the corresponding unparametrized equivalence problem.  Dropping bars over the objects associated with the normalized moving frame, we have the results in Table \ref{table:generic-summary}.  By Theorem \ref{thm:existence}, we have:
 
 \begin{thm}[Bonnet theorem for hyperbolic generic surfaces]  The integrability conditions in Table \ref{table:generic-summary} are the only local obstructions to the existence of a hyperbolic 2-generic surface with prescribed invariants.
 \end{thm}

  \begin{table}[h]
 \begin{center}
 \begin{tabular}{|c|c|} \hline
 Normalized moving frame structure equations & Integrability conditions\\ \hline \hline
 $\begin{array}{l}
 d\mf{0} = 2(\kappa_2\theta_1 + \kappa_1\theta_2)\mf{0} + \theta_1 \mf{1} + \theta_2 \mf{2} \nonumber\\
 d\mf{1} = (\tau \theta_1 + b_{12} \theta_2) \mf{0} - (\kappa_2 \theta_1 - \kappa_1\theta_2) \mf{1} - \theta_1 \mf{3} + \theta_2 \mf{4} \nonumber\\
 d\mf{2} = (b_{21} \theta_1 + \epsilon \tau \theta_2) \mf{0} + (\kappa_2 \theta_1 - \kappa_1\theta_2) \mf{2} - \epsilon\theta_2 \mf{3} + \theta_1 \mf{4} \nonumber\\
 d\mf{3} = - 2\epsilon\theta_2 \mf{1} - 2\theta_1  \mf{2} \nonumber\\
 d\mf{4} = (b_{21} \theta_1 + \epsilon \tau \theta_2) \mf{1} + (\tau \theta_1 + b_{12} \theta_2) \mf{2} - 2(\kappa_2\theta_1+\kappa_1\theta_2) \mf{4}\\
 \end{array}$ & 
  $\begin{array}{lcl}
 d\theta_1 = -\kappa_1 \theta_1 \wedge \theta_2,\\
 d\theta_2 = \kappa_2 \theta_1 \wedge \theta_2, \\
 \tau_{,1} = \epsilon (b_{21,2} + 4\kappa_1 b_{21}) - 2 \kappa_2 \tau \\
 \tau_{,2} = b_{12,1} + 4\kappa_2 b_{12} - 2\kappa_1 \tau\\
 0 = -2(\kappa_{1,1} - \kappa_{2,2}) \tau + {\textsf T}(\kappa_1,\kappa_2,\epsilon)
 \end{array}$ \\ \hline \hline
 Definitions & Fundamental invariants  \\ \hline\hline
 $\begin{array}{ll}\begin{array}{lccl}
 b_{12} = \frac{1}{2}\kappa_{1,1} - \frac{3}{2} \kappa_{2,2} - \kappa_1 \kappa_2 + \epsilon,
 && d\kappa_1 = \kappa_{1,1} \theta_1 +  \kappa_{1,2} \theta_2\\
 b_{21} = \frac{1}{2} \kappa_{2,2} - \frac{3}{2} \kappa_{1,1} - \kappa_1 \kappa_2 + \epsilon,
 && d\kappa_2 = \kappa_{2,1} \theta_1 +  \kappa_{2,2} \theta_2
 \end{array}\\
  \qquad\textsf{T} =
  b_{12,11} + 4\kappa_{2,1} b_{12} + 7\kappa_2 b_{12,1} + 12\kappa_2{}^2 b_{12}\\
 \qquad\qquad - \epsilon (b_{21,22} + 4\kappa_{1,2} b_{21} + 7\kappa_1 b_{21,2} + 12\kappa_1{}^2 b_{21} ).
 \end{array}$ & 
 $\begin{array}{c}
 \epsilon = \pm 1\\
 \tau \mbox{ (if $\epsilon = 1$)}, \quad \tau^2 \mbox{ (if $\epsilon = -1$)}\\
 \kappa_G = \kappa_1 \kappa_2 \\
 \kappa_H{}^2 = (\kappa_1 + \kappa_2)^2
 \end{array}$ \\ \hline
 \end{tabular}
 \caption{Invariant description of hyperbolic 2-generic surfaces}
 \end{center}
 \label{table:generic-summary}
 \end{table}
   
 \subsection{The conjugate manifold}
 \label{sec:conjugate-mfld}
 
 We express $\dim(M')$ in terms of $\kappa_1, \kappa_2, \tau$.  Let $\mf{}$ be any 3-adapted moving frame for $M$.  Since $\beta_3 = 0$, then by the $d\mf{4}$ equation in \eqref{1-str-eqs}, the tangent space $T_{[\mf{4}]} M' \subset T_{[\mf{4}]} \Q \cong \mf{4}^\perp / \ell_{\mf{4}}$ is spanned by the coefficients of $\theta_1$ and $\theta_2$, i.e.\
 \begin{align}
 \textsf{w}_1 = b_{21} \mf{1} + b_{11} \lambda_{11} \mf{2} \qbox{and} \textsf{w}_2 = b_{11} \lambda_{22} \mf{1} + b_{12} \mf{2} \,\,\mod \mf{4}.
 \label{conj-basis}
 \end{align}
 
 \begin{prop}  \label{prop:conjugate-point} $\dim(M')= \rank\mat{cc}{ \epsilon \tau & \bar{b}_{12} \\ \bar{b}_{21} & \tau }$, where $\bar{b}_{12}, \bar{b}_{21}, \tau$ were defined in \eqref{bij}.
 \end{prop}
 
 \begin{proof}
 $\dim(M')= \rank\mat{cc}{ b_{11} \lambda_{22} & b_{12} \\ b_{21} & b_{11} \lambda_{11} }$.  Since $\dim(M')$ is geometric, use the normalized moving frame.
 \end{proof}
 
 \begin{cor} $\dim(M') = 0$ iff $\tau = 0$ and $\kappa_{1,1} = \kappa_{2,2} = \epsilon - \kappa_1 \kappa_2$.
 \end{cor}
 \begin{example} While every singly-ruled hyperbolic surface with $\dim(M') = 0$ is homogeneous, the analogous assertion in the generic hyperbolic case is false: The integrability conditions reduce to:
 \begin{align}
 \kappa_{1,1} = \kappa_{2,2} = \epsilon - \kappa_1 \kappa_2, \qquad d\theta_1 = -\kappa_1 \theta_1 \wedge \theta_2, \qquad d\theta_2 = \kappa_2 \theta_1 \wedge \theta_2. \label{generic-dim0}
 \end{align}
 Since $\{ \theta_1 \}$ and $\{ \theta_2 \}$ are both Frobenius, introduce coordinates $u,v$ such that $\theta_1 = f du$, $\theta_2 = gdv$, where $f,g$ are nonvanishing functions of $u,v$.  The $\theta_1,\theta_2$ equations in \eqref{generic-dim0} yield $\kappa_1 = \frac{f_v}{fg}$, $\kappa_2 = \frac{g_u}{fg}$.  The remaining equation yields
 \begin{align*}
 \frac{1}{f} \l( \frac{f_v}{fg} \r)_u = \frac{1}{g} \l( \frac{g_u}{fg}\r)_v = \epsilon - \frac{f_v g_u}{f^2 g^2}
 \quad\Longleftrightarrow\quad & F_{uv} = G_{uv} = \epsilon e^{F+G}, \qbox{where} F = \ln|f|, \quad G = \ln|g|.
 \end{align*}
 Thus, $G = F + \phi(u) + \psi(v)$.  Reparametrizing $\tilde{u} = \tilde{u}(u)$ and $\tilde{v} = \tilde{v}(v)$, we can without loss of generality assume that $G=F$.  We should solve $F_{uv} = \epsilon e^{2F}$ which is almost the Liouville equation $z_{xy} = e^z$.  By means of Backl\"und transformations, the latter has the well-known general solution $z = \ln\l( \frac{2 X'(x) Y'(y)}{(X(x) + Y(y))^2} \r)$, which we use to get the general solution $F = \frac{1}{2}\ln\l( \frac{U'(u) V'(v)}{(U(u)+V(v))^2} \r)$ of the former in the case $\epsilon = 1$.  Set $(U(u),V(v)) = (u^3,v^3)$, so $F = G = \frac{1}{2} \ln\l( \frac{9u^2 v^2}{(u^3+v^3)^2} \r)$.  Hence, $f = g = 3\sqrt{\frac{u^2 v^2}{(u^3+v^3)^2}}$ and $\kappa_1 = \frac{v^3 - 2u^3}{3u^2 v}$, $\kappa_2 = \frac{u^3 - 2v^3}{3uv^2}$.  Since all integrability conditions are satisfied, a surface $M$ locally exists with $\kappa_1,\kappa_2$ as given and $\tau=0$.  Since $\kappa_1,\kappa_2$ are nonconstant, $M$ is nonhomogeneous.
 \end{example}
 
 If $\dim(M') = 1$, then $\tau = 0$ and exactly one of $b_{12}$ or $b_{21}$ is zero.  In this case, $M'$ is a null curve.
 If $\dim(M') = 2$, then $b_{11}{}^2 \lambda_{11} \lambda_{22} - b_{12} b_{21} \neq 0$.  The conformal structure on $M'$, expressed in terms of \eqref{conj-basis}, is $\mat{cc}{b_{21} b_{11} \lambda_{11} & b_{12} b_{21} + b_{11}{}^2 \lambda_{11} \lambda_{22} \\ b_{12} b_{21} + b_{11}{}^2 \lambda_{11} \lambda_{22} & b_{11} b_{12} \lambda_{22} }$.  Its determinant is $-(b_{11}{}^2 \lambda_{11} \lambda_{22} - b_{12} b_{21})^2 < 0$, so $M'$ is hyperbolic.
 Now classify $M'$ at 2nd order (c.f. Table \ref{2-classification}).  For any 3-adapted moving frame $\mf{}$ on $M$, $\mfp{} = (\mfp{0}, \mfp{1}, \mfp{2}, \mfp{3}, \mfp{4}) = (\mf{4},\mf{2},\mf{1},\mf{3},\mf{0})$ is at least 1-adapted  on $M'$.  Writing out the structure equations for $\mfp{}$ and comparing with \eqref{1-str-eqs}, the corresponding MC forms for the adaptation to $M'$ are related to those for the adaptation to $M$ by:
 \[
 \begin{array}{l} \theta_3' = 0,\\ \beta_3' = 0, \end{array}
 \qquad 
 \begin{array}{l} \alpha_{11}' = -\alpha_{11},\\ \alpha_{22}' = -\alpha_{22}, \end{array}
 \qquad
 \begin{array}{l} \alpha_{12}' = \alpha_{21}, \\ \alpha_{21}' = \alpha_{12}, \end{array}
 \qquad
 \begin{array}{l} \theta_1' = \beta_1,\\ \theta_2' = \beta_2, \end{array}
 \qquad 
 \begin{array}{l} \beta_1' = \theta_1, \\ \beta_2' = \theta_2. \end{array}
 \]
 Inverting \eqref{generic-beta12}, we have $\mat{c}{\theta_1\\ \theta_2} = \displaystyle\frac{1}{b_{11}{}^2 \lambda_{11} \lambda_{22} - b_{12} b_{21}} \mat{cc}{ b_{11} \lambda_{22} & -b_{12} \\ -b_{21} & b_{11} \lambda_{11}} \mat{c}{ \beta_1 \\ \beta_2 }$.  Writing $\alpha_{12}' = \lambda_{11}' \theta_1' + \lambda_{12}' \theta_2' = \frac{\lambda_{22}}{b_{11}{}^2 \lambda_{11} \lambda_{22} - b_{12} b_{21}} (-b_{21} \theta_1' + b_{11} \lambda_{11} \theta_2')$ and $\alpha_{21}' = \lambda_{12}' \theta_1' + \lambda_{22}' \theta_2' =  \lambda_{11} \theta_1 = \frac{\lambda_{11}}{b_{11}{}^2 \lambda_{11} \lambda_{22} - b_{12} b_{21}} (b_{11}\lambda_{22} \theta_1' - b_{12}  \theta_2')$, we obtain
 \[
 \lambda_{11}' = \frac{-\lambda_{22} b_{21}}{b_{11}{}^2 \lambda_{11} \lambda_{22} - b_{12} b_{21}}, \qquad
 \lambda_{12}' = \frac{b_{11} \lambda_{11}\lambda_{22}}{b_{11}{}^2 \lambda_{11} \lambda_{22} - b_{12} b_{21}}, \qquad
 \lambda_{22}' = \frac{-\lambda_{11} b_{12}}{b_{11}{}^2 \lambda_{11} \lambda_{22} - b_{12} b_{21}}
 \]
 or
 \[
 \lambda_{11}' = \frac{r_2{}^2}{r_1} \l( \frac{-\epsilon \bar{b}_{21}}{\epsilon \tau^2 - \bar{b}_{12} \bar{b}_{21}} \r), \qquad
 \lambda_{12}' = \frac{1}{r_1} \l( \frac{\epsilon\tau}{\epsilon \tau^2 - \bar{b}_{12} \bar{b}_{21}} \r), \qquad
 \lambda_{22}' = \frac{1}{r_1 r_2{}^2} \l( \frac{-\bar{b}_{12}}{\epsilon \tau^2 - \bar{b}_{12} \bar{b}_{21}} \r)
 \]
 where $r_1, r_2$ leading to the normalized frame were defined after \eqref{generic-H3-change}.  The CSG for $M'$ is given by $\mfp{c} = \mfp{3} + 2\lambda_{12}' \mfp{0} = \mf{3} + 2\lambda_{12}' \mf{4}$.  Thus, $\mfp{}$ is 2-adapted iff $\tau = 0$.  Since $\beta_3' = 0$, $\mfp{}$ is moreover 3-adapted for $M'$.  Hence, $(M')' = M$.  
 
  \begin{thm}
 If $M$ is hyperbolic 2-generic surface and $\dim(M') = 2$, then $M'$ is also hyperbolic.
 Moreover, $(M')' = M$ iff $\tau = 0$ iff the central tangent spheres to $M$ and $M'$ agree at conjugate points.
 \end{thm}

 From Section \ref{sec:2-adaptation}, we see that changing $\mfp{3}$ to $\mfp{c}$ does not affect $\lambda_{11}',\lambda_{22}'$.   Hence, $(\lambda_{11}',\lambda_{22}')$ classify $M'$ at 2nd order.  Taking into the differential order of $b_{ij}$, this is a 4th order classification of $M$.
 \begin{table}[h]
 \begin{center}
 \begin{tabular}{|c|c|} \hline
 2nd order classification of $M'$ & Invariant characterization\\ \hline\hline
 indefinite sphere & $\tau \neq 0$ and $\bar{b}_{12} = \bar{b}_{21} = 0$ \\ \hline
 singly-ruled & $\tau \neq 0$ and exactly one of $\bar{b}_{12}$ or $\bar{b}_{21}$ is zero \\ \hline
 2-hyperbolic & $\epsilon \tau^2 \neq \bar{b}_{12} \bar{b}_{21}$, \quad $\epsilon \bar{b}_{12} \bar{b}_{21} < 0$ \\ \hline
 2-elliptic &  $\epsilon \tau^2 \neq \bar{b}_{12} \bar{b}_{21}$, \quad $\epsilon \bar{b}_{12} \bar{b}_{21} > 0$\\ \hline
 \end{tabular}
 \caption{Fourth order classification of hyperbolic 2-generic $M$ when $M'$ is a hyperbolic surface}
  \label{conjugate-2-classification}
 \end{center}
 \end{table}

 \begin{rem}
 Examples of 2-generic CSI surfaces $M$ with $M'$ 2-elliptic and 2-hyperbolic will be given in Section \ref{sec:generic-CSI}.  However, it is unclear whether the first two cases ($M'$ an indefinite sphere or $M'$ singly-ruled) are non-vacuous.  
 \end{rem}

 \subsection{Contact spheres and Dupin cyclides}
 
 In the 2-elliptic case, $\lambda_{11} \lambda_{22} > 0$ for any 2-adapted moving frame $\mf{}$, which we moreover assume is 3-adapted.  In this case $M$ has no asymptotic directions, but carries a net of curvature lines.  From $a_{ij} = diag(\lambda_{11},\lambda_{22})$ and $\metric_{ij} = \mat{cc}{0 & 1\\ 1& 0}$, the eigenvalues and eigendirections of $a^i_k = \mu^{ij} a_{jk} = \mat{cc}{0 & \lambda_{22}\\ \lambda_{11} & 0}$ are $s_\pm = \pm \sqrt{\lambda_{11} \lambda_{22}}$ and $\mf{\pm} = \lambda_{22} \mf{1} + s_\pm \mf{2}$.  From the $d\mf{0}$ structure equation, these eigendirections (modulo $\mf{0}$) are respectively given by the vanishing of $\theta_\pm = \lambda_{22} \theta_2 + s_\mp  \theta_1$.  The {\em contact spheres} are $B_\pm = \mf{3} + 2s_\pm \mf{0}$.  Under the action of $H_3$, we have that $[\bar{B}_\pm] = [B_\pm]$ or $[B_\mp]$.  Thus, the pair $([B_+],[B_-])$ is geometric.  Differentiating,
 \begin{align*}
 dB_\pm &= d\mf{3} + 2(ds_\pm) \mf{0} + 2s_\pm d\mf{0} \\
 &= -2\alpha_{21} \mf{1} - 2\alpha_{12} \mf{2}
 	+ \frac{d(\lambda_{11} \lambda_{22})}{s_\pm} \mf{0} 
	+ 2 s_\pm [(\alpha_{11} + \alpha_{22})\mf{0} + \theta_1 \mf{1} + \theta_2 \mf{2}]\\
 &= -2\lambda_{22} \theta_2 \mf{1} - 2\lambda_{11} \theta_1 \mf{2} 
 	+ \frac{1}{s_\pm} \l[\lambda_{11} d\lambda_{22} + \lambda_{22} d\lambda_{11} \r]\mf{0} 
	 + 2s_\pm [(\alpha_{11} + \alpha_{22})\mf{0} + \theta_1 \mf{1} + \theta_2 \mf{2}]\\
 &= -\frac{2}{\lambda_{22}} \theta_\pm \mf{\mp}
 	+ \frac{1}{s_\pm} \l[\lambda_{11} \lambda_{222} \theta_2 + \lambda_{22} \lambda_{111} \theta_1 \r]\mf{0}\\
 &= -\frac{2}{\lambda_{22}} \theta_\pm \mf{\mp}
 	+ \frac{1}{2s_\pm} \l[\frac{\lambda_{11} \lambda_{222}}{\lambda_{22}} (\theta_+ + \theta_-) + \frac{\lambda_{22} \lambda_{111}}{s_-} (\theta_+ - \theta_-) \r]\mf{0}
 \end{align*}
 As the 1-form coefficients of $\mf{0}$ and $\mf{\pm}$ in $dB_\pm$ are linearly independent, the congruence determined by each of $[B_\pm]$ depends on 2 parameters.  However, along the curvature directions determined by $\theta_+$ and $\theta_-$, we have
 \begin{align*}
 \theta_+ = 0: \quad  dB_+ =  \frac{a_+}{2s_+}\theta_- \mf{0}; \qquad
 \theta_- = 0: \quad dB_- =  \frac{a_-}{2s_-} \theta_+ \mf{0}, \quad \qbox{where}  a_\pm = \frac{\lambda_{11} \lambda_{222}}{\lambda_{22}} \pm \frac{\lambda_{22} \lambda_{111}}{s_+}.
 \end{align*}
 We have
 \begin{align*}
 &a_+ = 0 \qRa dB_+  = \l( -\frac{2}{\lambda_{22}} \mf{-}
 	+ \frac{\lambda_{11} \lambda_{222}}{\lambda_{22} s_+} \mf{0}\r) \theta_+ \qRa dB_+ = 0 \qbox{along} \theta_+ = 0,\\
 &a_- = 0 \qRa dB_-  = \l( -\frac{2}{\lambda_{22}} \mf{+}
 	+ \frac{\lambda_{11} \lambda_{222}}{\lambda_{22} s_-}  \mf{0} \r)  \theta_-\qRa dB_- = 0 \qbox{along} \theta_- = 0.
 \end{align*}
 If $a_+ = 0$ or $a_- = 0$, then $B_+$ or $B_-$ respectively depend on only 1 parameter.  Since we always have $[\mf{0}] \in \S{B_\pm}$, then $M$ is the envelope of a 1-parameter family of (indefinite) spheres and $M$ is called a {\em canal surface}.
 
 \begin{defn}
 If $M$ is the envelope of two 1-parameter families of spheres, then $M$ is called a {\em Dupin cyclide}.
 \end{defn}
 
 We see from above that $M$ is a Dupin cyclide iff $a_+ = a_- = 0$.  This condition is equivalent to $\lambda_{111} = \lambda_{222} = 0$, or $\kappa_1 = \kappa_2 = 0$.  By the integrability conditions, $\tau_{,1} = \tau_{,2} = 0$, so $\tau$ is constant.  By Theorem \ref{thm:homogeneity}, we have:
 
 \begin{thm}
 A hyperbolic surface $M \subset \Lg$ is a Dupin cyclide iff $M$ is 2-elliptic with $\kappa_1 = \kappa_2 = 0$ and $\tau \in \R$ constant.  Moreover, distinct $\tau \in \R$ correspond to $CSp(4,\R)$-inequivalent classes of Dupin cyclides.  Any Dupin cyclide $M$ is a CSI surface with $\dim(M') \neq 0$: (1) $\tau = \pm 1$ iff $\dim(M') = 1$, and (2) $\tau \neq \pm 1$ iff $\dim(M') = 2$.
 \end{thm}
 
 \begin{rem}
 It is well-known that all Dupin cyclides in Euclidean signature are inversions of the standard torus.  Since inversions generate the conformal group, then there is only a single equivalence class in that setting.  In Lorentzian signature, there are infinitely many non-equivalent Dupin cyclides.
 \end{rem}
 
  \begin{example}
 In the next section, we show for $rt = -1$ that $\kappa_1 = \kappa_2 = 0$ and $\tau = 2$, so this is a Dupin cyclide $M$ with $\dim(M') = 2$.  $(r,s,t) = ( -e^{-(u+v)}, u-v, e^{u+v})$ is a null parametrization.  We have $\uvprod = 4$ and $I_1 = I_2 = I_3 = -4$, which implies $\lambda_{11} \lambda_{22} = \frac{I_1 I_2}{\uvprod^4} = \frac{1}{16}$. From \eqref{N}, we have $ \normal = \l(0,-2e^{-(u+v)}, 0, 2e^{u+v}, -4\r)$.  From \eqref{3-lift-generic}, a 3-adapted moving frame is
 \begin{align*}
 \mf{0} &= \x = (1, -e^{-(u+v)}, u-v, e^{u+v}, -1 - (u-v)^2) \\
 \mf{1} &= \x_u = (0, e^{-(u+v)}, 1, e^{u+v}, - 2(u-v)), \\
 \mf{2} &= \frac{1}{4} \x_v = \frac{1}{4} (0, e^{-(u+v)}, -1, e^{u+v}, 2(u-v)), \\
 \mf{3} &= \frac{\iota}{2}\l( \normal - {\bf x} \r) = -\frac{\iota}{2}\l( 1,e^{-(u+v)}, u-v, -e^{u+v}, 3 - (u-v)^2 \r), \\
 \mf{4} &= {\bf Z} + \frac{1}{8} \normal - \frac{1}{16} \x = \frac{-1}{16} \l(1, 3 e^{-(u+v)}, u-v, -3e^{u+v}, -9 - (u-v)^2 \r).
 \end{align*}
 The points $[\mf{1}], [\mf{2}]$ are located on the sphere at infinity $\S{{\bf Z}}$ in the coordinate chart $\psi$.  From $[\mf{4}]$, the conjugate surface $M'$ has equation $rt = -9$, which has $\kappa_1 = \kappa_2 = 0$, $\tau = 2$, so is again a Dupin cyclide and $M'$ is $CSp(4,\R)$-equivalent to $M$.  The contact spheres are  $\S{{\bf B_1}}, \S{{\bf B_2}}$, where
 \begin{align*}
 [{\bf B_1}] &= [0,e^{-(u+v)},0,-e^{u+v},2] &&\qRa \S{{\bf B_1}} :\quad r e^{u+v}-t e^{-(u+v)} + 2 = 0,\\
 [{\bf B_2}] &= [ 1,0, u-v, 0, 1 - (u-v)^2 ] &&\qRa \S{{\bf B_2}} :\quad rt-s^2 + 2s(u-v)+1-(u-v)^2 = 0.
 \end{align*} 
 \begin{center}
 \includegraphics[scale=0.4]{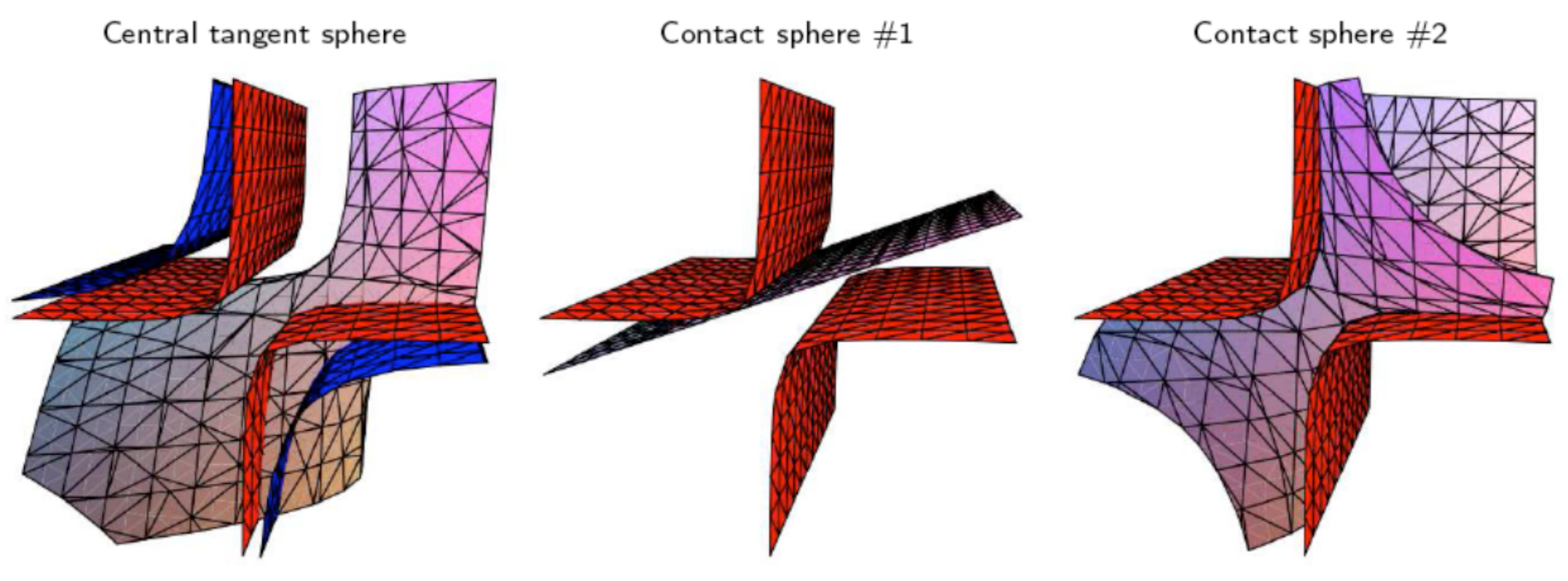}
 \end{center}
 More generally, one can show that the conjugate surface to $rt = -k^2$ is $rt = -9k^2$.  Hence, for $M$ above, $(M')' \neq M$.
 \end{example}

 \subsection{2-generic CSI surfaces and other examples}
 \label{sec:generic-CSI}
 
 Let $M$ be 2-generic CSI, i.e. $\kappa_1,\kappa_2, \tau \in \R$ are constant, so $\bar{b}_{12} = \bar{b}_{21} = \epsilon - \kappa_1 \kappa_2$.  The remaining integrability conditions \eqref{tau-int1}-\eqref{tau-int2} reduce to
 \begin{align*}
 0 = 2\kappa_1 \bar{b}_{12} - \epsilon \kappa_2 \tau, \qquad
 0 = 2\kappa_2 \bar{b}_{12} - \kappa_1 \tau, \qquad
 0 = \textsf{T} = 12(\kappa_2{}^2 - \epsilon \kappa_1{}^2 ) \bar{b}_{12}.
 \end{align*}
 Solving these equations and using Proposition \ref{prop:conjugate-point} and the 2nd order classification of $M'$ in Section \ref{sec:conjugate-mfld}, we have:
 \begin{align}
 \begin{array}{|c|c|c|c|c|c|c|c|} \hline
 \dim(M') & \epsilon & \kappa_1 &\kappa_2 & \tau & \mbox{Inequivalent cases} & \mbox{Sub-classification of } M'\\ \hline\hline
 0 & \pm1 & \neq 0 & \frac{\epsilon}{\kappa_1} & 0 & 0 < \kappa_1 \leq 1 & -\\
 1 & 1 & 0 & 0  & \pm 1 & \mbox{all} & -\\
 2 & 1 & 0 & 0 & \neq \pm 1 & \mbox{all}&\mbox{2-elliptic}\\
 2 & -1 & 0 & 0 & \in \R & \tau \geq 0&\mbox{2-hyperbolic}\\
 2 & 1 & \neq 0 & \kappa_1 & 2(1 - \kappa_1{}^2) \neq 0 & 0 < \kappa_1 \neq 1 & \mbox{2-elliptic}\\
 2 & 1 & \neq 0 & -\kappa_1& -2(1 + \kappa_1{}^2) & \kappa_1 > 0 & \mbox{2-elliptic}\\ \hline
 \end{array} \label{2-gen-CSI}
 \end{align}
 

 \begin{rem} Comparing with Table 4 in \cite{The2008} we see that $(\epsilon,m,n,B)$ corresponds to $(\epsilon,\kappa_1,-\kappa_2,\tau)$ here.
 \end{rem}
 
 \begin{example} Let $\epsilon = \pm 1$ and $m \in \R$ nonzero.  Consider the surfaces $M = M_{\epsilon,m}^\pm$ given by 
 \begin{align}
 r = -\frac{1}{3} (\epsilon m u^3 + v^3),\quad
 s = -\frac{1}{2} (\epsilon m^2 u^2 - v^2), \quad
 t = -(\epsilon m^3 u + v). \label{maxsym}
 \end{align}
 Here, $M_{\epsilon,m}^+$ corresponds to $u+mv > 0$ and $M_{\epsilon,m}^-$ corresponds to $u-mv<0$.  Let $\x = (1,r,s,t,rt-s^2)$.  Then $\ambient{\x_u}{\x_u} = \ambient{\x_v}{\x_v} = 0$ and $\uvprod = \epsilon m (u + mv)^2$. We calculate $I_1 = -m^3(mv+u)^2$, $I_2 = -\epsilon m(mv+u)^2$, $I_3 = 0$, hence $I_1 I_2 = \epsilon m^4(mv+u)^4$, so $\sgn(I_1 I_2) = \epsilon \neq 0$, hence such surfaces are 2-elliptic if $\epsilon = 1$ and 2-hyperbolic if $\epsilon = -1$. Define $\varsigma = \sgn(mv+u)$.  Using \eqref{kappa12tau}, we find that
  \begin{align*}
  \kappa_1 = -\varsigma\epsilon |m|, \qquad
  \kappa_2 = \frac{-\varsigma }{|m|}, \qquad
  \tau = 0 \qRa \kappa_1\kappa_2 = \epsilon \qRa \dim(M') = 0.
 \end{align*}
 From \eqref{3-lift-generic}, the conjugate point is located at $[\mf{4}] = [{\bf Z}] = [0,0,0,0,1]$, i.e. the center of the sphere at infinity.
 Note $(\kappa_1)^2 = m^2$, $(\kappa_2)^2 = \frac{1}{m^2}$ are independent of $c$ and $\sgn(m)$.  By Theorem \ref{thm:generic-param-soln}, $M_{\epsilon,m}^+$ is $CSp(4,\R)$-equivalent to any of $M_{\epsilon,m}^-$, $M_{\epsilon,-m}^+$, $M_{\epsilon,\frac{1}{m}}^+$.  A minimal list of inequivalent surfaces is given by all $M_{\epsilon,m}^+$ such that $\epsilon = \pm1$, $m \in (0,1]$.
 
 Under the substitution $(r,s,t) = (z_{xx},z_{xy},z_{yy})$, these surfaces correspond to the class of maximally symmetric hyperbolic PDE of generic type and were studied in \cite{The2008}, \cite{The2010}.  Define $\alpha = 1 - \epsilon m^4$.  The PDE given by \eqref{maxsym} was shown in \cite{The2008} to be contact-equivalent to 
 \begin{align}
 \alpha = 0: \quad 3rt^3 + 1 = 0; \qquad \alpha \neq 0: \quad\frac{(3r - 6st + 2t^3)^2}{(2s-t^2)^2} = c(m) := (1 + \epsilon m^4)\l( 1 + \frac{\epsilon}{m^4} \r). \label{maxsym-implicit}
 \end{align}
 Examining the calculation, we see that the contact transformations used resulted in vertical transformations on the fibres of $\pi^2_1 : J^2 \ra J^1$.  On these fibres $J^2|_\jet$, $CSp(C_\jet,[\eta]) \cong CSp(4,\R)$ transformations were used.  Thus, the surfaces \eqref{maxsym} and \eqref{maxsym-implicit} are $CSp(4,\R)$-equivalent.
 We note that $c(m) = c(-m) = c(\frac{1}{m})$.  Thus, $\epsilon = \pm 1$, $m \in (0,1]$ suffices and this is the parameter range specified in \cite{The2008}.
 \end{example}

 From Proposition \ref{prop:red-var}, let us examine hyperbolic 2-generic surfaces $M$ given by $F(r,t) = 0$ with $F_r F_t < 0$.  By the implicit function theorem, we may assume that $F$ is given locally by $F(r,t) = r - f(t) = 0$ with $f'(t) > 0$.  By Lemma \ref{lem:conf-flat}, a null parametrization exists, which we write as
 \[
 (r,s,t) = (f(t(u,v)),s(u,v),t(u,v)): \qquad f'(t)(t_u)^2 = (s_u)^2, \quad f'(t)(t_v)^2 = (s_v)^2, \quad \Gamma = \det\mat{cc}{s_u & t_u \\ s_v & t_v} \neq 0.
 \]
 None of $s_u,s_v,t_u,t_v$ can vanish anywhere.  Let $F(t)$ be an antiderivative of $\sqrt{f'(t)}$.  Using the $CSp(4,\R)$ transformation $(r,s,t) \ra (r,-s,t)$ (induced from $(x,y,z) \ra (-x,y,z)$) and interchanging $u,v$ if necessary, we may assume $\sqrt{f'(t)} t_u = s_u$, $\sqrt{f'(t)} t_v = -s_v$.  (If instead $\sqrt{f'(t)} t_v = s_v$, then $F(t) = s$ and so $(r,s,t) = (f(t), F(t), t)$, which is a curve.)  From the first equation, we have $F(t) = s + g(v)$.  Differentiation yields $\sqrt{f'(t)} t_v = s_v + g'(v)$ and substitution in the second yields $-2s_v = g'(v)$, so $s_{uv} = 0$.  Write $s = s_1(u) - s_2(v)$.  Since $s_u,s_v \neq 0$, then we can reparametrize so $s = u - v$.  Hence, 
 \[
 t_u = t_v = \frac{1}{\sqrt{f'(t)}}, \qquad \Gamma = t_v + t_u = \frac{2}{\sqrt{f'(t)}} \neq 0
 \]
 Integration yields $F(t) = u+v$, so that $t = F^{-1}(u+v)$.  This makes sense since $F'(t) \neq 0$, so $F$ is 1-1.  Thus, our parametrization is $(r,s,t) = (f(F^{-1}(u+v)), u-v, F^{-1}(u+v))$.
 We have 
 
 We have $\uvprod = 4$ and $I_1 = I_2 = I_3 = \frac{2f''(t)}{(f'(t))^{3/2}}$.  Since $I_1 I_2 > 0$, $M$ is 2-elliptic, c.f. Prop \ref{prop:red-var}.  From \eqref{kappa12tau}:
 \begin{align}
   \kappa_1 = \kappa_2 = \sgn(f''(t))\frac{ 4 f'(t)^2 }{(f''(t))^2} [Sf](t), \qquad
     \tau = 2 -  \frac{16 f'(t)^4}{f''(t)^3} \l( \frac{[Sf](t)}{f'(t)} \r)' \label{2-gen-rt}
 \end{align}
 where $[Sf](t) = \frac{f'''(t)}{f'(t)} - \frac{3(f''(t))^2}{2f'(t)^2}$ is the Schwarzian derivative of $f$.

  \begin{example}  The following are examples of hyperbolic 2-elliptic CSI surfaces $M$ with $\dim(M') = 2$, c.f. \eqref{2-gen-CSI}:
 (i) $r = \frac{1}{3} t^3$ with $t > 0$; $\kappa_1 = \kappa_2 = -4$, $\tau = -30$, and (ii) $r = e^t$: $\kappa_1 = \kappa_2 = -2$, $\tau = -6$.
 After some calculations, one can show that the conjugate surface $M'$ is: (i) $r = \frac{625}{2187} t^3$, and (ii)  $r = 9e^{t - \frac{8}{3}}$ .
 \end{example}

  The Schwarzian vanishes on linear fractional transformations, hence $f(t) = -\frac{1}{t}$ satisfies $[Sf](t) = 0$ and $\kappa_1 = \kappa_2 = 0$, $\tau = 2$ from \eqref{2-gen-rt}.  Hence, $M = \{ rt = -1 \}$ is the simplest example of a Dupin cyclide.  It has $\dim(M') = 2$.

 \section{Conclusions}
 \label{sec:conc}

 As described in Section \ref{sec:inv-LG-PDE}, our $CSp(4,\R)$-invariant study of hyperbolic surfaces in $\Lg$ yields contact-invariant information for hyperbolic PDE in the plane.  Our view of PDE from this perspective has led to a new and simple proof of contact-invariance of the \MA/ PDE, a geometric interpretation of of elliptic, parabolic, hyperbolic PDE, a reinterpretation of class 6-6, 6-7, 7-7 hyperbolic PDE as the 2nd order classification of hyperbolic surfaces in $\Lg$, as well as many new higher order contact invariants for hyperbolic PDE.  Certainly our emphasis here has been on surface theory since studying surfaces in a 3-dimensional space is much simpler than working on a 7-manifold in the PDE setting.  However, much work still remains to be done on studying the specific geometry and solution methods associated with each contact-invariant class of PDE which we have identified.  We conclude with several questions worth investigating:
 \begin{enumerate}
 \item Carry out a similar study for curves, elliptic (spacelike), and parabolic (null) surfaces in $\Lg$.  Relate the study of null surfaces and null curves to Cartan's famous ``five variables'' paper \cite{Cartan1910}.
 \item For each $CSp(4,\R)$-invariant class of hyperbolic surfaces, find coframe structure equations and investigate the geometry of the corresponding classes of hyperbolic PDE.
 \item Study submanifold theory in general \Lag/s $LG(n,2n)$ modulo $CSp(2n,\R)$.  This study will be significantly more difficult than in $\Lg$ as there is no longer a connection to conformal geometry.  The first-order distinguished structure on the tangent spaces of $LG(n,2n)$ is now a degree $n$ cone instead of a quadratic cone as in the conformal case.\\
 \end{enumerate}

\noindent {\bf Acknowledgements}\\

 We thank Igor Zelenko for pointing out that the fibres of $\pi^2_1 : J^2(\R^n,\R) \ra J^1(\R^n,\R)$ are diffeomorphic to $LG(n,2n)$.  This was the initial seed that grew into this work.  We also thank Niky Kamran and Abraham Smith for fruitful discussions and encouragement.
 We gratefully acknowledge financial support from the National Sciences and Engineering Research Council of Canada in the form of an NSERC Postdoctoral Fellowship.
 
 \appendix

\section{Parametric description of moving frames}
 \label{app:param}
 
 \subsection{Null parametrization and differential syzygies}

 Given a smooth hyperbolic surface $M \subset \Lg$ with basepoint $o$, we may (using the $CSp(4,\R)$-action) assume $o = \{ e_1,e_2\}$.  Let $(r,s,t)$ be standard coordinates about $o$ corresponding to $[1,r,s,t,rt-s^2] \in \Q$, c.f. Section \ref{sec:prelim}, expressed with respect to the basis $\mathcal{B}$ \eqref{B-basis-scalar-prod} on which $\ambient{\cdot}{\cdot}$ has matrix \eqref{B-basis-scalar-prod}.  Let $(r(u,v),s(u,v),t(u,v))$ be a local parametrization $i: U \ra M$ and let $\x = \x(u,v) = (1,r,s,t,rt-s^2)$.
 Since $\ambient{\x}{\x} = 0$, then $\ambient{\x}{\x_u} = \ambient{\x}{\x_v} = 0$, where
 $\x_u = (0,r_u,s_u,t_u,r_u t + rt_u - 2ss_u)$,
 $\x_v = (0,r_v,s_v,t_v,r_v t + rt_v - 2ss_v)$. 
 In order to substantially simplify our later formulas we will assume that $u,v$ are {\em null coordinates}, the local existence of which is guaranteed:
  
 \begin{lem}[Existence of null coordinates \cite{Weinstein1995}] \label{lem:conf-flat}
 Let $(S,\metric)$ be a smooth surface with a Lorentzian metric $\metric$.  Given any $p \in S$, there exist (smooth) coordinates $u,v$ such that $\partial_u, \partial_v$ are null vectors and $\metric = f du dv$ with $f > 0$.
 \end{lem}
 
 The $u,v$ coordinate lines are null curves with respect to $[\metric]$, where $\metric = dr \odot dt - ds \odot ds$.  Analytically, 
 \begin{align}
 \ambient{\x_u}{\x_u} = 2(r_u t_u - s_u{}^2) = 0, \qquad
 \ambient{\x_v}{\x_v} = 2(r_v t_v - s_v{}^2) = 0. \label{null-relations}
 \end{align}
 Since $[\metric]$ on $M$ is non-degenerate, then $\uvprod := \ambient{\x_u}{\x_v} = r_u t_v + r_v t_u - 2 s_u s_v \neq 0.$
Define $ {\bf Z} = (0,0,0,0,1)$ and
 \begin{align}
 \normal = \l(0, \det\mat{cc}{r_u & s_u\\ r_v & s_v}, \frac{1}{2} \det\mat{cc}{r_u & t_u\\ r_v & t_v}, \det\mat{cc}{s_u & t_u\\ s_v & t_v}, \det\mat{ccc}{r & s& t \\ r_u & s_u & t_u\\ r_v & s_v & t_v}\r),
 \label{N}
 \end{align}
 which is like a Lorentzian cross-product.
 We have $\normal  \in \{ \x, \x_u, \x_v, {\bf Z} \}^\perp$.  Using  \eqref{null-relations}, we find $\uvprod_N := \sqrt{ - \frac{1}{2}\ambient{\normal}{\normal} } = \frac{1}{2} \uvprod$ so that $\ambient{\normal}{ \normal} = -2 (\uvprod_N)^2$.
 Define $I_1 = \ambient{\normal}{\x_{uu}}$, $I_2 = \ambient{\normal}{\x_{vv}}$, $I_3 = \ambient{\normal}{\x_{uv}}$, which can also be written
 \begin{align}
 I_1 = \det\mat{ccc}{ r_{uu} & s_{uu} & t_{uu} \\ r_u & s_u & t_u \\ r_v & s_v & t_v}, \qquad
 I_2 = \det\mat{ccc}{ r_{vv} & s_{vv} & t_{vv} \\ r_u & s_u & t_u \\ r_v & s_v & t_v}, \qquad
 I_3 = \det\mat{ccc}{ r_{uv} & s_{uv} & t_{uv} \\ r_u & s_u & t_u \\ r_v & s_v & t_v}.
 \label{I123}
 \end{align}
 The functions $I_1, I_2$ are the MA invariants, as demonstrated in \eqref{1-lambda}.
 These functions satisfy some non-trivial differential syzygies.  First, note that with respect to the basis $\{ \x, \x_u, \x_v, \normal, {\bf Z} \}$,
 \begin{align*}
 \x_{uu} &= \frac{\uvprod_u}{\uvprod} \x_u - \frac{2I_1}{\uvprod^2} \normal, \qquad
 \x_{uv} = - \frac{2I_3}{\uvprod^2} \normal + \uvprod {\bf Z}, \qquad
 \x_{vv} = \frac{\uvprod_v}{\uvprod} \x_v - \frac{2I_2}{\uvprod^2} \normal \\
 \normal_u &= -\frac{I_3}{\uvprod} \x_u - \frac{I_1}{\uvprod} \x_v +\frac{\uvprod_u}{\uvprod} \normal, \qquad
 \normal_v = -\frac{I_2}{\uvprod} \x_u - \frac{I_3}{\uvprod} \x_v +\frac{\uvprod_v}{\uvprod} \normal
 \end{align*}
 where we have used
 \begin{align*}
  &\ambient{\x_{uu}}{\x} = \ambient{\x_{vv}}{\x} = \ambient{\x_{uu}}{\x_u} = \ambient{\x_{vv}}{\x_v} = \ambient{\x_{uv}}{\x_u} = \ambient{\x_{uv}}{\x_v} = 0,\\
  &\ambient{\x_{uu}}{\x_v} = \uvprod_u, \qquad \ambient{\x_{vv}}{\x_u} = \uvprod_v, \qquad \ambient{\x_{uv}}{\x} = -\uvprod,   \qquad \ambient{\normal_u}{\normal} =-\frac{\uvprod \uvprod_u}{2}, \qquad \ambient{\normal_v}{\normal} =-\frac{\uvprod \uvprod_v}{2},\\
  &\ambient{\normal_u}{\x} = \ambient{\normal_v}{\x} = 0, \qquad \ambient{\normal_u}{\x_u} = -I_1, \qquad \ambient{\normal_u}{\x_v} = \ambient{\normal_v}{\x_u} = -I_3, \qquad \ambient{\normal_v}{\x_v} = -I_2.
 \end{align*}
 which are differential consequences of \eqref{null-relations}.
 We examine integrability conditions.  After some simplification, 
 \begin{align*}
 0 &= \x_{vvu} - \x_{uvv} 
  = \l( (\ln|\uvprod|)_{uv} + \frac{2(I_1 I_2 - I_3{}^2)}{\uvprod^3}  \r) \x_v + 2\l( \l(\frac{I_3}{\uvprod^2}\r)_v - \frac{1}{\uvprod} \l( \frac{I_2}{\uvprod}\r)_u\r) \normal \\
 0 &= \x_{uuv} - \x_{uvu} 
 = \l( (\ln|\uvprod|)_{uv} + \frac{2(I_1 I_2 - I_3{}^2)}{\uvprod^3} \r) \x_u + 2\l( \l( \frac{I_3}{\uvprod^2}\r)_u - \frac{1}{\uvprod} \l( \frac{I_1}{\uvprod}\r)_v  \r) \normal \\
 0 &= \normal_{uv} - \normal_{vu} 
 =  \l( \l( \frac{I_2}{\uvprod}\r)_u -\uvprod\l(\frac{I_3}{\uvprod^2}\r)_v \r)\x_u + \l( \uvprod \l( \frac{I_3}{\uvprod^2}\r)_u - \l(\frac{I_1}{\uvprod}\r)_v \r) \x_v 
 \end{align*}
 Thus, we obtain the differential syzygies
 \begin{align}
 (\ln|\uvprod|)_{uv} = \frac{2(I_3{}^2 - I_1 I_2)}{\uvprod^3}, \qquad
 \l( \frac{I_3}{\uvprod^2}\r)_u = \frac{1}{\uvprod} \l( \frac{I_1}{\uvprod}\r)_v, \qquad
 \l(\frac{I_3}{\uvprod^2}\r)_v = \frac{1}{\uvprod} \l( \frac{I_2}{\uvprod}\r)_u.
 \label{diff-syz}
 \end{align}

 \begin{rem} The identities \eqref{diff-syz} are extremely complicated, yet appear deceivingly simple.  The latter two identities in full yields two degree 5 (differential) polynomial in the 15 variables $r_u,r_v,..., t_{uu},t_{uv},t_{vv}$ (since the third order derivative terms cancel), each having 62 terms.  Using MAPLE's {\tt Groebner} package, we have verified that these differential polynomials vanish on the ideal generated by the relations \eqref{null-relations} and their differential consequences.
 \end{rem}

 \subsection{1-adaptation}
 \label{app:1-adaptation}
 
 We assume a regular parametrization, so $\x_u, \x_v$ never vanish.  Take our initial 1-adapted moving frame $\mf{}$ to be
 \begin{align}
 \displaystyle
 \mf{0} = \x, \quad
 \mf{1} = \x_u, \quad
 \mf{2} = \frac{1}{\uvprod} \x_v, \quad
 \mf{3} = \frac{2}{\uvprod} \iota\normal, \quad
 \mf{4} = {\bf Z}, 
 \label{1-adapted-lift}
 \end{align}
 where $\iota = \pm 1$ guarantees $\mf{}$ differs from $\mathcal{B}_H$ by an element of $O^+(2,3)$.  Assume the parameter domain $U$ is connected, so $\iota$ is constant.
  Using \eqref{1-adapted-lift}, \eqref{1-str-eqs}, $\theta_1 = \ambient{d\mf{0}}{\mf{2}}$, $\theta_2 = \ambient{d\mf{0}}{\mf{1}}$ and $\alpha_{12} = \frac{1}{2} \ambient{d\mf{1}}{\mf{3}}$, $\alpha_{21} = \frac{1}{2} \ambient{d\mf{2}}{\mf{3}}$ are
 \begin{align}
 \l\{ \begin{array}{l}
 \theta_1 = du\\
 \theta_2 = \uvprod dv
 \end{array}\r., \quad
 \l\{ \begin{array}{l}
 \alpha_{12} = \lambda_{11} \theta_1 + \lambda_{12} \theta_2 \\
 \alpha_{21} = \lambda_{12} \theta_1 + \lambda_{22} \theta_2
 \end{array}\r., \qbox{where} \lambda_{11} = \frac{\iota I_1}{\uvprod}, \qquad 
 \lambda_{22} = \frac{\iota I_2}{\uvprod^3}, \qquad
 \lambda_{12} = \frac{\iota I_3}{\uvprod^2}. \label{1-lambda}
 \end{align}
 Note that under null reparametrizations, we have
  \begin{align*}
 \begin{array}{c|ccccccccc}
 (\tilde{u},\tilde{v}) & \tilde\iota & \tilde\uvprod & \tilde\normal & \tilde{I}_1 & \tilde{I}_2 & \tilde{I}_3 & \tilde\lambda_{11} & \tilde\lambda_{22} & \tilde\lambda_{11} \tilde\lambda_{22} \\ \hline
  (f(u), g(v)) &  \sgn(f_u) \iota & \frac{1}{f_u g_v} \uvprod & \frac{1}{f_u g_v} \normal & \frac{I_1}{(f_u)^3 g_v} & \frac{I_2}{f_u (g_v)^3} &  \frac{I_3}{(f_u)^2 (g_v)^2} & \frac{\sgn(f_u)}{(f_u)^2} \lambda_{11} & \sgn(f_u) (f_u)^2 \lambda_{22} & \lambda_{11} \lambda_{22}\\
 (v,u) & \sgn(\uvprod) \iota & \uvprod & -\normal & -I_2 & -I_1 & -I_3 & -\sgn(\uvprod) \uvprod^2 \lambda_{22} & -\frac{\sgn(\uvprod)}{\uvprod^2} \lambda_{11} &\lambda_{11} \lambda_{22}
 \end{array}
 \end{align*}
 In particular, note that $\tilde\alpha_{12} = \frac{1}{|f_u|} \alpha_{12}$ under $(\tilde{u},\tilde{v}) = (f(u),g(v))$ and $\tilde\alpha_{12} = -\sgn(\uvprod) \uvprod \alpha_{21}$ under $(\tilde{u},\tilde{v}) = (v,u)$.  This is to be expected that the MC forms are not necessarily invariant under reparametrizations since our choice of 1-adapted frame in \eqref{1-adapted-lift} is not.  For example, if $(\tilde{u},\tilde{v}) = (f(u),g(v))$, then $(\mft{0},\mft{1},\mft{2},\mft{3},\mft{4}) = (\mf{0},\frac{1}{f_u} \mf{1}, f_u \mf{2},\sgn(f_u)\mf{3},\mf{4})$.

 \subsection{\MA/ invariants}
 \label{app:MA}

 Let us give another set of formulas for the MA invariants $I_1,I_2$, given in \eqref{I123}.
  Let $F(r,s,t) = 0$ be an implicit description of a hyperbolic surface $M$ endowed with null coordinates $(u,v)$.  Recall that $I_1 = \ambient{\normal}{\x_{uu}}$, $I_2 = \ambient{\normal}{\x_{vv}}$, where we can take $\normal = (0,-F_t, \frac{1}{2} F_s, -F_r, -rF_r - sF_s - tF_t)$.  (This may differ from \eqref{N} by an overall scaling, but this will not affect the classification result.)  Thus, $I_1$ and $I_2$ have the simple expressions
 \begin{align}
 I_1 = -(F_r r_{uu} + F_s s_{uu} + F_t t_{uu}), \qquad  I_2 = -(F_r r_{vv} + F_s s_{vv} + F_t t_{vv}).
 \label{I1I2-a}
 \end{align}
 evaluated on the surface.  There is another natural way to express these invariants.  Let 
 \[
 n_1 = \mat{c}{r_u\\ s_u\\ t_u}, \qquad  n_2 = \mat{c}{r_v\\ s_v\\ t_v}, \qquad \Hess(F) = \mat{ccc}{ F_{rr} & F_{rs} & F_{rt}\\ F_{sr} & F_{ss} & F_{st}\\F_{tr} & F_{ts} & F_{tt}}.
 \]
 Differentiating $F=0$, we obtain $F_r r_u + F_s s_u + F_t t_u = 0$, $F_r r_v + F_s s_v + F_t t_v = 0$.  
 Differentiating again, one finds that $ 0= dF((n_1)_u)) + \transpose{n_1} \Hess(F) n_1$ and
 $0= dF((n_2)_v)) + \transpose{n_2} \Hess(F) n_2$.
 Combining this with \eqref{I1I2-a}, we obtain
 \begin{align}
 I_1 = \transpose{n_1} \Hess(F) n_1, \qquad
 I_2 = \transpose{n_2} \Hess(F) n_2.
 \label{I1I2-b}
 \end{align}

 \subsection{2-adaptation}
 
 Let $s_3 =\lambda_{12} = \frac{\iota I_3}{\uvprod^2}$. The CTS is given by $\mf{c} = \mf{3} + 2s_3 \mf{0}$.   Hence, a 2-adapted lift is
 \begin{align}
 \mf{0} = \x, \qquad \mf{1} = \x_u, \qquad \mf{2} = \frac{1}{\uvprod} \x_v, \qquad \mf{3} = \mf{c} = \frac{2 \iota}{\uvprod}\l( \normal + \frac{I_3}{\uvprod} {\bf x} \r), \qquad \mf{4} = {\bf Z} - \frac{2I_3}{\uvprod^3} \normal - \frac{I_3{}^2}{\uvprod^4} \x.
 \label{2-lift}
 \end{align}
 For \eqref{2-lift}, we calculate $\theta_1 = du$, $\theta_2 = \uvprod dv$ as before, and  $\lambda_{12} = 0$, while $\lambda_{11},\lambda_{22}$ are as in \eqref{1-lambda}.  Moreover,
 \begin{align*}
 &\alpha_{11} + \alpha_{22} = 0, \qquad
 \alpha_{22} - \alpha_{11} = \frac{\uvprod_u}{\uvprod} du, \qquad \alpha_{12} = \lambda_{11} du, \qquad \alpha_{21} = \lambda_{22} \uvprod dv,\\
 & d\lambda_{11} + \lambda_{11} (3\alpha_{11} - \alpha_{22}) = \lambda_{111} \theta_1 + \lambda_{112} \theta_2, \qquad
     d\lambda_{22} + \lambda_{22} (3\alpha_{22} - \alpha_{11}) = \lambda_{221} \theta_1 + \lambda_{222} \theta_2,\\
 &\lambda_{111} =  \uvprod^2\l( \frac{\iota I_1}{\uvprod^3} \r)_u, \qquad \lambda_{112} = \frac{1}{\uvprod}\l(  \frac{\iota I_1}{\uvprod} \r)_v, \qquad \lambda_{221} 
 = \frac{1}{\uvprod^2} \l( \frac{\iota I_2}{\uvprod} \r)_u, \qquad \lambda_{222} = \frac{1}{\uvprod}\l( \frac{\iota I_2}{\uvprod^3} \r)_v,\\
 &\beta_1 = \frac{2I_1 I_3}{\uvprod^3} du + \frac{I_3{}^2}{\uvprod^3} dv, \qquad
 \beta_2 = \frac{I_3{}^2}{\uvprod^4} du + \frac{2I_2 I_3}{\uvprod^4} dv, \\
 & \beta_3 = d\l( \frac{\iota I_3}{\uvprod^2} \r) = \lambda_{112} \theta_1 + \lambda_{221} \theta_2  \qRa \lambda_{112} = \l( \frac{\iota I_3}{\uvprod^2}\r)_u, \qquad \lambda_{221} = \frac{1}{\uvprod} \l( \frac{\iota I_3}{\uvprod^2}\r)_v.
 \end{align*}
 Consequently, from the two expressions for $\lambda_{112},\lambda_{221}$ we recover again the latter two identities in \eqref{diff-syz}.

 \subsection{3-adaptation for singly-ruled surfaces}
 \label{app:3-SR}
 
 Assume $M$ is singly-ruled, $I_1 \neq 0$, $I_2 = 0$.  Hence, $v$ is the parameter along the null ruling and \eqref{diff-syz} becomes
  \begin{align}
 (\ln|\uvprod|)_{uv} = \frac{2 I_3{}^2}{\uvprod^3}, \qquad
 \l( \frac{I_3}{\uvprod^2}\r)_u = \frac{1}{\uvprod} \l( \frac{I_1}{\uvprod}\r)_v, \qquad
 \l(\frac{I_3}{\uvprod^2}\r)_v = 0.
 \label{SR-diff-syz}
 \end{align}
  Let $s_1 = \frac{\lambda_{111}}{3\lambda_{11}} = \frac{1}{3} \l(\ln\l| \frac{I_1}{\uvprod^3} \r|\r)_u$, $s_2 =  - \frac{\lambda_{112}}{\lambda_{11}} = -\frac{1}{\uvprod} \l(\ln\l|\frac{I_1}{\uvprod}\r|\r)_v$.  The PNC is $\mf{n_2} = \frac{1}{\uvprod}\x_v + s_2 \x$ and the SNC is $\mf{n_1'} = \x_u + s_1 \x$.  Hence, a 3-adapted lift for a singly-ruled hyperbolic surface is 
 \begin{align}
 & \mf{0} = \x, \qquad \mf{1} = \x_u + \frac{1}{3} \l(\ln\l| \frac{I_1}{\uvprod^3} \r|\r)_u \x , \qquad \mf{2} = \frac{1}{\uvprod} \l( \x_v -  \l( \ln \l|  \frac{I_1}{\uvprod} \r|\r)_v \x \r), \qquad \mf{3} = \frac{2 \iota}{\uvprod}\l( \normal + \frac{I_3}{\uvprod} {\bf x} \r), \\
 & \mf{4} = {\bf Z} - \frac{2I_3}{\uvprod^3} \normal -  \frac{1}{\uvprod}\l( \ln \l|  \frac{I_1}{\uvprod} \r|\r)_v \x_u + \frac{1}{3\uvprod} \l(\ln\l| \frac{I_1}{\uvprod^3} \r|\r)_u \x_v -  \l( \frac{1}{3\uvprod} \l(\ln\l| \frac{I_1}{\uvprod^3} \r|\r)_u \l( \ln \l|  \frac{I_1}{\uvprod} \r|\r)_v  + \frac{I_3{}^2}{\uvprod^4} \r)\x
 \label{3-lift-ruled}
 \end{align}
 For this framing, $\alpha_{12} = \lambda_{11} du$,  $\alpha_{21} = \beta_3 = 0$, and $\lambda_{111}=\lambda_{112} = 0$, so $3\alpha_{11} - \alpha_{22} = -d(\ln|\lambda_{11}|)$ is exact.  Moreover,
 \begin{align*}
 & \alpha_{11} 
  =  - \frac{1}{3} \l( \ln\frac{|I_1|}{|\uvprod|^{3/2}}\r)_u du, \qquad
  \alpha_{22} = \frac{\uvprod_u}{2\uvprod} du + \l( \ln\l| \frac{I_1}{\uvprod}\r|\r)_v dv, \qquad
  \beta_1 = \Lambda_{11} \theta_1 - \frac{ \Lambda_{12} }{3} \theta_2, \qquad 
  \beta_2 = \Lambda_{12} \theta_1,
 \end{align*}
 where $\beta_1 = -\ambient{d\mf{1}}{\mf{4}}$ and $\beta_2 = -\ambient{d\mf{2}}{\mf{4}}$, hence
 \begin{align}
 \Lambda_{11} = \frac{2I_1 I_3}{\uvprod^3} + \frac{\uvprod}{3} \l( \frac{1}{\uvprod} \l(\ln\l| \frac{I_1}{\uvprod^3} \r|\r)_u \r)_u - \frac{1}{9} \l( \l(\ln\l| \frac{I_1}{\uvprod^3} \r|\r)_u\r)^2 , \qquad 
  \Lambda_{12} = \frac{1}{\uvprod} \l( \ln \frac{|\uvprod|^{3/2}}{|I_1|}\r)_{uv} \label{Lambda11-12}
 \end{align}
 The absence of a $\theta_2$ term in $\beta_2$ is equivalent to  $\l(\frac{1}{\uvprod} \l( \frac{I_1}{\uvprod} \r)_v \r)_v = 0$, a differential consequence of \eqref{SR-diff-syz}.  From \eqref{Lambda121},
 \begin{align*}
  \Lambda_{111} 
  = \Lambda_{11} \l( \ln \frac{|\Lambda_{11}| \uvprod^2}{|I_1|^{4/3}} \r)_u, \qquad
  \Lambda_{121} 
  = \Lambda_{12} \l(\ln \frac{|\Lambda_{12}|\uvprod^2}{|I_1|^{2/3}} \r)_u = -\frac{3}{\uvprod} \Lambda_{11,v},
 \end{align*}
 and the absence of the $\theta_2$ term in the $d\Lambda_{12}$ equation \eqref{Lambda121} yields the syzygy
 \[
  (\ln|\Lambda_{12}|)_v + 2 \l( \ln\frac{|I_1|}{|\uvprod|}\r)_v = 0 \qRa \l( \frac{\Lambda_{12} I_1{}^2}{\uvprod^2} \r)_v = 0
 \]
%
 Under the null reparametrization $\tilde{u} = f(u)$, $\tilde{v} = g(v)$, we have
 \[
 \tilde\Lambda_{11} = \frac{ \Lambda_{11}}{(f_u)^2}, \quad 
 \tilde\Lambda_{12} = \Lambda_{12}, \quad 
 \tilde\Lambda_{121} = \frac{\Lambda_{121}}{f_u}, \quad
 \tilde\Lambda_{111} = \frac{\Lambda_{111}}{(f_u)^3}, \quad
 \tilde\zeta_1 = \zeta_1, \quad
 \tilde\zeta_2 = \sgn(f_u)\zeta_2, \quad
 \tilde\zeta = \sgn(f_u)\zeta, \quad
 \]
 where
  $\zeta_1 = \frac{\Lambda_{11}}{|\lambda_{11}||\Lambda_{12}|^{1/2}}$,
  $\zeta_2 = \frac{\delta_2\Lambda_{121}}{4|\lambda_{11}|^{1/2}|\Lambda_{12}|^{5/4}}$, and
  $\zeta = \frac{2\delta_1 \Lambda_{111}}{|\Lambda_{11}|^{3/2}}$.  Thus, $\delta_1 = \sgn(\Lambda_{11})$, $\delta_2 = \sgn(\Lambda_{12})$, $\zeta_1, (\zeta_2)^2, \zeta^2$ are invariants of unparametrized surfaces.
  
 \subsection{3-adaptation for generic surfaces}
 \label{sec:app-3-gen}
 
 Let us assume that $M$ is 2-generic.  Let $s_1 = - \frac{\lambda_{221}}{\lambda_{22}} = -\l(\ln\l|\frac{I_2}{\uvprod}\r|\r)_u$, $s_2 = - \frac{\lambda_{112}}{\lambda_{11}} = -\frac{1}{\uvprod} \l(\ln\l|\frac{I_1}{\uvprod}\r|\r)_v$.  The normalizing cones are given by $\mf{n_1} = \x_u + s_1 \x$ and $\mf{n_2} = \frac{1}{\uvprod}\x_v + s_2 \x $.  Hence, a 3-adapted lift is
 \begin{align}
 &\mf{0} = \x, \qquad
 \mf{1} = \x_u -\l(\ln\l|\frac{I_2}{\uvprod}\r|\r)_u \x, \qquad
 \mf{2} = \frac{1}{\uvprod} \l( \x_v -\l(\ln\l|\frac{I_1}{\uvprod}\r|\r)_v \x \r), \qquad
 \mf{3} = \frac{2 \iota}{\uvprod}\l( \normal + \frac{I_3}{\uvprod} {\bf x} \r), \\ 
&\mf{4} = {\bf Z} - \frac{2I_3}{\uvprod^3} \normal - \frac{1}{\uvprod} \l(\ln\l|\frac{I_1}{\uvprod}\r|\r)_v \x_u - \frac{1}{\uvprod}\l(\ln\l|\frac{I_2}{\uvprod}\r|\r)_u  \x_v + \l( \frac{1}{\uvprod} \l(\ln\l|\frac{I_1}{\uvprod}\r|\r)_v \l(\ln\l|\frac{I_2}{\uvprod}\r|\r)_u - \frac{I_3{}^2}{\uvprod^4} \r)\x.
 \label{3-lift-generic}
 \end{align} 
 For this framing, $\alpha_{12} = \lambda_{11} du$, $\alpha_{21} = \lambda_{22} \uvprod dv$ (c.f. \eqref{1-lambda}), $\beta_3=0$, and $\lambda_{112}=\lambda_{221} = 0$.  Moreover,
 \begin{align*}
 & \alpha_{11} = -\frac{\uvprod_u}{2\uvprod} du +  \l( \ln\l| \frac{I_2}{\uvprod}\r|\r)_u du, \qquad 
 	&&\beta_1 = \l( \frac{2I_1 I_3}{\uvprod^3} - \frac{\uvprod^2}{I_2} \l(\frac{I_3}{\uvprod^2}\r)_{uv}   \r) du + \l( \frac{I_3{}^2}{\uvprod^3}  + \l(\ln\l|\frac{I_2}{\uvprod}\r|\r)_{uv} \r) dv \\
 & \alpha_{22} = \frac{\uvprod_u}{2\uvprod} du + \l( \ln\l| \frac{I_1}{\uvprod}\r|\r)_v dv, \qquad
 	&&\beta_2 
  = \frac{1}{\uvprod} \l( \frac{I_3{}^2}{\uvprod^3} - \l(\ln\l|\frac{I_1}{\uvprod}\r|\r)_{uv} \r)du + \frac{1}{\uvprod}\l(  \frac{2I_2 I_3}{\uvprod^3} - \frac{\uvprod^2}{I_1}  \l( \frac{I_3}{\uvprod^2}\r)_{uv}   \r) dv  \\
   & \lambda_{111}  = \lambda_{11} \l( \ln\frac{|I_1| |I_2|^3}{|\uvprod|^6} \r)_u, \qquad &&\lambda_{222} = \frac{\lambda_{22}}{\uvprod} \l( \ln\frac{|I_1|^3 |I_2|}{|\uvprod|^6} \r)_v.
 \end{align*}
 For $\beta_2$, we used the differential identity \eqref{diff-syz}. From \eqref{generic-beta12}, and since $\theta_1 = du$, $\theta_2 = \uvprod dv$, we have
 \[
 b_{11} \lambda_{11} = \frac{2I_1 I_3}{\uvprod^3} - \frac{\uvprod^2}{I_2} \l(\frac{I_3}{\uvprod^2}\r)_{uv}, \quad 
 b_{11} \lambda_{22} = \frac{1}{\uvprod^2}\l(  \frac{2I_2 I_3}{\uvprod^3} - \frac{\uvprod^2}{I_1}  \l( \frac{I_3}{\uvprod^2}\r)_{uv}   \r) 
 \qRa  b_{11} = \iota\l( \frac{2I_3}{\uvprod^2} - \frac{\uvprod^3}{I_1 I_2} \l(\frac{I_3}{\uvprod^2}\r)_{uv} \r).
 \]
 From \eqref{kappa12} and \eqref{bij}, we have
 \begin{align}
  \kappa_1 
			= \frac{ \uvprod^3 }{|I_1| |I_2|} \l( \frac{|I_1|^{3/4} |I_2|^{1/4}}{|\uvprod|^{3/2}} \r)_v, \quad
  \kappa_2 
  			= \frac{|\uvprod|^3 }{|I_1||I_2|} \l( \frac{|I_1|^{1/4} |I_2|^{3/4}}{|\uvprod|^{3/2}} \r)_u, \quad
   \tau 
   	=  \frac{\sgn(I_1 \uvprod)\uvprod^2}{\sqrt{|I_1 I_2|}}\l( \frac{2I_3}{\uvprod^2} - \frac{\uvprod^3}{I_1 I_2} \l(\frac{I_3}{\uvprod^2}\r)_{uv}  \r) \label{kappa12tau}
 \end{align}
 Taking into account the formulas from Section \ref{app:1-adaptation}, we calculate the effect of null reparametrizations:
 \begin{align*}
 \begin{array}{c|ccc}
 (\tilde{u},\tilde{v}) & \tilde\kappa_1 & \tilde\kappa_2 & \tilde\tau \\ \hline
 (f(u), g(v)): & \sgn(f_u) \kappa_1 & \sgn(f_u) \kappa_2 &\tau  \\
 (v,u): & \sgn(\uvprod) \kappa_2 & \sgn(\uvprod) \kappa_1 & \epsilon \tau
 \end{array}
 \end{align*}

 \bibliographystyle{alpha}	
 \bibliography{dthe2000}		

 \end{document}